\documentclass[
        10pt,
        final
        ]{article}
\usepackage[utf8]{inputenc}
\usepackage[T1]{fontenc}
\usepackage{etoolbox}
\usepackage{amsmath}
\usepackage{amsfonts}
\usepackage{mathtools}
\usepackage{amssymb}
\usepackage{pifont}
\usepackage{xpatch}
\usepackage{amsthm}
\usepackage{stmaryrd}
\usepackage{enumitem}
\usepackage{mathrsfs}
\usepackage{ifdraft}

\usepackage[numeric, msc-links, nobysame]{amsrefs}

\usepackage{calc}
\usepackage{booktabs}
\usepackage{titletoc}
\usepackage{etoolbox}
\usepackage{caption}
\usepackage{varioref}
\usepackage{xifthen}
\usepackage{tikz}
\usepackage[skins, breakable, hooks]{tcolorbox}
\usetikzlibrary{calc, cd}

\usepackage[a4paper]{geometry}
\usepackage{leading}

\leading{13pt}
\allowdisplaybreaks

\makeatletter
\AtBeginDocument{%
  \let\[\@undefined
  \DeclareRobustCommand{\[}{\begin{equation}}%
  \let\]\@undefined
  \DeclareRobustCommand{\]}{\end{equation}}%
}
\makeatother
\mathtoolsset{showonlyrefs,showmanualtags}

\usepackage{hyperref}
\usepackage[atend]{bookmark}
\hypersetup{
  colorlinks, 
  linkcolor=-red!75!green!50, 
  citecolor=-red!75!green!50,
  urlcolor=green!30!black
}
\AtBeginDocument{%
  \clearpage
  \bookmark[level=1, named=FirstPage]{Cover page}%
}
\BookmarkAtEnd{%
 \bookmarksetup{startatroot}%
 \bookmark[named=LastPage, level=0]{Last page}%
}

\newtheorem{Theorem}{Theorem}

\theoremstyle{definition}
\newtheorem{Question}[Theorem]{Question}

\newtheorem{Proposition}{Proposition}[section]
\newtheorem{Lemma}[Proposition]{Lemma}
\newtheorem{Corollary}[Proposition]{Corollary}

\theoremstyle{definition}

\makeatletter
\def\th@remark{%
  \thm@headfont{\itshape}%
  \normalfont 
  \thm@preskip\topsep 
  \thm@postskip\thm@preskip
}
\makeatletter
\theoremstyle{remark}
\newtheorem{Remark}[Proposition]{Remark}

\AtBeginEnvironment{Example}{\pushQED{\qed}}
\AtEndEnvironment{Example}{\popQED}
\AtBeginEnvironment{Remark}{\pushQED{\qed}}
\AtEndEnvironment{Remark}{\popQED}

\tcbset{common/.style={
  breakable,
  enhanced jigsaw,
  before skip=10pt,
  colback=#1,
  grow to right by=0.25cm,
  grow to left by=0.25cm,
  if odd page={right=0.25cm,left=0.25cm}{right=0.25cm,left=0.25cm},
  toggle enlargement,
  boxsep=0pt,
  toprule=0.07mm,
  bottomrule=0.07mm,
  leftrule=0mm,
  rightrule=0mm,
  arc=0mm,
  }
}
\newlength{\normalparindent}
\AtBeginDocument{\setlength{\normalparindent}{\parindent}}
\tcbset{before upper={\setlength{\parindent}{\normalparindent}}} 

\colorlet{thm-color}{blue!5}
\definecolor{q-color}{HTML}{ffb703}
\colorlet{question-color}{q-color!10}
\colorlet{proof-color}{gray!8}
\colorlet{remark-color}{green!3}
\tcolorboxenvironment{proof}{common={proof-color}}
\tcolorboxenvironment{Remark}{common={remark-color}}
\tcolorboxenvironment{Question}{common={question-color}}
\foreach \t in {Proposition, Theorem, Lemma, Corollary, Conjecture}
  {
    \expandafter\tcolorboxenvironment\expandafter{\t}{
        common={thm-color},
    }
  }

\newlist{thmlist}{enumerate}{1}
\setlist[thmlist]{
        nolistsep,
        ref={\mdseries\textup{(\emph{\roman*})}},
        label={\mdseries\textup{(\emph{\roman*})}},
        listparindent=\parindent
        }
\newcommand\thmitem[1]{\textup{(\emph{\romannumeral #1})}}

\makeatletter
\titlecontents{section}
  [3.8em] 
  {}
  {\textsection\thecontentslabel. }
  {}
  {\titlerule*[0.25pc]{.}\contentspage}

\renewcommand\tableofcontents{%
    \begingroup
    \centering
    {\normalsize\bfseries\contentsname}\par
    \small
    \vspace*{0.5pc}%
    \begin{minipage}{.7\textwidth}
    \@starttoc{toc}%
    \end{minipage}\par
    \endgroup
    \vspace*{2pc}
    }
\makeatother

\setlist[itemize]{itemsep=1ex, listparindent=1em, parsep=0pt, label=\textbullet}

\newcommand\NN{\mathbb{N}}
\newcommand\ZZ{\mathbb{Z}}
\newcommand\CC{\mathbb{C}}
\newcommand\QQ{\mathbb{Q}}
\newcommand\kk{\Bbbk}

\let\AA\relax
\newcommand\AA{\mathbb{A}}
\newcommand\GG{\mathbb{G}}

\newcommand\B{\mathcal{B}}
\newcommand\N{\mathcal{N}}
\newcommand\Nil{\mathfrak{N}}
\renewcommand\d{\mathrm{d}}
\newcommand\W{\mathscr{W}}
\newcommand\X{\mathscr{X}}
\let\epsilon\varepsilon

\newcommand\id{\mathsf{id}}
\newcommand\place{\mathord-}

\DeclareMathOperator{\Res}{Res}
\let\hom\relax
\DeclareMathOperator{\hom}{Hom}
\DeclareMathOperator{\gr}{gr}
\DeclareMathOperator{\LND}{LND}
\DeclareMathOperator{\Der}{Der}
\DeclareMathOperator{\InnDer}{InnDer}
\DeclareMathOperator{\End}{End}
\DeclareMathOperator{\Ext}{Ext}
\DeclareMathOperator{\Aut}{Aut}
\DeclareMathOperator{\Exp}{Exp}
\DeclareMathOperator{\slope}{sl}
\DeclareMathOperator{\ad}{ad}
\DeclareMathOperator{\img}{img}
\let\H\relax
\DeclareMathOperator\H{H}
\DeclareMathOperator\HH{HH}

\DeclarePairedDelimiter\lin{\langle}{\rangle}
\DeclarePairedDelimiter\gen{\langle}{\rangle}

\newcommand\Alg{\mathrm{Alg}}
\newcommand\inter[1]{\llbracket#1\rrbracket}
\newcommand\normalorder[1]{{:}\mkern1mu#1\mkern1.6mu{:}}
\def\lact{\mathbin\triangleright}

\newcommand\claim[2][.8]{%
  \begin{minipage}[c]{#1\displaywidth}%
  \itshape
  #2
  \end{minipage}%
}

\makeatletter
\newcases{c-dcases*}{\qquad}{%
  \hfil$\m@th\displaystyle{##}$\hfil}{##\hfil}{\lbrace}{.}
\newcases{r-dcases*}{\qquad}{%
  \hfil$\m@th\displaystyle{##}$}{##\hfil}{\lbrace}{.}
\makeatother

\newcommand\newterm[2][]{\textbf{\itshape\color{blue!50!black}#2}}

\ifoptionfinal{}{
  \usepackage[draft]{showlabels}
  
}

\title{%
  On the derivations and automorphisms\\
  of the algebra $\kk\langle x,y\rangle/(yx-xy-x^N)$}
\author{Mariano Suárez-Álvarez\thanks{%
  The author is a research member of CONICET. The work leading to this paper
  was done at Universidad de Buenos Aires, in Argentina, and at the Guangdong
  Technion Israel Institute of Technology, in Shantou, Guangdong Province,
  China, and was partially supported by \textsc{PIP-CONICET 11220200101855CO}.
  }}
\ifoptionfinal
  {\date{February 19, 2024}}
  {\date{Started on March 2021; compiled \today}}

\begin{document}

\maketitle

\tableofcontents

\phantomsection
\addcontentsline{toc}{section}{Introduction}%
\section*{Introduction}

In this paper we fix a field~$\kk$ of characteristic zero and a
non-negative integer~$N$, and study the algebra~$A_N$ freely
generated by two letters~$x$ and~$y$ subject to the relation
  \[
  yx-xy = x^N.
  \]
with the objective of computing as explicitly as it is possible (to us!)
some of its invariants of homological nature.

For low values of $N$ the algebra $A_N$ is very well-known: when $N=0$ it is
the first Weyl algebra, which we can view as the algebra of regular
differential operators on the affine line; when $N=1$ it is the enveloping
algebra of the non-abelian Lie algebra of dimension~$2$; and when $N=2$ it
is the so-called Jordan plane of non-commutative geometry
\citelist{\cite{AS} \cite{SZ}}. On the other hand, the family of algebras
that we will study is contained in a larger one that has received a lot of
attention: if for an arbitrary polynomial $h\in\kk[x]$ we let $A_h$ be the
algebra freely generated by letters~$x$ and~$y$ subject to the relation
$yx-xy=h(x)$, then of course our algebra~$A_N$ is $A_{x^N}$. One way to
explain the interest of this larger family of algebras is by saying that it
consists, up to isomorphism, of all skew-polynomial extensions of~$\kk[x]$
apart from the $1$-parameter families of quantum planes and quantum Weyl algebras
\cite{AD}*{Proposition 3.2}. The algebras of the form~$A_h$ have been
studied in detail by G.~Benkart, S.~Lopez and M.~Ondrus in the series of
papers \citelist{\cite{BLO:1} \cite{BLO:2} \cite{BLO:3}}. Our
algebras~$A_N$ are, in many senses, the \emph{worst} of the lot.

The main motivation for this work was the problem of giving an explicit
description of the first Hochschild cohomology space $\HH^1(A_N)$ of~$A_N$,
which we view as the space of outer derivations of the algebra~$A_N$, that
is, the quotient $\Der(A_N)/\InnDer(A_N)$ of the Lie algebra $\Der(A_N)$ of
all derivations of~$A_N$ by its ideal of inner derivations. This cohomology
space and, in fact, the full Lie algebra $\Der(A_N)$ have been studied in
detail before --- J.~Dixmier \citelist{\cite{Dixmier:1} \cite{Dixmier:2}}
and R.~Sridharan \cite{Sridharan} for the Weyl algebra, E.N.~Shirikov
\citelist{\cite{Shirikov:1} \cite{Shirikov:2} \cite{Shirikov:3}} for the
case $N=2$, A.~Nowicki \cite{Nowicki} and most notably G.~Benkart, S.~Lopez
and M.~Ondrus \cite{BLO:3} for the general case of the algebras~$A_h$ and,
building upon that, S.~Lopez and A.~Solotar \cite{LSo} for the Lie module
structure of the cohomology of~$A_h$ over~$\HH^1(A_h)$ --- and we can say
that both $\Der(A_N)$ and $\HH^1(A_N)$ are very well understood. What we
were after, though, was a description of the elements of~$\HH^1(A_N)$ as
cohomology classes of actual, explicit derivations, because we needed to do
further calculations with them. In particular, while doing some
calculations regarding the characteristic morphism
$\HH^\bullet(A_N)\to\mathscr{Z}_{\gr}(D^b(A_N))$ for this algebra --- which
connects the Hochschild cohomology of the algebra with the graded center of
the derived category of the category of modules of the algebra, for example
as in \cite{Lowen} --- certain rational numbers insistently appeared and
required an explanation.

\bigskip

Let us describe the results we obtain. The algebra~$A_N$ can be endowed
with a grading with respect to which the generators~$x$ and~$y$
are in degrees~$1$ and~$N-1$, respectively, and this grading induces others
in many objects constructed from~$A_N$. For example, the space~$\Der(A_N)$
and, for each $p\geq0$, the Hochschild cohomology space $\HH^p(A_N)$ are
$\ZZ$-graded vector spaces. Our result about~$\HH^1(A_N)$ is the following:

\begin{Theorem}\label{thm:hh1}
Suppose that $N\geq2$, let $q$ be a variable, and let $(c_i(q))_{i\geq0}$
be the sequence of polynomials in~$\QQ[q]$ such that
  \[ \label{eq:cj}
  \sum_{j\geq0}c_j(q)\frac{t^j}{j!} = \frac{t}{(1-qt)^{-1/q}-1}.
  \]
\begin{thmlist}

\item If $l$ is a positive integer and $i$ and~$j$ are the integers such
that $l+1=i+j(N-1)$ and $1\leq i\leq N-1$, then there is a homogeneous
derivation $\partial_l:A_N\to A_N$ of degree~$l$ such that
  \begin{align}
  \partial_l(x) &= x^iy^j, \\
  \partial_l(y) &= 
        \begin{multlined}[t][.75\displaywidth]
        (N-i)x^{i-1}
        \frac{1}{j+1}
        \sum_{i=0}^{j}
        \binom{j+1}{i}
        c_i(N-1)
        x^{i(N-1)}y^{j+1-i} \\
        +\sum_{s+2+t=N}(s+1)x^{s+i}y^jx^t 
        \end{multlined}
        \label{eq:ply}
  \end{align}

\item If $l$ is an integer such that $-N+1\leq i\leq 0$, then there is a
unique homogeneous derivation $\partial_l:A_N\to A_N$ of degree~$l$ such that
  \[
  \partial_l(x) = 0, 
  \qquad 
  \partial_l(y) = x^{l+N-1}.
  \]
There is moreover a homogeneous derivation~$E:A_N\to A_N$ of degree~$0$
such that
  \[
  E(x) = x,
  \qquad
  E(y) = (N-1)y.
  \]

\item The graded vector space~$\HH^1(A_N)$ is locally finite and its Hilbert
series is 
  \[
  h_{\HH^1(A_N)}(t) = 1 + \frac{t^{-N+1}}{1-t}.
  \]
If $l$ is an non-zero integer such that $l\geq-N+1$, then the homogeneous
component~$\HH^1(A_N)_l$ of degree~$l$ in~$\HH^1(A_N)$ is freely spanned by
the cohomology class of the derivation $\partial_l$ described above. On the
other hand, the component $\HH^1(A_N)_0$ of degree~$0$ is freely spanned by
the cohomology classes of the derivations~$\partial_0$ and~$E$.

\end{thmlist}
\end{Theorem}

The sequence of polynomials $(c_j(q))_{q\geq0}$ that appears here --- of
which the first few are tabulated in Table~\vref{tbl:ciq} --- is a
$q$-variant of the sequence of Bernoulli numbers, to which it degenerates
as $q$ goes to~$0$: indeed, the limit of the right hand side of the
defining equality~\eqref{eq:cj} as~$q$ approaches~$0$ is $t/(e^t-1)$, the
exponential generating function of the Bernoulli numbers, and 
for all $j\in\NN_0$ the constant term $c_j(0)$ is exactly the $j$th
Bernoulli number. On the other hand, the leading coefficient of~$c_j(q)$ is
$(-1)^jj!G_j$, with $G_j$ denoting the $j$th Gregory coefficient, certainly of much
lesser fame.

It should be observed that, while the formula in~\eqref{eq:ply} somehow
indicates that the limiting case $q\leadsto0$ corresponds to letting the
integer~$N$ «converge» to~$1$ (which, of course, makes no sense), none of
the claims of the theorem holds at the limit $N=1$! Indeed, the theorem
excludes the cases in which $N<2$, and that is because they are rather
different. When $N=0$ we have, according to a calculation carried out
originally by Dixmier in~\cite{Dixmier:2}, that $\HH^1(A)=0$, so there is
nothing that needs being made explicit. When $N=1$, on the other hand, the space
$\HH^1(A)$ is one-dimensional and spanned by the cohomology class of a
derivation $d_0:A\to A$ such that $d_0(x)=0$ and $d_0(y)=1$, which is
homogeneous of degree~$0$.

The way in which we prove this theorem is rather indirect --- a direct
proof of the first part of the statement would probably be quite
unpleasant! We first compute $\HH^0(A)$ and~$\HH^2(A)$ and, in particular,
their Hilbert series, and using that and an argument involving Euler
characteristics we deduce the Hilbert series of~$\HH^1(A)$: this tells us
of what degrees there are non-inner homogeneous derivations, and how many.
We then show that such a derivation can be modified appropriately until it
satisfies the conditions described in the theorem. We do this in
Sections~\ref{sect:cohomology}, \ref{sect:phi} and~\ref{sect:explicit},
whose main results are the Propositions~\ref{prop:hh1-series},
\ref{prop:hh1-n1}, \ref{prop:hh1-low} and~\ref{prop:hh1-high} that we
subsumed in Theorem~\ref{thm:hh1} above.

\bigskip

Once we have explicit derivations whose cohomology classes
span~$\HH^1(A_N)$, we can do things with them. The general qualitative
structure of the Lie algebra~$\HH^1(A)$ was described by G.~Benkart,
S.~Lopez and M.~Ondrus in \cite{BLO:3}: its center is one-dimensional and a
complement to the derived subalgebra
$\HH^1(A_N)'\coloneqq[\HH^1(A_N),\HH^1(A_N)]$, this derived subalgebra has
a unique maximal nilpotent ideal~$\Nil$ of nilpotency index~$N$, and the
quotient $\HH^1(A_N)'/\Nil$ is isomorphic to the Lie algebra $\Der(\kk[x])$
of derivations of~$\kk[x]$ or, equivalently, of regular vector fields on
the affine line~$\AA^1_\kk$, which is often called the Witt algebra.
Starting from Theorem~\ref{thm:hh1} we can prove --- this is
Corollary~\ref{coro:mlmm} in the text below --- the following:

\begin{Theorem}
There is a sequence of derivations $(L_j)_{j\geq-N+1}$ of~$A_N$ whose
cohomology classes freely span the derived subalgebra~$\HH^1(A)'$ such that
whenever $l$ and~$m$ are integers with $l$,~$m\geq-N+1$ we have
  \[ 
  [L_l, L_m] \sim
    \begin{dcases*}
    0 & if $i+u>N$ or $l+m<-N+1$; \\
    \frac{l(v+1)-m(j+1)}{N-1}L_{l+m} & if $i+u\leq N$,
    \end{dcases*}
  \]
with $i$, $j$, $u$ and~$v$ the unique integers such that
$l+1=i+j(N-1)$, $m+1=u+v(N-1)$, $1\leq i, u\leq N-1$, and
$j$,~$v\geq-1$. 
\end{Theorem}

In fact, for each non-zero $j$ the derivation~$L_j$ here is just a scalar
multiple of the derivation~$\partial_j$ of Theorem~\ref{coro:mlmm}, $L_0$
is a scalar multiple of the derivation~$E$, and the center of~$\HH^1(A)$ is
the span of the class of~$\partial_0$. In particular, the degree of~$L_j$
is~$j$.

This shows that the derived subalgebra~$\HH^1(A)'$ is, more or less, an
infinitesimal deformation of order~$N$ of the Witt algebra~$\Der(\kk[x])$
--- which is now realized as the subalgebra spanned by the sequence
$(L_{(-N+1)j})_{j\geq-1}$. It would be very interesting to have \emph{a
priori} reasons for this.

\bigskip

The first Hochschild cohomology of an algebra plays ---~in principle, but
usually not in reality ---  the role of the Lie algebra of the group of
outer automorphism of the algebra. For our algebra this does not quite
work, and since the units of~$A_N$ are all scalar and therefore central its
usual outer automorphism group coincides with the plain automorphism group.
There is an alternative notion of inner-automorphism that is useful in Lie
theory, though, in which we call an automorphism inner if it is a
composition of exponentials of locally $\ad$-nilpotent elements, and in our
situation it does do something, as the following result shows.

\begin{Theorem}
Let $\kk[x]\bowtie\kk^\times$ be the group whose underlying set is the
cartesian product $\kk[x]\times\kk^\times$ and whose multiplication is such that
  \[
  (f,\lambda)\cdot(g,\mu) = (\mu^{N-1}f+g\cdot\lambda,\lambda\mu)
  \]
for all $f$,~$g\in\kk[x]$ and all~$\lambda$,~$\mu\in\kk^\times$. 
\begin{thmlist}

\item There is
an isomorphism of groups
  \[
  \Phi:\kk[x]\bowtie\kk^\times\to\Aut(A_N)
  \]
such that $\Phi(f,\lambda)(x) = \lambda x$ and
$\Phi(f,\lambda)(y) = \lambda^{N-1}y + f$
for all $(f,\lambda)$ in~$\kk[x]\bowtie\kk^\times$.

\item The set of locally $\ad$-nilpotent elements of~$A_N$ is $\kk[x]$, and
for each $f\in\kk[x]$ we have $\exp\ad(f)=\phi_{x^Nf,1}$. The subset
$\Exp(A_N)$ of exponentials of locally $\ad$-nilpotent elements is a normal
subgroup of the automorphism group~$\Aut(A_N)$, and the map~$\Phi$ above
induces an isomorphism
  \[
  \overline\Phi:
  \frac{\kk[x]}{(x^N)}\bowtie\kk^\times
  \to
  \frac{\Aut(A_N)}{\Exp(A_N)}.
  \]
In particular, the quotient $\Aut(A_N)/\Exp(A_N)$ has a natural structure
of a Lie group over~$\kk$ of dimension~$N+1$, solvable of class~$2$ and, in
fact, an extension of~$\kk^\times$ by~$\kk^N$.

\item For each $g\in\kk[x]$ there is a derivation~$d_g:A_N\to A_N$ such
that $d_g(x)=0$ and $d_g(y)=g$, and it is locally nilpotent. The map
  \[
  g\in\kk[x]\mapsto d_g\in\Der(A_N)
  \]
is injective, and its image is the set of locally nilpotent derivations
of~$A_N$, which happens to be an abelian Lie subalgebra of~$\Der(A_N)$. The
set of exponentials of locally nilpotent derivations of~$A_N$ 
is the normal subgroup
  \(
  \Aut_0(A_N) \coloneqq \{\phi_{g,1}:g\in\kk[x]\}
  \),
and it sits in an extension of groups
  \[
  \begin{tikzcd}
  0 \arrow[r]
    & \Aut_0(A) \arrow[hook, r]
    & \Aut(A) \arrow[r, "\det"]
    & \kk^\times \arrow[r]
    & 1
  \end{tikzcd}
  \]
in which $\det(\phi_{\lambda,f})=\lambda$ for all $(\lambda,f)\in
\kk[x]\bowtie\kk^\times$, and which is split by the morphism
$\lambda\in\kk^\times\mapsto\phi_{0,\lambda}\in\Aut(A)$.

\end{thmlist}
\end{Theorem}

We have collected in this statement the results of
Propositions~\ref{prop:aut}, \ref{prop:ad-nilpotent} and~\ref{prop:lnd},
Corollaries~\ref{coro:aut-quot} and~\ref{coro:aut-0}. This information is
useful, for example, when computing the action of~$\Aut(A_N)$ on derived
objects, like Hochschild cohomology or cyclic homology, on which
exponentials of locally $\ad$-nilpotent elements tend to act trivially. For
now, let us say that such an explicit description of the automorphism group
allows us to compute its center easily. It turns out to be significant in
several ways:

\begin{Theorem}
There is a locally nilpotent derivation $\partial_0:A_N\to A_N$ such that
$\partial_0(x)=0$ and $\partial_0(y)=x^{N-1}$ such that the map
  \[ \label{eq:sigma-t}
  t \in \kk \mapsto \sigma_t\coloneqq\exp t\partial_0 \in\Aut(A_N)
  \]
is an injective $1$-parameter subgroup of~$\Aut(A_N)$ whose image is the
center of~$\Aut(A_N)$. 
\begin{thmlist}

\item The infinitesimal generator~$\partial_0$ of this $1$-parameter
subgroup is not inner and its class in~$\HH^1(A)$ spans the center of this
Lie algebra.

\item The element~$x$ of~$A$ is normal, in that $xA=Ax$, and the
automorphism $\sigma_1:A_N\to A_N$ is the automorphism associated to it, so
that $ax=x\sigma_1(a)$ for all $a\in A_N$.

\item The algebra $A_N$ is twisted Calabi--Yau of dimension~$2$, and the
automorphism $\sigma_1$ is its modular (or Nakayama) automorphism, so that in particular
there is an automorphism of $A_N$-bimodules
  \[
  \H^2(A_N,A_N\otimes A_N)
        = \Ext_{A_N^e}(A_N,A_N\otimes A_N)
        \to {}_{\sigma_1}A_N.
  \]

\item The derivation~$\partial_0$ preserves the canonical «order» filtration
on~$A_N$, so it induces a derivation $\overline\partial_0:\gr A_N\to\gr
A_N$ on the corresponding associated graded algebra. If we endow $\gr A$
with its standard Poisson structure coming from the commutator of~$A_N$,
then $\overline\partial_0$ is the modular derivation of~$A_N$ in the sense
of A.~Weinstein \cite{Weinstein}, and the corresponding modular flow
  \[
  \sigma:t\in\kk\mapsto\exp t\overline\partial_0\in\Aut(\gr A)
  \]
is exactly the $1$-parameter group of automorphisms induced by the
flow~\eqref{eq:sigma-t} above.

\end{thmlist}
\end{Theorem}

This theorem combines the results of Lemma~\ref{lemma:sigma-1},
Corollaried~\ref{coro:center} and~\ref{coro:expp0},
Proposition~\ref{prop:CY}, and Remark~\ref{rem:modular}. All the objects
mentioned in this theorem are canonical. For example, as the units of~$A_N$
are central, the modular automorphism of~$A_N$ as a twisted Calabi--Yau
algebra is well-determined. As we wrote in the theorem,
the derivation~$\partial_0$ is not inner, but it is «logarithmically
inner», in that it coincides with the restriction to~$A_N$ of the
derivation
  \[
  a\in (A_N)_x \mapsto \frac{1}{x}[x,a] \in (A_N)_x
  \]
of the localization~$(A_N)_x$ of~$A_N$ at its normal element~$x$.

\bigskip

A natural thing to do at this point is to describe 
the finite subgroups of $\Aut(A_N)$, and that is easy since we know the
group very well. Our Proposition~\ref{prop:finite-subgroups} implies, among
other things, the following:

\begin{Theorem}
Every finite subgroup of~$\Aut(A)$ is cyclic, and conjugated to the
subgroup generated by~$\phi_{0,\lambda}$, with $\lambda$ a root of unity
in~$\kk$. 
\end{Theorem}

Of course, with this result at hand the obvious next thing to do, following
the classics, is to describe the invariant subalgebras corresponding to the
finite subgroups of~$\Aut(A_N)$. This appears to be fairly difficult, and
we do not consider this problem here. We instead take a less classical
direction and try to extend the result of the theorem and find all actions
of finite dimensional Hopf algebras on~$A_N$ --- that is, to put it in a
colorful language, to find all \emph{quantum finite groups} of
automorphisms of~$A_N$. Now, at that level of generality we do not know how
to approach the problem, so we restrict ourselves to looking for all
actions of generalized Taft Hopf algebras on~$A_N$. We know that all finite
groups of automorphisms are cyclic, and generalized Taft algebras can be
viewed as «quantum thickenings» of cyclic groups, so this is a reasonable
first step. What we find is the following result.

\begin{Theorem}\label{thm:taft}
Let $n$ and~$m$ be integers such that $1<m$ and $m\mid n$, and let
$\lambda\in\kk^\times$ be a primitive $m$th root of unity in~$\kk$. There
is no inner-faithful action of the generalized Taft algebra
$T_n(\lambda,m)$ on~$A$.
\end{Theorem}

We refer to Section~\ref{sect:taft} of the paper --- which ends with a proof of this
theorem, stated there as Proposition~\ref{prop:taft-actions} --- for the
precise description of what we mean by generalized Taft algebra, and for the
definition of inner-faithfulness, which is due to T.~Banica and J.~Bichon \cite{BB}.
Let us just say here that inner-faithfulness intends to be to actions of
Hopf algebras what faithfulness is to actions of groups.

The negative nature of this theorem is disappointing, but not
unexpected. For example, this «no quantum finite symmetries» phenomenon
occurs on Weyl algebras, algebras of differential operators, and more
generally certain algebraic quantizations, as shown by J.~Cuadra,
P.~Etingof and C.~Walton in~\citelist{\cite{CEW:1}\cite{CEW:2}} and by
those last two authors in~\cite{EW}. There are many other classes of
algebras, though, which exhibit non-trivial quantum symmetries, and the
study of this is an extremely interesting line of work. In this direction,
one should mention the work of S.~Montgomery and H.-J.~Schneider \cite{MS}
on finite dimensional algebras generated by one element, extended
later by Z.~Cline~\cite{Cline}; L.~Centrone and F.~Yasumura \cite{CY} on
finite dimensional algebras; Y.~Bahturin and S.~Montgomery \cite{BM} on
matrix algebras; R.~Kinser and C.~Walton \cite{KW} on path algebras;
J.~Gaddis, R.~Won and D.~Yee \cite{GWY} on quantum planes and quantum Weyl
algebras; Z.~Cline and J.~Gaddis \cite{CG} on quantum affine spaces,
quantum matrices, and quantized Weyl algebras; and J.~Gaddis and R.~Won
\cite{GW} in quantum generalized Weyl algebras.

The way we prove Theorem~\ref{thm:taft} is by studying the twisted
derivations of our algebra~$A_N$. This is a natural approach at this point,
for it consists of fixing an automorphism $\phi$ of the algebra and
computing $\HH^1(A_N,{}_\phi A_N)$, the first Hochschild cohomology space
of the algebra~$A_N$ with values in the $A_N$-bimodule that can be obtained
from~$A_N$ by twisted the left action by the automorphism~$\phi$, and we
can more or less use the same ideas that we used to compute the regular
Hochschild cohomology of~$A_N$. We do this in Section~\ref{sect:twist}. We
do not describe here the precise results obtained therein because of their
rather technical nature. We would like, nonetheless, to direct the reader's
attention to our Remark~\ref{rem:general-nonsense}, which explains the way
we construct twisted derivations in terms of the Gerstenhaber algebra structure
of Hochschild cohomology, and which is probably of more general applicability.
Finally, it should be remarked that the study of twisted Hochschild
cohomology should be useful to study actions of general finite dimensional
Hopf algebras, since we know from the classification theory of
A.~Andruskiewitsch and H.-J.~Schneider~\cite{ASch}  that finite dimensional
Hopf algebras are in many cases «built» from group-like and skew-primitive
elements, which give rise to automorphisms and twisted derivations of the
algebras upon which they act.

\bigskip

All the work described above involves the first Hochschild cohomology Lie
algebra~$\HH^1(A_N)$ and the automorphism group~$\Aut(A_N)$. Our algebra
also has a non-zero second Hochschild cohomology space $\HH^2(A_N)$, and
according to the classical deformation theory of M.~Gerstenhaber
\cite{Gerstenhaber:deformation}, its elements produce «by integration»
formal deformations of the algebra~$A_N$ just as the elements of~$\HH^1(A)$
produce (more or less\dots) automorphisms by exponentiation --- indeed, as
$\HH^3(A)=0$ this procedure of integration of elements of~$\HH^2(A)$ is
unobstructed, so always possible. We leave for future work the explicit
construction of these formal deformations, in the same vein as we
constructed explicitly derivations spanning $\HH^1(A)$. 

The natural thing to do, actually, is to study this problem for the whole
family of algebras~$A_h$ of~\citelist{\cite{BLO:1}\cite{BLO:2}\cite{BLO:3}}
to which the algebra~$A_N$ belongs at the same time: these algebras deform,
under appropriate hypotheses, to one another, and they all appear as
deformations of~$A_N$. The algebra~$A_N$ is the most singular element of
the class, so the geometry of the family at that point is the central point
to elucidate. This was the main motivation for the study of the
automorphism group of~$A_N$ described above, in fact. The description of
the Lie action of~$\HH^1(A_N)$ on~$\HH^2(A_N)$ that can be read off from
the work~\cite{LSo} of S.~Lopes and A.~Solotar should also be important to
do this.

\bigskip

Let us finish this introduction with a question. The form of the
derivations of~$A_N$ that we find in Section~\ref{sect:explicit} and, in
particular, the observations made at the end of Section~\ref{sect:phi}
about the elements~$\Phi_j$, hint at the idea that the algebra~$A_1$ that we
get by putting $N=1$ is not the correct «limit as $N$ converges to~$1$ of
the algebra~$A_N$». 

\begin{Question}
Is there an algebra~$\tilde A_1$ which better reflects the result of
«taking the limit of~$A_N$ as $N$ goes to~$1$» in that it has exterior
derivations involving the Bernoulli polynomials that we found in
equation~\eqref{eq:bern} at the end of Section~\ref{sect:phi} and the
Faulhaber formula?
\end{Question}

We could say that the family $(A_N)_{N\geq1}$ is, in a not very precise sense, flat with
respect to~$N$, but our calculation of~$\HH^1(A_N)$ shows that this space
is constant \emph{except} at~$N=1$. An algebra~$\tilde A_1$, obtained by
something like «dimensional regularization» but with respect to~$N$, would
fix this. Another example of the anomalous status of the algebra~$A_1$ in
the family $(A_N)_{N\geq1}$ is provided by Remark~\ref{rem:laguerre} below.

\section{Preliminaries}

We fix a field~$\kk$ of characteristic zero and an integer~$N\geq1$, and let $A$
be the algebra freely generated by two letters~$x$ and~$y$ subject to the
sole relation
  \[
  yx-xy = x^N.
  \]
If $N=0$, then this is a Weyl algebra, and if $N=1$ what we get is
isomorphic to the universal enveloping algebra of the non-abelian Lie
algebra of dimension~$2$. In general, we can view this algebra as the skew
polynomial algebra $\kk[x][y;\theta]$, with $\theta:\kk[x]\to\kk[x]$ the
derivation such that $\theta(x)=x^N$, and it follows from this that $A$ is
a noetherian domain, and that $\B=\{x^iy^j:i,j\geq0\}$ is a basis for~$A$.
There is a grading on~$A$ that makes~$x$ and~$y$ homogeneous of degrees~$1$
and~$N-1$, respectively. The homogeneous components of~$A$ are spanned by
the monomials in the basis~$\B$ that they contain, and they are all finite
dimensional exactly when $N\geq2$.

As $yx=xy+x^N=x(y+x^{N-1})$, it is easy to check that $Ax=xA$, that is,
that the element~$x$ is normal ---~in Lemma~\ref{lemma:normal} below we
will see that in fact it generates, together with the non-zero scalars, the
set of all the non-zero normal elements as a monoid. The set $S\coloneqq
\{x^i:i\geq0\}$ is then a left and right denominator set in~$A$ and we can
consider the localization $A_x\coloneqq S^{-1}A$ at~$S$. If $\kk[x^{\pm1}]$
is the algebra of Laurent polynomials and
$\tilde\theta:\kk[x^{\pm1}]\to\kk[x^{\pm1}]$ is the extension of the
derivation~$\theta$ to~$\kk[x^{\pm1}]$, then it is immediate that $A_x$ can
be viewed as the skew polynomial algebra~$\kk[x^{\pm1}][y;\tilde\theta]$.
The algebra~$A_x$ is graded, with $x^{-1}$ of degree~$-1$.

There is, on the other hand, an algebra filtration~$(F_k)_{k\geq-1}$ on~$A$
that has, for each $k\geq-1$, its $k$th layer $F_k$ spanned by the set of monomials
$\{x^iy^j:i\geq0,j\leq k\}$, and whose associated graded algebra $\gr A$ is
freely generated as a commutative algebra by the principal
symbols~$\overline x\coloneqq x+F_{-1}\in F_0/F_1$ and $\overline
y\coloneqq y+F_{0}\in F_1/F_0$  of~$x$ and~$y$, and these have degree~$0$
and~$1$, respectively. In particular ---~and we will use this fact all the
time without mentioning it~--- for all $i$,~$j\in\NN_0$ we have that
  \[
  y^ix^j \equiv x^iy^j \mod F_{i-1}.
  \]
As it is well-known, the non-commutativity of the algebra~$A$ gives rise to
a Poisson algebra structure on $\gr A$: this is the
biderivation~$\{\place,\place\}:\gr A\times\gr A\to\gr A$ uniquely
determined by the condition that
  \[
  \{\overline y,\overline x\} = \overline x^N.
  \]

Computing in~$A$ usually involves a significant amount of reordering
products, and the following lemma gives two important special cases of
this that are moreover related by a pleasing symmetry. If $k$ and~$x$ are
elements of a commutative ring~$\Lambda$ and $i\in\NN_0$, we define,
following Rafael Díaz and Eddy Pariguan \cite{DP}, the \newterm{Pochhammer
$k$-symbol} to be
  \[
  (x)_{k,i} \coloneqq x(x+k)(x+2k)\cdots(x+(i-1)k).
  \]
We will use a few times the equality
  \[ \label{eq:kpoch}
  \sum_{j\geq0}(a)_{k,j}\frac{t^j}{j!} = (1-kt)^{-a/k},
  \]
valid in the algebra~$\Lambda[[t]]$, provided, so that it actually makes
sense, that $\Lambda$ contains~$\QQ$. It can be proved by noticing that the
two sides are in the kernel of the differential operator
$(1-kt)\frac{\d}{\d t}-a$ and have the same constant term.

\begin{Lemma}\label{lemma:comm}
For each $i$,~$j\geq0$ we have that
  \begin{gather}
  y^jx^i = \sum_{t=0}^{j} (i)_{N-1,j-t} \binom{j}{t} x^{i+(j-t)(N-1)}y^t
     \label{eq:comm:1}
\shortintertext{and}
  x^iy^j = \sum_{t=0}^j(-i)_{N-1,j-t}\binom{j}{t}x^{(j-t)(N-1)}y^tx^i.
  \end{gather}
\end{Lemma}

These two equalities tell us how to move $x^i$ from one side of~$y^j$ to
the other.

\begin{proof}
Let us start by proving an identity involving Pochhammer $k$-symbols that
is a version of the well-known Zhu--Vandermonde identity. We
fix~$k\in\kk$ and for each $a$ in~$\kk$ consider, as above, the formal
series $f_a\coloneqq\sum_{i\geq0}(a)_{k,i}t^i/i!\in\kk[[t]]$. We have that
$f_af_b=f_{a+b}$ for all choices of~$a$ and~$b$ in~$\kk$: this can be checked
by showing that both sides of the equality are annihilated by the operator
$(1-kt)\frac{\d}{\d t}-(a+b)$ and have the same constant term. Writing this
equality in terms of coefficients we see that for all $j\geq0$ we have
that
  \[
  \sum_{i=0}^j \frac{(a)_{k,j-i}}{(j-i)!} \frac{(b)_{i,k}}{i!}
        = \frac{(a+b)_{k,j}}{j!}.
  \]
Let us now prove the lemma. An obvious induction proves that
  \[ \label{eq:comm:3}
  y^jx = \sum_{t=0}^{j} (1)_{N-1,j-t} \binom{j}{t}
                             x^{1+(j-t)(N-1)}y^t
  \]
for all $j\geq0$, and this is the special case of the
identity~\eqref{eq:comm:1} of the lemma in which $i=1$. Let us now
fix~$i\geq0$ and assume inductively that the identity~\eqref{eq:comm:1}
holds. Then
  \begin{align}
  y^jx^{i+1} 
    &= \sum_{t=0}^{j} (i)_{N-1,j-t} \binom{j}{t} x^{i+(j-t)(N-1)}y^tx 
\intertext{and using~\eqref{eq:comm:3} we see this is}
    &= \sum_{t=0}^{j} 
       \sum_{s=0}^{t} 
          (i)_{N-1,j-t} \binom{j}{t} 
          (1)_{N-1,t-s} \binom{t}{s}
          x^{i+1+(j-s)(N-1)}y^s
       \\
    &= \sum_{s=0}^{j} 
       \frac{j!}{s!}
       \left(
       \sum_{t=0}^{j-s} 
           \frac{(i)_{N-1,j-s-t}}{(j-s-t)!}
           \frac{(1)_{N-1,t}}{t!}
       \right)
          x^{i+1+(j-s)(N-1)}y^s
       \\
    &= 
       \sum_{s=0}^{j} 
          (i+1)_{N-1,j-s}
          \binom{j}{s}
          x^{i+1+(j-s)(N-1)}y^s.
  \end{align}
This proves~\eqref{eq:comm:1} for all values of~$j$. To prove the remaining
identity, we invert the one we already have: we fix $i$ and~$j$ in~$\NN_0$ and
compute that
  \begin{align}
  \MoveEqLeft[4]
  \sum_{t=0}^j(-i)_{N-1,j-t}\binom{j}{t}x^{(j-t)(N-1)}y^tx^i \\
    &= \sum_{t=0}^j\sum_{s=0}^{t}
                (-i)_{N-1,j-t} (i)_{N-1,t-s}
                \binom{j}{t} \binom{t}{s}
                x^{i+(j-s)(N-1)}
                y^s
        \\
    &= \sum_{s=0}^{j}
                \frac{j!}{s!}
                \left(
                \sum_{t=0}^{j-s}
                \frac{(-i)_{N-1,j-t-s}}{(j-t-s)!}
                \frac{(i)_{N-1,t}}{t!}
                \right)
                x^{i+(j-s)(N-1)}
                y^s
     = x^i y^j,
  \end{align}
which is what we want.
\end{proof}

\begin{Remark}\label{rem:laguerre}
When $N=1$ the identity~\eqref{eq:comm:1} tells us that
$y^jx^i=x^i(y+i)^j$, and we can write it, passing to the
localization~$A_x$, in the form
  \[ \label{eq:xyx:N1}
  x^{-i}y^jx^i = (y+i)^j.
  \]
Let us suppose now that $N\neq1$ and show how to view the
identity~\eqref{eq:comm:1} in a similar way also in this case. In the
localization~$A_x$ we have that
  \[
  x^{N-1}(x^{-N+1}y - y x^{-N+1})x^{N-1}
        = yx^{N-1} - x^{N-1}y
        = (N-1)x^{2N-2},
  \]
so that
  \[
  \left[ \frac{x^{-N+1}}{N-1}, y \right] = 1.
  \]
The subalgebra of~$A_x$ generated by $z\coloneqq x^{-N+1}/(N-1)$ and $y$ is
thus isomorphic to the first Weyl algebra. In this subalgebra we can
consider normal-order products ---~well-known in quantum field theory, for
example~--- and powers: we consider the symbol $\normalorder{zy}$ with the
property that $(\normalorder{zy})^i=z^iy^i$ for all $i\geq0$.

For each choice of $\alpha\in\CC$ and $j\in\NN_0$ we let
$L^{(\alpha)}_j\in\CC[t]$ be the $j$th generalized Laguerre polynomial with
parameter~$\alpha$. This is the unique solution of the Laguerre
differential equation
  \[
  tu'' + (\alpha+1-t)u' + ju = 0
  \]
that is a polynomial with leading coefficient $(-1)^j/j!$. It has
degree~$j$ and can be found to be given by the formula
  \[ \label{eq:laguerre}
  L^{(\alpha)}_j(t) = \sum_{i=0}^j(-1)^i\binom{j+\alpha}{j-i}
        \frac{t^i}{i!}.
  \]
When $\alpha>-1$, the sequence $(L^{(\alpha)}_j)_{j\geq0}$ is obtained from
$(t^j)_{j\geq1}$ by Gram--Schmidt orthogonalization over the
interval~$[0,+\infty)$ with respect to the weighting function~$t^\alpha
e^{-t}$ associated to the gamma distribution. We refer
to~\cite{AAR}*{\textsection 6.3} for more information on these
polynomials.

With this setup, we claim now that for all $i$,~$j\in\NN_0$ have that
  \[ \label{eq:xyx:NN}
  x^{-i}(\normalorder{zy})^j x^i
  = (-1)^j j! \, L^{(-i/(N-1)-j)}_j(\normalorder{zy}).
  \]
We view this as a reasonable generalization of~\eqref{eq:xyx:N1} for $N$
greater than~$1$ --- in any case, the fact that the commutation relations
of the algebra can be written in terms of hyper\-geometric functions is
interesting! Notice that the polynomial $(-1)^jj!L_j^{(-i/(N-1)-j)}$ is
monic, so the coefficient $(-1)^jj!$ appears here just as a
consequence of the standard normalization of Laguerre polynomials. We can
prove the equality~\eqref{eq:xyx:NN} by evaluating its right hand side using
the explicit formula~\eqref{eq:laguerre} for the Laguerre polynomials and
then using the first identity of Lemma~\ref{lemma:comm}. 
\end{Remark}

Another very useful observation that we will use many times is the
following one.

\begin{Lemma}\label{lemma:centralizer:x}
The centralizer of~$x$ in~$A$ is~$\kk[x]$.
\end{Lemma}

\begin{proof}
If $u$ is an element of~$A\setminus\kk[x]$ that commutes with~$x$, then
there exist $l\geq1$ and $a_0$,~\dots,~$a_l\in\kk[x]$ such that
$u=\sum_{i=0}^la_iy^i$ and $a_l\neq0$, and we have that
  \[
  0 = [u,x] = \sum_{i=0}^la_l[y^i,x]
        \equiv la_lx^Ny^{l-1} \mod F_{l-2},
  \]
which is absurd. As every element of~$\kk[x]$ obviously commutes with~$x$,
this proves the lemma.
\end{proof}

Going a bit further, we can describe the center of~$A$:

\begin{Proposition}\label{prop:center}
The center of~$A$ is~$\kk$.
\end{Proposition}

\begin{proof}
If $u$ is a central element of~$A$ then in particular $u$ commutes with~$x$
and we know from the lemma that this implies that $u\in\kk[x]$. As it also
commutes with~$y$, we also have that $0=[y,x]=x^Nu'$, so that $u$ is in
fact in~$\kk$.
\end{proof}

We will need further on the following related result, which has an entirely
similar proof.

\begin{Proposition}\label{prop:center:x}
The center of the localization~$A_x$ is~$\kk$.
\end{Proposition}

\begin{proof}
Let $z$ be a central element of~$A_x$. There is a positive integer
$l\in\NN$ such that $x^lz$ is in~$A$, and, as this product commutes
with~$x$, we know that it belongs to~$\kk[x]$. We thus see that $z$ is
in~$\kk[x^{\pm1}]$, and therefore that $0=[y,z]=x^nz'$: it is in fact a
scalar.
\end{proof}

\section{The automorphism group}

After the preliminaries of the previous section, our first objective is to
compute the group of automorphisms of the algebra~$A$, and to do that we
will start by finding the \newterm{normal} elements of~$A$, that is, those
elements~$u$ of~$A$ such that $uA=Au$.

\begin{Lemma}\label{lemma:normal}
The set of non-zero normal elements in~$A$ is $\N\coloneqq\{\lambda
x^i:\lambda\in\kk^\times,i\geq0\}$.
\end{Lemma}

\begin{proof}
The element~$x$ is normal in~$A$, since $yx=x(y+x^{N-1})$, and then it is
clear that all the elements of the set~$\N$ are normal in~$A$. Conversely,
let $u$ be a non-zero normal element in~$A$, and let $l\geq0$ and
$a_0$,~\dots,~$a_l\in\kk[x]$ be such that $u=\sum_{i=0}^la_iy^i$ and
$a_l\neq0$. As $u$ is normal, the right ideal~$uA$ is a bilateral ideal and
therefore $[u,x]\in uA$: there exists a $b\in A$ such that $[u,x]=ub$. As
$[u,x]=\sum_{i=0}^la_i[y^i,x]\in F_{l-1}$, this is only possible if in fact
$[u,x]=0$, so that $u\in\kk[x]$. We also have that $x^Nu'=[y,u]\in uA$ and,
since of course $x^Nu'$ is also in~$\kk[x]$, we have that $u$
divides~$x^Nu'$: this implies that $u$ is $\lambda x^i$ for
some~$\lambda\in\kk$ and some integer~$i\geq0$, so that $u$ belongs to the
set~$\N$.
\end{proof}

In view of this lemma, the element~$x$ generates the monoid of normal
elements of~$A$ up to non-zero scalars. Since $A$ is a domain, each
non-zero normal element~$u$ in~$A$ determines uniquely an automorphism of
algebras $\nu_u:A\to A$ such that $au=u\nu_u(a)$ for all $a\in A$. Let us
record for future use the description of the automorphism associated
to~$x$, which we will write~$\sigma_1$ for reasons that will be clear later.

\begin{Lemma}\label{lemma:sigma-1}
The automorphism~$\sigma_1:A\to A$ such that $ax=x\sigma_1(a)$ for all
$a\in A$ has $\sigma_1(x)=x$ and $\sigma_1(y)=y+x^{N-1}$. \qed
\end{Lemma}

This humble origin of~$\sigma_1$ as the automorphism associated to the
normal element~$x$ can actually be painted in a somewhat more impressive
way:

\begin{Proposition}\label{prop:CY}
The algebra~$A_N$ is twisted Calabi--Yau algebra of dimension~$2$, and its
modular automorphism is precisely the automorphism $\sigma_1:A_N\to A_N$ of
Lemma~\ref{lemma:sigma-1}.
\end{Proposition}

What we call here modular automorphism of a twisted Calabi--Yau algebra is
often called the Nakayama automorphism of the algebra.

\begin{proof}
That $A_N$ is twisted Calabi--Yau of dimension~$2$ is a consequence of the
fact that it is an Ore extension of~$\kk[x]$, which is a Calabi--Yau
algebra of dimension~$1$: this follows from a theorem of L.-Y.~Liu, S.~Wang
and Q.-S.~Wu proved in \cite{LWW}. That the modular automorphism of~$A_N$
is~$\sigma_1$ is also a consequence of their result.
\end{proof}

There is a right action of the multiplicative group~$\kk^\times$
on~$\kk[x]$ by algebra automorphisms such that $x\cdot\lambda = \lambda x$
for all $\lambda\in\kk^\times$. We denote $\kk[x]\bowtie\kk^\times$ the
group whose underlying set is the cartesian
product~$\kk[x]\times\kk^\times$ and whose product is such that
  \[
  (f,\lambda)\cdot(g,\mu) = (\mu^{N-1}f+g\cdot\lambda,\lambda\mu)
  \]
for all $f$,~$g\in\kk[x]$ and all~$\lambda$,~$\mu\in\kk^\times$. This
group\footnote{In general, if $G$ and $H$ are two groups, and $\lact$ and
$\triangleleft$ are a left action and a right action of $H$ on~$G$ by group
automorphisms such that $h\lact(g\triangleleft h')=(h\lact g)\triangleleft
h'$, then we can construct a group $G\bowtie H$ with underlying set
$G\times H$ and multiplication such that
$(g,h)\cdot(g',h')=((g\triangleleft h')(h\lact g'),hh')$. One can see that
this is isomorphic to a direct product of~$G$ and~$H$ with respect to an
action of $G$ on~$H$ constructed from~$\lact$ and~$\triangleleft$, but
surprisingly we have not found this direct construction in the literature.}
shows up in the following result.

\begin{Proposition}\label{prop:aut}
For each choice of $f\in\kk[x]$ and $\lambda\in\kk^\times$ there is unique
automorphism $\phi_{f,\lambda}:A\to A$ such that 
  \[
  \phi_{f,\lambda}(x) = \lambda x,
  \qquad
  \phi_{f,\lambda}(y) = \lambda^{N-1}y + f.
  \]
The function
$\Phi:(f,\lambda)\in\kk[x]\bowtie\kk^\times\mapsto\phi_{f,\lambda}\in\Aut(A)$
is an isomorphism of groups.
\end{Proposition}

Recall that we have the standing hypothesis that $N\geq1$: when $N=0$, the
algebra~$A$ is the first Weyl algebra and its automorphism group, which was
computed by Jacques Dixmier \cite{Dixmier:3} and Leonid Makar-Limanov
\cite{ML}, is both significantly larger and much less abelian. The case in
which $N=1$ was treated by Martha Smith in~\cite{Smith}, and much later the
case of a general Ore extension~$A_h$ by Jeffrey Bergen in~\cite{Bergen}. 
Our small class of extensions allows for a less complicated argument,
though.

\begin{proof}
An easy calculation shows that for each choice of $f\in\kk[x]$ and
$\lambda\in\kk^\times$ there is indeed an algebra endomorphism
$\phi_{f,\lambda}:A\to A$ mapping~$x$ and~$y$ to~$\lambda x$
and~$\lambda^{N-1}y+f$, respectively: we thus have a function
$(f,\lambda)\in\kk[x]\bowtie\kk^\times
\mapsto\phi_{f,\lambda}\in\End_{\Alg}(A)$. This is the map~$\Phi$ described
in the statement of the proposition. A direct calculation shows this map is
a morphism of monoids, so it actually takes values in~$\Aut(A)$. It is
obvious that it is injective, and we will now show that it is also
surjective.

Let $\phi:A\to A$ be an automorphism of algebras and, as before, let us
write $\N$ for the set of non-zero normal elements of~$A$, which is a
monoid with respect to multiplication. Of course, we have
that~$\phi(\N)=\N$ and that the restriction~$\phi|_\N:\N\to\N$ is an
automorphism of monoids which is the identity on the subset~$\kk^\times$
of~$\N$: it follows immediately from this that there is a
scalar~$\lambda\in\kk^\times$ such that $\phi(x)=\lambda x$. Moreover,
since $0=\phi([y,x]-x^N)=[\phi(y),\lambda x]-\lambda^Nx^N$ and
$[\lambda^{N-1}y,\lambda x]=\lambda^Nx^N$, we see that
$[\phi(y)-\lambda^{N-1}y,x]=0$ and, therefore, that
$\phi(y)-\lambda^{N-1}y\in\kk[x]$. This tells us that $\phi$ is in the
image of the map~$\Phi$, so that this mar is surjective.
\end{proof}

With the very explicit description of the group $\Aut(A)$ that this
proposition gives we can immediately compute its center:

\begin{Corollary}\label{coro:center}
For each $t\in\kk^\times$ there is a unique automorphism $\sigma_t:A\to A$
such that $\sigma_t(x)=x$ and $\sigma_t(y)=y+tx^{N-1}$, and it is central
in~$\Aut(A)$. The function
  \[
  t\in\kk\mapsto\sigma_t\in\Aut(A)
  \]
is an injective morphism of groups whose image is precisely the center
of~$\Aut(A)$. \qed
\end{Corollary}

The center of~$\Aut(A)$ is therefore a $1$-parameter subgroup which goes
through the automorphism~$\sigma_1$ of Lemma~\ref{lemma:sigma-1}. We will
find later an infinitesimal generator for this $1$-parameter subgroup as a
class in the first Hochschild cohomology of the algebra. 

\bigskip

Let $\Lambda$ be an algebra. If $u\in \Lambda$, then, as usual, the \newterm{inner
derivation} corresponding to~$u$ is   
  \[
  \ad(u):a\in \Lambda\mapsto[u,a]\in \Lambda,
  \]
and we say that $u$ is \newterm{locally $\ad$-nilpotent} if the
derivation~$\ad(u)$ is locally nilpotent. When that is the case we can
consider the exponential of~$\ad(u)$, namely the automorphism
  \[
  \exp\ad(u) : a\in\Lambda
               \mapsto 
               \sum_{i\geq0}\frac{\ad(u)^i(a)}{i!}\in \Lambda,
  \]
because for each $a$ in~$\Lambda$ the series appearing here is in fact a
finite sum. We put
  \[
  \Exp(\Lambda) \coloneqq \{ \exp\ad(u):\text{$u\in \Lambda$ is 
                        locally $\ad$-nilpotent} \}.
  \]
This is a conjugation-invariant subset of~$\Aut(\Lambda)$ but, in general,
not a subgroup. We refer to G.~Freudenburg's book \cite{Freudenburg} for
general information about locally nilpotent derivations, albeit in a
commutative setting.

\begin{Proposition}\label{prop:ad-nilpotent}
The set of $\ad$-locally nilpotent elements in~$A$ is $\kk[x]$ and for each
$f\in\kk[x]$ we have that
  \[
  \exp\ad(f)=\phi_{x^Nf',1}.
  \]
The set~$\Exp(A)$ coincides with $\{\phi_{f,1}:f\in x^N\kk[x]\}$ and is
a normal subgroup of~$\Aut(A)$.
\end{Proposition}

The locally $\ad$-nilpotent elements of the first Weyl algebra were
described by Dixmier in~\cite{Dixmier:3}*{Théorème 9.1}: when viewing the
algebra as that of differential operators on the line, they are the
elements that belong to the orbits of constant coefficient differential
operators under the action of the automorphism group of the algebra. Here
one could use the same description, but it is much less interesting: the
automorphism group of a Weyl algebra is much larger.

\begin{proof}
For each $a\in\kk[x]\setminus0$ let us write
$\nu(a)\coloneqq\max\{t\in\NN_0:a\in x^t\kk[x]\}$. On the other hand, if
$u$ is an element of~$A$ that is not in~$\kk[x]$, then there exist $l\geq1$
and $a_0$,~\dots,~$a_l\in\kk[x]$ such that $u=\sum_{i=0}^la_iy^i$ and
$a_l\neq0$, and we will call the rational number $\slope(u)\coloneqq
\nu(a_l)/l$ the \newterm{slope} of~$u$. We will start by showing that
  \[ \label{eq:sl}
  \claim{if $u$ and $v$ are elements of $A\setminus\kk[x]$ and
  $\slope(v)>\slope(u)$, then $[u,v]\in A\setminus\kk[x]$ and
  $\slope([u,v])>\slope(u)$.}
  \]
To do so, let $u$,~$v\in A\setminus\kk[x]$, so that there are
$l$,~$m\geq1$, $a_0$,~\dots,~$a_l$, $b_0$,~\dots,~$b_m\in\kk[x]$ such that
$u=\sum_{i=0}^la_iy^i$, $v=\sum_{j=0}^mb_jy^j$, $a_l\neq0$ and $b_m\neq0$,
and let us suppose that $\slope(v)>\slope(u)$, so that
  \[\label{eq:sl:1}
  \nu(b_m)/m>\nu(a_l)/l.
  \]
We have that
  \begin{align}
  [u,v] 
       &= \sum_{i=0}^l\sum_{j=0}^m
          \Bigl(
          a_i[y^i,b_j]y^j
          +
          b_j[a_i,y^j]y^i
          \Bigr)
       \equiv
          a_l[y^l,b_m]y^m
          +
          b_m[a_l,y^m]y^l
          \\
     &\equiv x^N(la_lb_m' - mb_ma_l') y^{l+m-1} \mod F_{l+m-2}.
        \label{eq:sl:2}
  \end{align}
There are $c$,~$d\in\kk[x]$ not divisible by~$x$ such that
$a_l=x^{\nu(a_l)}c$ and $b_m=x^{\nu(b_m)}d$, and 
  \[ \label{eq:sl:3}
  la_lb_m' - mb_ma_l' 
  = x^{\nu(a_l)+\nu(b_m)-1}
    \Bigl(
    \bigl(l\nu(b_m)-m\nu(a_l)\bigr)cd
    +
    x\bigl(lcd'-mdc'\bigr)
    \Bigr).
  \]
If the left hand side of this equality is~$0$, then ---~since $x$ does not
divide the product~$cd$~--- we have that $l\nu(b_m)-m\nu(a_l)=0$: this is
absurd, since we are assuming that the inequality~\eqref{eq:sl:1} holds.
Going back to~\eqref{eq:sl:2}, we see that $[u,v]\in F_{l+m-1}\setminus
F_{l+m-2}$ and, in particular, that $[u,v]\in A\setminus\kk[x]$, since
$l+m-1\geq1$. Moreover, in view of~\eqref{eq:sl:2} and~\eqref{eq:sl:3} and
the fact that $x$ does not divide the product~$cd$, we have
  \begin{align}
  \slope([u,v]) 
        &= \frac{\nu(a_l)+\nu(b_m)+N-1}{l+m-1}
\intertext{and this is easily seen to be}
        &> \frac{\nu(a_l)}{l} = \slope(u)
  \end{align}
using~\eqref{eq:sl:1} and the fact that $N\geq1$. This proves the
claim~\eqref{eq:sl} above.

Suppose now that $u$ is an element of~$A\setminus\kk[x]$, and let $m$ be an
integer such that $m>\slope(u)$. If we put $v_i\coloneqq\ad(u)^i(x^my)$ for
each non-negative integer~$i$, then an obvious induction
using~\eqref{eq:sl}, starting with the observation that $v_0\in
A\setminus\kk[x]$ and $\slope(v_0)=m>\slope(u)$, shows that $v_i\neq0$ for
all $i\in\NN_0$. This proves that the element~$u$ is not locally
$\ad$-nilpotent and, therefore, that the set of locally $\ad$-nilpotent
element of~$A$ is contained in~$\kk[x]$.

Conversely, if $f$ is an element of~$\kk[x]$ then we have that
$\ad(f)(F_i)\subseteq F_{i-1}$ for all $i\in\NN_0$, so that
$\ad(f)^{i+1}(F^i)=0$ for all $i\in\NN_0$: this shows that~$f$ is locally
$\ad$-nilpotent and, with that, the first claim of the proposition.
Moreover, a direct calculation show that $(\exp\ad(f))(x)=x$ and
$(\exp\ad(f))(y)=y+x^Nf'$, so that $\exp\ad(f)$ is the
automorphism~$\phi_{x^Nf',1}$. That $\Exp(A)$ is $\{\phi_{f,1}:f\in
x^N\kk[x]\}$ is now clear, that it is a subgroup of~$\Aut(A)$ follows
another simple calculation, and its normality is obvious.
\end{proof}

The subgroup of~$\Aut(A)$ generated by the exponentials of the locally
$\ad$-nilpotent elements of~$A$, which we have just shown to be
precisely~$\Exp(A)$, is cognate to the group of inner automorphisms of a
Lie algebra ---~as given by Roger Carter in~\cite{Carter}*{\textsection 3.2}, for
example~--- so it makes sense to view the quotient $\Aut(A)/\Exp(A)$ as a
Lie-theoretic version of the \newterm{outer automorphism group} of~$A$. We
will describe this quotient below. For comparison, the \emph{usual} inner
automorphism group of~$A$, in the sense of associative algebras, is
trivial, as the units of~$A$ are central, so that the usual outer
automorphism group of~$A$ is just $\Aut(A)$. On the other hand, the X-inner
automorphism group of~$A$ ---~that is, the group of automorphisms of~$A$
that are restrictions of inner automorphisms of the Martindale left
quotient ring of~$A$, considered originally by Vladislav K.\,Har\v{c}enko
in~\citelist{\cite{Harchenko:1}\cite{Harchenko:2}}~--- is isomorphic
to~$\ZZ$: this was computed by Jeffrey Bergen in~\cite{Bergen}*{Theorem~2.6}. 

\begin{Corollary}\label{coro:aut-quot}
Let $\xi$ be the class of $x$ in the quotient $Q\coloneqq\kk[x]/(x^N)$, let
$\kk^\times$ act on the right on this quotient by algebra automorphisms in
such a way that $\xi\cdot\lambda = \lambda\xi$ for all
$\lambda\in\kk^\times$, and let $Q\bowtie\kk^\times$ be the group that as a
set coincides with $Q\times\kk^\times$ and whose product is such that
  \[
  (p,\lambda)\cdot(q,\mu) = (\mu^{N-1}p+q\cdot\lambda,\lambda\mu)
  \]
whenever $(p,\lambda)$ and $(q,\mu)$ are two elements
of~$Q\times\kk^\times$. The function
  \[
  Q\bowtie\kk^\times \to \frac{\Aut(A)}{\Exp(A)} 
  \]
that maps each pair $(f+(x^N),\lambda)$ to the class of the
automorphism~$\phi_{f,\lambda}$ is an isomorphism of groups. The quotient
$\Aut(A)/\Exp(A)$ has therefore a natural structure of Lie group over~$\kk$
of dimension $N+1$, solvable of class~$2$ and, in fact, an extension of
$\kk^N$ by $\kk^\times$. \qed
\end{Corollary}

The most interesting part of this is, of course, that dividing by the group
of inner automorphisms allows us to go from the infinite-dimensional group
$\Aut(A)$ to a finite-dimensional one. As elements of~$\Exp(A)$ has a
tendency to act trivially on objects associated to~$A$, this is useful. The
proof of the corollary is immediate given the description of have
of~$\Aut(A)$ and of~$\Exp(A)$, and we omit it.

\bigskip

Let us determine now the locally nilpotent derivations of our algebra ---
we already know those which are inner.

\begin{Proposition}\mbox{}\label{prop:lnd}
\begin{thmlist}

\item If $g\in\kk[x]$, then there is a unique derivation $d_g:A\to A$ such
that $d_g(x)=0$ and $d_g(y)=g$, it is locally nilpotent, it is inner
exactly when $g\in x^N\kk[x]$, and for all $t\in\kk$ we have that
$\exp(td_g) = \phi_{tg,1}$.

\item If $d:A\to A$ is a locally nilpotent, there is exactly one
polynomial~$g\in\kk[x]$ such that $d=d_g$.

\end{thmlist}
\end{Proposition}

\begin{proof}
Let $g$ be an element of~$\kk[x]$. That there is
indeed a derivation~$d_g:A\to A$ such that $d_g(x)=0$ and $d_g(y)=g$
follows by a trivial calculation using the presentation of~$A$. Since it is
a derivation, for all $i$,~$j\in\NN_0$ we have that
  \[
  d_g(x^iy^j) 
        = \sum_{s+1+t=j}x^iy^sgy^t
        \equiv jgx^{i}y^{j-1} \mod F_{j-2}.
  \]
This implies that $d_g(F_j)\subseteq F_{j-1}$ for all $j\in\NN_0$ and, in
particular, that $d_g$ is locally nilpotent. If $t\in\kk$, then
$\exp(td_g)(x)=x$ because $d_g(x)=0$ and $\exp (td_g)(y)=y+td_g(y)=y+tg$
because $d_g^2(y)=0$, and this tells us that
$\exp(t\partial_0)=\phi_{tg,1}$.

If there is an element~$a\in A$ such that $d_g=\ad(a)$, then
$0=d_g(x)=[a,x]$, so that $a\in\kk[x]$ because of
Lemma~\ref{lemma:centralizer:x}, and therefore $g=d_g(y)=[a,y]\in x^NA$.
Conversely, if $g=x^Nh$ for some $h\in\kk[x]$ and we chose $k\in\kk[x]$
such that $k'=-h$, then $d_g=\ad(k)$, so that $d_g$ is inner. With
this we have proved all the claims in part~\thmitem{1} of the proposition.

Next, let $d:A\to A$ be a locally nilpotent derivation of~$A$, so that, in
particular, there are two non-negative integers~$l$ and~$m$ such that
$d^{l+1}(x)=0$ and $d^{m+1}(y)=0$. As~$d$~is locally nilpotent, for each
$t\in\kk$ we can consider the automorphism $\exp(td):A\to A$. In view of
Proposition~\ref{prop:aut}, for each $t\in\kk$ there exist a non-zero
scalar~$\lambda_t\in\kk^\times$ and a polynomial~$f_t\in\kk[x]$ such that
$\exp(td)=\phi_{f_t,\lambda_t}$ and therefore
  \[ \label{eq:dxy}
  \sum_{i=0}^lt^i\frac{d^i(x)}{i!} = \lambda_t x, 
  \qquad
  \sum_{i=0}^mt^i\frac{d^i(y)}{i!} = \lambda_t^{N-1}y + f_t. 
  \]
Let $t_0$,~\dots,~$t_l$ be $l$ pairwise different elements of~$\kk$. The
Vandermonde matrix built out of those $l+1$ scalars is invertible: it
follows that from the $l+1$ equalities that we get from the first
one in~\eqref{eq:dxy} by replacing~$t$ by each of~$t_0$,~\dots,~$t_l$ we
can solve for~$d(x)$ and find that there is a scalar~$\alpha$ such that
  \[ \label{eq:dxx}
  d(x)=\alpha x.
  \]
Proceeding similarly with the second equation in~\eqref{eq:dxy} we see that
there are a scalar~$\beta\in\kk$ and a polynomial~$g\in\kk[x]$ such that 
  \[ \label{eq:dyy}
  d(y)=\beta y+g.
  \]

Of course, we must have that $d(yx-xy-x^N)=0$, and writing this out we find
that $\beta=(N-1)\alpha$. On the other hand, from~\eqref{eq:dxx}
and~\eqref{eq:dyy} we can see immediately that there is a sequence of
polynomial~$(g_n)_{n\geq0}$ in~$\kk[x]$ such that
  \[
  d^n(y) = (N-1)^n\alpha^n y + g_n
  \]
for all $n\in\NN_0$. As $d^m(y)=0$, this implies that $\alpha=0$. As $g$ is
uniquely determined by the derivation~$d$, this proves the
claim~\thmitem{2} of the proposition.
\end{proof}

This proposition allows us to give a particularly important example of a
locally nilpotent derivation:

\begin{Corollary}\label{coro:expp0}
There is a derivation $\partial_0:A\to A$ such that $\partial_0(x)=0$ and
$\partial_0(y)=x^{N-1}$, it is locally nilpotent, not inner, and
  \[ \label{eq:expp0}
  \exp(t\partial_0) = \sigma_t.
  \]
for all $t$. \qed
\end{Corollary}

Here $\sigma_t$ is the central automorphism of~$A$ that we described in
Corollary~\ref{coro:center}, so the derivation~$\partial_0$ is an
infinitesimal generator for the $1$-parameter subgroup of~$\Aut(A)$ that is
its center. We will show later that the class of the
derivation~$\partial_0$ generates the center of~$\HH^1(A)$. The point of
the following remark is that this derivation is also of Poisson-theoretic
interest.

\begin{Remark}\label{rem:modular}
In proving this proposition we have noted that $\partial_0(F_j)\subseteq
F_{j-1}$ for all $j\in\NN_0$, and this implies immediately that the
derivation $\partial_0$ induces a derivation $\overline\partial_0:\gr
A\to\gr A$ on the associated graded algebra~$\gr A$ of~$A$ such that
  \[
  \overline\partial_0(\overline x)=0,
  \qquad
  \overline\partial_0(\overline y)=\overline x^{N-1}.
  \]
This is a Poisson derivation of~$\gr F$ and, in fact, it is the
\newterm{modular derivation} of that Poisson algebra, in the sense of Alan
Weinstein in~\cite{Weinstein}, that is, the map
  \[ \label{eq:md}
  f\in\gr A\mapsto \div H_f\in\gr A.
  \]
As a consequence of this, the morphism in Corollary~\ref{coro:center} that,
according to Corollary~\ref{coro:expp0}, arises by exponentiation
from~$\partial_0$, is precisely the \newterm{modular flow} of the Poisson
algebra $\gr A$. In~\eqref{eq:md} we have written
$H_f\coloneqq\{f,\place\}$ for the Hamiltonian derivation corresponding to
an element~$f$ of~$\gr A$, and $\div$ for the divergence operator, so that 
  \[
  X\lact\d\overline x\wedge\d\overline y 
        = \div L\cdot\d\overline x\wedge\d\overline y
  \]
for each $X\in\Der(\gr A)$, with $\lact$ the action of vector fields on
differential forms.
\end{Remark}

\begin{Remark}
The derivation~$\partial_0$ is not inner, but it is a «logarithmic
derivation» in that
  \[
  \partial_0(a) = \frac{1}{x}[x,a]
  \]
for all $a\in A$. Here the right hand side of the equality is, in
principle, an element of the localization~$A_x$ of~$A$ at~$x$, but it turns
out to be in~$A$.
\end{Remark}

Knowing the locally nilpotent derivations we obtain another nice subgroup
of~$\Aut(A)$.

\begin{Corollary}\mbox{}\label{coro:aut-0}
\begin{thmlist}

\item The set $\LND(A)$ of all locally nilpotent derivations of~$A$ is an
abelian subalgebra of the Lie algebra~$\Der(A)$ of all derivations of~$A$.

\item The set of the exponentials of the elements of~$\LND(A)$ is the
normal abelian subgroup 
  \[
  \Aut_0(A) \coloneqq \{\phi_{g,1}:g\in\kk[x]\}
  \]
of~$\Aut(A)$. 

\item The function $\det:\Aut(A)\to\kk^\times$ such that
$\det(\phi_{f,\lambda})=\lambda$ whenever $f\in\kk[x]$ and
$\lambda\in\kk^\times$ is a morphism of groups. The sequence
  \[
  \begin{tikzcd}
  0 \arrow[r]
    & \Aut_0(A) \arrow[hook, r]
    & \Aut(A) \arrow[r, "\det"]
    & \kk^\times \arrow[r]
    & 1
  \end{tikzcd}
  \]
is an extension of groups that is split by the morphism
$\lambda\in\kk^\times\mapsto\phi_{0,\lambda}\in\Aut(A)$. \qed

\end{thmlist}
\end{Corollary}

All this follows immediately from Proposition~\ref{prop:lnd}. This
corollary describes the algebraic actions of the additive group $\GG_a$ on
the algebra~$A$. Indeed, such a thing is the same thing as a right
algebra-comodule structure on~$A$ over the coordinate Hopf algebra~$\kk[t]$
of~$\GG_a$, which is a coassociative morphism of algebras $\phi:A\to
A\otimes\kk[t]$. In fact, a linear map $A\to A\otimes\kk[t]$ is determined
by a sequence $(\phi_i)_{i\geq0}$ of linear maps~$A\to A$ such that for
each $a\in A$ the sequence $(\phi_i(a))_{i\geq0}$ has almost all its
components equal to zero and $\phi(a)=\sum_{i=0}^\infty\phi_i(a)$, and it
is a algebra-comodule structure on~$A$ over~$\kk[x]$ exactly when
$\phi_1:A\to A$ is a locally nilpotent derivation and $\phi_i=\phi_1^i/i!$
for all $i\in\NN_0$.

It should be remarked that the space of locally nilpotent derivations of an
algebra is most often not a subalgebra of the Lie algebra of
derivations of the algebra.

\bigskip

On the end of the spectrum opposite to where the $\GG_a$-actions lie are
the actions of finite groups. With the description of the automorphism
group of our algebra that we have we can easily find these. Later,
in Section~\ref{sect:taft}, we will consider more generally coactions of
some finite-dimensional Hopf algebras --- and that will require considerably
more work.

\begin{Proposition}\label{prop:finite-subgroups}\mbox{}
\begin{thmlist}

\item If $\phi$ is an element of~$\Aut(A)$ of finite order~$m$, then exists
a unique $\lambda\in\kk^\times$ of order exactly~$m$ such that $\phi$ is
conjugated in~$\Aut(A)$ to~$\phi_{0,\lambda}$.

\item If $\lambda$ is an element of order~$m$ in~$\kk^\times$, then the
automorphism $\phi_{0,\lambda}$ has order~$m$, and if 
$G\coloneqq\gen{\phi_{0,\lambda}}$ is the subgroup of~$\Aut(A)$ generated
by it, then the subalgebra~$A^G$ of invariants
under~$G$ is the $m$th Veronese subalgebra
$A^{(m)}=\bigoplus_{i\geq0}A_{mi}$.

\item Every finite subgroup of~$\Aut(A)$ is cyclic, and conjugated to the
subgroup generated by~$\phi_{0,\lambda}$, with $\lambda$ a root of unity
in~$\kk$. 

\end{thmlist}
\end{Proposition}

The first two parts of the proposition imply that for all $m\in\NN$ the
number of conjugacy classes of~$\Aut(A)$ of elements of order~$m$ coincides
with the number of elements of order~$m$ in~$\kk^\times$. On the other
hand, the third part tells us that set of conjugacy classes of finite
subgroups of~$\Aut(A)$ are in bijection with the set of integers~$m$ such
that there is a primitive $m$th root of unity in~$\kk$, the bijection being
given by taking the order.

\begin{proof}
Let $(f,\lambda)\in\kk[x]\bowtie\kk^\times$, let $k\in\NN$ and let us
suppose that the automorphism~$\phi_{f,\lambda}$ has order~$k$. Since
  \[
  \phi_{f,\lambda}^k(x) = \lambda^kx,
  \qquad
  \phi_{f,\lambda}^k(y) = \lambda^{k(N-1)}y
        + \sum_{i=0}^{k-1}\lambda^{i(N-1)}f(\lambda^{k-1-i}x),
  \]
we see that $\lambda^k=1$, so that $\lambda$ has finite order
in~$\kk^\times$, and that 
  \[ \label{eq:fxk}
  \sum_{i=0}^{k-1}\lambda^{i(N-1)}f(\lambda^{k-1-i}x) = 0.
  \]
Let $m$ be the order of~$\lambda$ in~$\kk^\times$, so that $m$ divides~$k$.
If $h\in\kk[x]$, then
  \begin{align}
  (\phi_{h,1}\circ\phi_{f,\lambda}\circ\phi_{h,1}^{-1})(x)
    &= \lambda x,
    \\
  (\phi_{h,1}\circ\phi_{f,\lambda}\circ\phi_{h,1}^{-1})(y)
    &= \lambda^{N-1}y+f(x)+\lambda^{N-1}h(x)-h(\lambda x)
  \end{align}
This tells us that replacing~$f$ by
$f+\lambda^{N-1}h(x)-h(\lambda x)$ does not change the conjugacy class
of~$\phi_{f,\lambda}$ in~$\Aut(A)$ and,
as the subspace 
  \(
  \gen[\Big]{\lambda^{N-1}h(x)-h(\lambda x):h\in\kk[x]}
  \)
of~$\kk[x]$ is spanned by the monomials $x^d$ with $d\not\equiv N-1\mod m$,
that up to conjugacy we can suppose that the polynomial~$f$ is a linear
combination of monomials~$x^d$ with $d\equiv N-1\mod m$. It is easy to
check now that when that is the case the left hand side of the
equality~\eqref{eq:fxk} is equal to~$kf(\lambda^{k-1}x)$, and therefore
that equality implies that $f=0$, since our ground field has
characteristic zero. This shows that every element of~$\Aut(A)$ of finite
order is conjugated to one of the form~$\phi_{0,\lambda}$ with $\lambda$ an
element of finite order in~$\kk^\times$, and that of course the order
of~$\phi_{0,\lambda}$ is equal to the order of~$\lambda$. To complete the
proof of the first part of the proposition we need only notice that if
$\lambda$ and~$\mu$ are two elements of~$\kk^\times$ the automorphisms
$\phi_{0,\lambda}$ and~$\phi_{0,\mu}$ are conjugated in~$\Aut(A)$ exactly
when $\lambda=\mu$: this is an immediate consequence of the fact that the
map 
  \[ \label{eq:pi}
  \pi:\phi_{f,\lambda}\in\Aut(A)\mapsto\lambda\in\kk^\times
  \]
is a morphism of groups.

If $m\in\NN$ and $\lambda$ is an element of order~$m$ in~$\kk^\times$, then
$\phi_{0,\lambda}(x^iy^j)=\lambda^{i+j(N-1)}x^iy^j$. We see immediately
from this that the subalgebra fixed by~$\phi_{0,\lambda}$, and by the
cyclic group it generates, is spanned by the monomials~$x^iy^j$ with $m\mid
i+j(N-1)$, that is, whose degree is divisible by~$m$. The claim~\thmitem{2}
of the proposition follows at once from this.

Suppose now that $G$ is a finite subgroup of $\Aut(A)$. If  $f$ and $g$ are
two elements of~$\kk[x]$ and $\lambda$ one of~$\kk^\times$ such that
$\phi_{f,\lambda}$ and~$\phi_{g,\lambda}$ are both in~$G$, then the
composition $\phi_{f,\lambda}\circ\phi_{g,\lambda}$ has finite order and
maps~$x$ and~$y$ to $x$ and to~$y+\lambda^{-N+1}(f-g)$, respectively: it
follows immediately from this that $f=g$. This means that the map~$\pi$
from~\eqref{eq:pi} is injective when restricted to~$G$ and therefore $G$ is
isomorphic to a finite subgroup of~$\kk^\times$ and, in particular, cyclic.
If $\phi$ is a generator of~$G$ and $m$ is the order of~$G$, we have seen
above that there is a primitive $m$th root of unity~$\lambda$ in~$\kk$ such
that $\phi$ is conjugated to~$\phi_{0,\lambda}$ and, therefore, $G$ is
conjugated to~$\gen{\phi_{0,\lambda}}$.
\end{proof}

Proposition~\ref{prop:finite-subgroups} gives us a classification of the
finite subgroups of~$\Aut(A)$ and their corresponding invariant
subalgebras, and it is natural to consider the classical problem of classifying
these up to isomorphism. This seems to be a rather complicated problem, as
even giving presentations for them is not easy. We will content ourselves
with describing two «extreme» instances of this problem.

Let $\lambda$ be a root of unity in~$\kk$, let $m$ be its order, and let us
suppose that $m>1$. We let $G$ be the subgroup generated
by~$\phi_{0,\lambda}$ in~$\Aut(A)$, which has order~$m$, so that the
invariant subalgebra~$A^G$ is the $m$th Veronese
subalgebra~$A^{(m)}$ of~$A$. 
\begin{thmlist}

\item\label{it:ag:1} Suppose first that $m$ divides $N-1$, and let $k$ be
the integer such that $N-1=km$. The invariant subalgebra~$A^G$ is easily
seen to be generated by~$x^n$ and $y$ and, since
  \[
  yx^n-x^ny=nx^{n+N-1}=n(x^n)^k,
  \]
isomorphic to $A_k=\kk\lin{x,y}/(yx-xy-x^k)$. We will see below, in
Corollary~\ref{coro:n-inv}, that as $m$ varies among the $\phi(N-1)$
positive divisors of~$N-1$ the algebras that we obtain in this way are
pairwise non-isomorphic. These algebras can obviously be generated by two
elements and not by one.

\item Suppose now that $N-1$ divides properly~$m$, and let $l$ be the
integer such that $m=l(N-1)$, which is at least~$2$. In this situation the
subalgebra~$A^G$ is generated by the $m$th homo\-geneous component of~$A$,
which has dimension $l+1$, so that that $A^G$ can be generated by~$l+1$
elements. On the other hand, $A^G$ cannot be generated by~$l$ elements:
indeed, if there existed $l$ elements~$a_1$,~\dots,~$a_l$ in~$A^G$ that
generate it as an algebra, the homogeneous components of those elements of
degree~$m$ would generate the $m$th homo\-geneous component of~$A$, and
this is absurd. We thus see that $A^G$ is minimally generated by~$l+1$
elements and that therefore the subalgebras that we obtain in this way are
pairwise non-isomorphic and, since $l\geq2$, also non-isomorphic to any of
the algebras that we found in~\ref{it:ag:1}.

\end{thmlist}
The smallest case not covered by these considerations is that in which
$N=3$ and $m=3$.

\section{Hochschild cohomology}
\label{sect:cohomology}

In this section we present a rather effortless calculation of the
Hochschild cohomology of the algebra~$A$. Our approach is directly targeted
at obtaining completely explicit representatives of cohomology classes.
As before, we suppose throughout that $N\geq1$. The Hochschild cohomology
of the first Weyl algebra, the algebra we get when $N=0$, is
well-know to be the same as that of the ground field --- this was computed
originally by Ramaiyengar Sridharan in~\cite{Sridharan}.

\bigskip

Let $V$ be the subspace of~$A$ spanned by~$x$ and~$y$, let $V^*$ be its
dual vector space, and let $(\hat x,\hat y)$ be the ordered basis of~$V ^*$
dual to~$(x,y)$. There is a projective resolution~$P_*$ of~$A$ as an $A$-bimodule
of the form
  \[ \label{eq:res}
  \begin{tikzcd}
  0 \arrow[r]
    & A\otimes\Lambda^2V\otimes A\arrow[r, "d_2"]
    & A\otimes V\otimes A\arrow[r, "d_1"]
    & A\otimes A
  \end{tikzcd}
  \]
with differentials such that
  \begin{align}
  d_1(1\otimes x\otimes 1) 
        &= x\otimes 1-1\otimes x, \\
  d_1(1\otimes y\otimes 1) 
        &= y\otimes 1-1\otimes y, \\
  d_2(1\otimes x\wedge y\otimes 1) 
        &= 
         \begin{multlined}[t][0.6\displaywidth]
          y\otimes x\otimes 1 + 1\otimes y\otimes x
          - x\otimes y\otimes 1 - 1\otimes x\otimes y \\
          - \sum_{s+1+t=N}x^s\otimes x\otimes x^t
         \end{multlined}
  \end{align}
and augmentation $\epsilon:A\otimes A\to A$ given by the multiplication
of~$A$. Clearly, $V$ is a homogeneous subspace of~$A$, and if we endow it
with the induced grading, and in turn $\Lambda^2V$ with the grading induced
by that of~$V$ and each component of the complex~\eqref{eq:res} with the
obvious tensor product gradings, that complex becomes a complex of graded
$A$-bimodules. Applying to it the functor~$\hom_{A^e}(-,A)$ we obtain, up
to standard identifications, the cochain complex
  \[ \label{eq:comp}
  \begin{tikzcd}
    & A \arrow[r, "\delta_0"]
    & A\otimes V^* \arrow[r, "\delta_1"]
    & A\otimes \Lambda^2V^* \arrow[r]
    & 0
  \end{tikzcd}
  \]
with differentials such that
  \begin{align}
  \delta_0(a) &= [x,a]\otimes\hat x+[y,a]\otimes\hat y, \\
  \delta_1(b\otimes\hat x+c\otimes\hat y) &=
    \left([y,b] + [c,x] - \sum_{s+1+t=N}x^sbx^t\right)
    \otimes\hat x\wedge\hat y
  \end{align}
for all choices of~$a$, $b$ and~$c$ in~$A$. The cohomology of this complex
is canonically isomorphic to the Hochschild cohomology~$\HH^\bullet(A)$
of~$A$, and we \emph{identify} the two. If we grade~$V^*$ so that $\hat x$
and~$\hat y$ have degrees~$-1$ and~$-(N-1)$, respectively, then the
complex~\eqref{eq:comp} is one of graded vector spaces and homogeneous maps
of degree~$0$, and, consequently, its cohomology~$\HH^\bullet(A)$ is also
graded. This grading on Hochschild cohomology coincides with the canonical
grading it gets from the fact that the algebra~$A$ is graded and has a free
resolution as a graded bimodule over itself by finitely generated free
bimodules --- we will not belabor this point, but it is important that the
grading we found is \emph{the} grading on~$\HH^\bullet(A)$.

\bigskip

The hardest part of the calculation is that of~$\HH^1(A)$, and we will
leave it for the end, as we will do it in a rather indirect way. As for
that of~$\HH^0(A)$, we have actually already done it:

\begin{Lemma}
The $0$th cohomology space $\HH^0(A)$, the kernel of the map~$\delta_0$, is
$\kk$, and its Hilbert series is therefore $h_{\HH^0(A)}(t) = 1$.
\end{Lemma}

\begin{proof}
We established in Proposition~\ref{prop:center} that the center of~$A$,
which coincides with the kernel of the map~$\delta_0$, is~$\kk$ and then
that its Hilbert series is~$1$ is clear. 
\end{proof}

We know from Lemma~\ref{lemma:normal} that $x^N$ is a normal element
in~$A$, so that the right ideal~$x^NA$ is a bilateral ideal. Moreover, that
ideal is related to commutators in the following way:

\begin{Lemma}\label{lemma:xN}
We have $[A,x]=[A,A]=x^NA$.
\end{Lemma}

\begin{proof}
If $i$,~$j\geq0$, then 
  \[
  [x^iy^j,x] 
        = \sum_{s+1+t=j}x^iy^s[y,x]y^t
        = \sum_{s+1+t=j}x^iy^sx^Ny^t
        \in x^NA
  \]
because~$x^NA$ is an ideal, and this implies that $[A,x]\subseteq x^NA$.

To prove the reverse inclusion, we will show that $x^{i+N}y^j\in[A,x]$ for
all $i$,~$j\geq0$ by induction on~$j$. If $j\geq0$ and
$x^{i+N}y^k\in[A,x]$ for all $i\geq0$ and all $k\in\{0,\dots,j-1\}$ ---~this
hypothesis is vacuous when $j=0$, and this starts the induction~--- then
there is an $u\in F_{j-1}$ such that $[x^iy^{j+1},x] = (j+1)x^{i+N}y^j +
u$ and, because of the hypothesis, a $v\in A$ such that $u=[v,x]$: we then
have that $x^{i+N}y^j=[(j+1)^{-1}x^iy^{j+1}-v,x]\in[A,x]$. The induction is
thus complete.

If $i$,~$j$,~$k$,~$l\geq0$, then we have
  \[
  [x^iy^j,x^ky^l]
        = x^i[y^j,x^k]y^l + x^k[x^i,y^l]y^j,
  \]
so to prove that $[A,A]$ is contained in~$x^NA$, and thus the rest of the
equalities asserted by the lemma, it is enough to compute
that
  \[
  [y^j,x^k]
        = \sum_{\substack{s+1+t=j\\u+1+v=k}}
                y^sx^u[y,x]x^vy^t
        \in x^NA
  \]
because~$[y,x]=x^N$ and $x^NA$ is a bilateral ideal.
\end{proof}

The description of~$[A,A]$ that we have now allows us to compute~$\HH^2(A)$
easily.

\begin{Lemma}
The image of the map~$\delta_1$ is $x^{N-1}A\otimes\hat x\wedge\hat y$, and
therefore 
  \[
  \HH^2(A)\cong A/x^{N-1}A\otimes\hat x\wedge\hat y.
  \]
If $N=1$ then this is of course~$0$, and if $N\geq2$, so that $\HH^2(A)$ is
graded with finite-dimensional homogeneous components, then the
Hilbert series if this space is
  \[
  h_{\HH^2(A)}(t) = \frac{t^{-N}}{1-t}.
  \]
\end{Lemma}

\begin{proof}\allowdisplaybreaks
For all $c\in A$ we have $\delta_1(c\otimes\hat y)=[c,x]\otimes\hat
x\wedge\hat y$, and this, together with Lemma~\ref{lemma:xN}, tells us that
$x^NA\otimes\hat x\wedge\hat y=\delta_1(A\otimes\hat
y)\subseteq\img\delta_1$. On the other hand, if $i$,~$j\geq0$, we have that
  \begin{align}
  \delta_1(x^iy^j\otimes\hat x)
        &= \left([y,x^iy^j] - \sum_{s+1+t=N}x^{s+i}y^jx^t\right)
           \otimes\hat x\wedge\hat y \\
        &= \left(ix^{i+N-1}y^j - \sum_{s+1+t=N}x^{s+i}y^jx^t\right)
           \otimes\hat x\wedge\hat y \\
        &= \left((i-N)x^{i+N-1}y^j - \sum_{s+1+t=N}x^{s+i}[y^j,x^t] \right)
           \otimes\hat x\wedge\hat y. 
  \end{align}
According to Lemma~\ref{lemma:xN}, the sum appearing in this last
expression is an element of~$x^NA$: it follows from this that
$\delta_1(x^iy^j\otimes\hat x)\subseteq x^{N-1}A$, and then that
  \[
  x^NA 
        \subseteq \img\delta_1 
        = \delta_1(A\otimes\hat x)+\delta_1(A\otimes\hat y)
        \subseteq x^{N-1}A + x^{N}A = x^{N-1}A.
  \]
Moreover, for each $j\geq0$ we have that
  \[
  \delta_1(y^j\otimes\hat x)
        = \left(-Nx^{N-1}y^j - \sum_{s+1+t=N}x^s[y^j,x^y] \right)
           \otimes\hat x\wedge\hat y.
  \]
Since the sum appearing here is in~$x^NA$ it is equal to~$[b,x]$ for
some~$b\in A$ and therefore
  \[
  -Nx^{N-1}y^j\otimes\hat x\wedge\hat y = \delta_1(y^j\otimes\hat x+b\otimes\hat y) 
        \in \img\delta_1.
  \]
Putting everything together we conclude that $\img\delta_1=x^{N-1}A$, as
we want. Finally, the quotient $A/x^{N-1}A$ is freely spanned by the
classes of the monomials~$x^iy^j$ with $0\leq i<N-1$ and $j\geq0$, and
there is exactly one such monomial of each degree in~$\NN_0$: the Hilbert
series of~$A/x^{N-1}A$ is thus $(1-t)^{-1}$, and the Hilbert series
of~$\HH^2(A)$ is as described in the lemma, because of the factor~$\hat
x\wedge\hat y$.
\end{proof}

At this point we know the Hilbert series of~$\HH^0(A)$ and of~$\HH^2(A)$,
and we can use the invariance of the Euler characteristic of a complex
under the operation of taking cohomology to determine the Hilbert series
of~$\HH^1(A)$.

\begin{Proposition}\label{prop:hh1-series}
If $N\geq2$, then
the Hilbert series of~$\HH^1(A)$ is
  \[
  h_{\HH^1(A)}(t) = 1 + \frac{t^{-N+1}}{1-t},
  \]
so that for all integers~$l$ we have that
  \[
  \dim\HH^1(A)_l =
    \begin{cases*}
    2 & if $l=0$; \\
    1 & if $l\geq-N+1$ and $l\neq0$; \\
    0 & if $l\leq -N$.
    \end{cases*}
  \]
\end{Proposition}

\begin{proof}
The Hilbert series
of~$A$ is
  \[
  h_A(t) = \frac{1}{(1-t)(1-t^{N-1})},
  \]
and then the Euler characteristic of the complex~\eqref{eq:comp} is
  \[
  h_A(t) - (t^{-1}+t^{-(N-1)})h_A(t)+t^{-N}h_A(t) = t^{-N}.
  \]
The invariance of the Euler characteristic when passing to cohomology
implies now that
  \[
  t^{-N} = h_{\HH^0(A)}(t) - h_{\HH^1(A)}(t) + h_{\HH^2(A)}(t),
  \]
and one can compute~$h_{\HH^1(A)}(t)$ from this equality, since we know the
values of both $h_{\HH^0(A)}(t)$ and~$h_{\HH^2(A)}(t)$, finding the formula
given in the proposition.
\end{proof}

In Proposition~\ref{prop:hh1-series} we excluded the case in which $N=1$,
which is special --- for one thing, the graded algebra~$A_N$ is not locally
finite-dimensional in that case, so we cannot even talk about its Hilbert
series! Let us deal with it now for the sake of completeness.

\begin{Proposition}\label{prop:hh1-n1}
If $N=1$, then there is a derivation~$d_0:A\to A$ such that $d_0(x)=0$
and~$d_0(y)=1$. It is homogeneous of degree~$0$, and its cohomology class
freely spans the vector space~$\HH^1(A)$, which is thus $1$-dimensional and
concentrated in degree~$0$.
\end{Proposition}

When $N=1$ the map $\ad(y):u\in A\mapsto [y,u]\in A$ acts on elements of
each degree~$l\in\ZZ$ by multiplication by~$l$. We will use this fact in
the proof.

\begin{proof}
A trivial calculation shows that there is indeed a derivation~$d_0:A\to A$
with $d_0(x)=0$ and $d_0(y)=1$, and it is clearly homogeneous of
degree~$0$. Were it inner, we would have an element~$c$ such that $[x,c]=0$
and $[y,c]=1$: the first equality implies that $c\in\kk[x]$, and then the
second one that $x^Nc'=1$, which is absurd.

Let now $d:A\to A$ be an arbitrary homogeneous derivation and let $l$ be
its degree. If~$a\coloneqq d(x)$ and~$b\coloneqq d(y)$, then $a$ and~$b$
are homogeneous of~$A$ of degree~$l+1$ and~$l$, respectively. As $d$ is a
derivation, we have that
  \[ \label{eq:zero}
  0 = d([y,x]-x) = [b,x] + [y,a] - a = [b, x] + la.
  \]
If $l\neq0$ this tells us that $a=-\frac{1}{l}[b,x]$, and using this we see
at once that $d$ is the inner derivation~$\tfrac{1}{l}\ad(b)$. Let us then
suppose that $l=0$. Now the equality above tells us that $b$
commutes with~$x$, so that is an homogeneous element of~$\kk[x]$ of
degree~$0$: in other words, we have $b\in\kk$. On the other hand, since $a$
has degree~$1$, there is an $f\in\kk[y]$ such that $a=xf$. There exists a
$g\in\kk[y]$ such that $f=g(y+1)-g(y)$, and then $[g,x]=a$. The
derivation~$d'\coloneqq d-\ad(g)$ is therefore homogeneous of degree~$0$,
cohomologous to~$d$, and has $d'(x)=0$ and $d'(y)=b$. We can thus conclude
that every homogeneous derivation is cohomologous lo a multiple of the
derivation~$d_0$. This proves the proposition.
\end{proof}

If we are asked for a reason explaining the difference between the case in
which $N=1$ and that in which $N\geq2$ exhibited by the last two results, a
good candidate for an answer is the following. Whatever the value of~$N$,
the algebra~$A$ has a grading, $A=\bigoplus_{j\geq0}A_j$, and there is
therefore a derivation $E:A\to A$ such that for all $j\geq0$ the
component~$A_j$ is invariant under~$E$ and the restriction~$E|_{A_j}:A_j\to
A_j$ is simply multiplication by the scalar~$j$. This is a diagonalizable
derivation and the homogeneous components of~$A$ are precisely its
eigenspaces. Now a difference arises: if $N=1$, the derivation~$E$ is
inner, since it coincides with~$\ad(y)$, while if~$N\geq2$ the derivation
not inner. This is behind the collapse of~$\HH^1(A)$ that occurs when
$N=1$.

\section{A sequence of special elements in our algebra}
\label{sect:phi}

In the next section we will exhibit explicit derivations that freely span
the first cohomology vector space of our algebra,~$\HH^1(A)$, and to do
that we will need some calculations that we carry out in this one.
Despite the fact that we will use these results there only when $N\geq2$,
in this section we work with an arbitrary $N$ in~$\NN$, because when $N=1$
something interesting happens.

\begin{Lemma}\label{lemma:phi}
For each $j\geq1$ there exists a unique element~$\Phi_j$ in~$Ay$
homogeneous of degree $j(N-1)$ such that $[\Phi_j,x]=x^Ny^{j-1}$, and it is
such that 
  \[
  \Phi_j\equiv\frac{1}{j}y^{j}-\frac{N}{2}x^{N-1}y^{j-1}\mod F_{j-2}
  \]
when $j\geq2$ and $\Phi_1=y$ when $j=1$.
\end{Lemma}

\begin{proof}
The existence and uniqueness of the elements~$\Phi_j$ is an immediate
consequence of the exactness of the sequence of graded vector spaces
  \[
  \begin{tikzcd}
  0 \arrow[r]
    & \kk[x] \arrow[r, hook]
    & A \arrow[r, "\ad(x)"]
    & x^NA \arrow[r]
    & 0
  \end{tikzcd}
  \]
together with the fact that $Ay$ is a complement for~$\kk[x]$ in~$A$. 
It is clear that $\Phi_1=y$. Suppose now that $j\geq2$.
Since $\Phi_j$ is homogeneous of degree~$k(N-1)$
there are scalars $a_0$,~$a_1$,~\dots,~$a_{j}$ in~$\kk$ such that
$\Phi_j=\sum_{i=0}^{j}a_ix^{i(N-1)}y^{j-i}$, and then working
modulo~$F_{j-3}$ we have that
  \begin{align}
  x^Ny^{j-1} 
       &= [\Phi_j,x]
        = \sum_{i=0}^{j-1}a_ix^{i(N-1)}[y^{j-i},x]
        \equiv a_0[y^j,x]+a_1x^{N-1}[y^{j-1},x] \\
       &\equiv ja_0x^Ny^{j-1} 
               +(j-1)\left(\frac{a_0Nj}{2}+a_1\right)x^{2N-1}y^{j-2}
  \end{align}
so that $\alpha_0=1/j$ and $a_1=-N/2$. The last claim of the lemma follows
from this.
\end{proof}

Let us consider a variable~$q$ and the sequence~$(c_i)_{i\geq0}$ of
elements of $\QQ[q]$ such that
  \[ \label{eq:ci}
  \sum_{i=0}^j\frac{(1)_{q,j+1-i}}{(j+1-i)!}  \frac{c_i(q)}{i!}
        = \delta_{j,0}
  \]
for all $j\in\NN_0$. This condition recursively determines the
sequence starting with $c_0=1$. 

\begin{table}
\small\centering
\renewcommand{\arraystretch}{1.2}
\begin{tabular}{@{\hskip2em}r@{\hskip3pc}c@{\hskip2em}}
$i$ & $c_i(q)$ \\ \toprule
$0$ & $1$ 
		\\ \midrule
$1$ & $-\frac{1}{2} q - \frac{1}{2}$ 
		\\ \midrule
$2$ & $-\frac{1}{6} q^{2} + \frac{1}{6}$ 
		\\ \midrule
$3$ & $-\frac{1}{4} q^{3} + \frac{1}{4} q$ 
		\\ \midrule
$4$ & $-\frac{19}{30} q^{4} + \frac{2}{3} q^{2} - \frac{1}{30}$ 
		\\ \midrule
$5$ & $-\frac{9}{4} q^{5} + \frac{5}{2} q^{3} - \frac{1}{4} q$ 
		\\ \midrule
$6$ & $-\frac{863}{84} q^{6} + 12 q^{4} - \frac{7}{4} q^{2} 
        + \frac{1}{42}$ 
		\\ \midrule
$7$ & $-\frac{1375}{24} q^{7} + 70 q^{5} - \frac{105}{8} q^{3} 
        + \frac{5}{12} q$ 
		\\ \midrule
$8$ & $-\frac{33953}{90} q^{8} + 480 q^{6} - \frac{1624}{15} q^{4} 
        + \frac{50}{9} q^{2} - \frac{1}{30}$ 
		\\ \midrule
$9$ & $-\frac{57281}{20} q^{9} + 3780 q^{7} - \frac{9849}{10} q^{5} 
        + 70 q^{3} - \frac{21}{20} q$ 
		\\ \midrule
$10$ & $-\frac{3250433}{132} q^{10} + 33600 q^{8} - \frac{29531}{3} q^{6} 
        + \frac{5345}{6} q^{4} - \frac{91}{4} q^{2} + \frac{5}{66}$ 
		\\ \midrule
$11$ & $-\frac{1891755}{8} q^{11} + 332640 q^{9} - \frac{214995}{2} q^{7} 
        + \frac{47025}{4} q^{5} - \frac{3465}{8} q^{3} + \frac{15}{4} q$ 
		\\ \bottomrule
\end{tabular}
\caption{The polynomials $c_i(q)$ for small values of~$i$.}
\label{tbl:ciq}
\end{table}

We want to give a few properties of these polynomials, and in order to do
that we need two classical sequences of rational numbers: that of the
Bernoulli numbers $(B_j)_{j\geq0}$ and of the Gregory
coefficients~$(G_i)_{i\geq1}$. The first one is uniquely characterized by
the equalities
  \[
  \sum_{i=0}^j\binom{j+1}{i}B_i = \delta_{j,0}, \qquad\forall j\geq0
  \]
and the second one by the equalities
  \[
  G_1 = \frac{1}{2},
  \qquad
  \sum_{i=1}^j (-1)^{i+1}\frac{G_i}{j+1-i} = \frac{1}{j+1}, 
  \qquad\forall j\geq2.
  \]
For convenience we put additionally~$G_0=1$. The second sequence is not
particularly famous, but quite a bit of information about it can be found
in the article \cite{MSV} that deals with the sequence of Cauchy numbers
$(C_j)_{j\geq0}$, which has $C_j=G_j/j!$ for all $j\geq0$. The numbers in
our sequence appear in the  formula for approximate integration discovered
by James Gregory in~1668, and that is why they are named after him ---
there is also a crater in the Moon (located at
\href{https://trek.nasa.gov/moon/#v=0.1&x=127.2&y=2.2&z=5&p=urn%3Aogc%3Adef%3Acrs%3AEPSG%3A%3A104903&d=&locale=&b=moon}{N$2^\circ12'0''$
E$127^\circ12'0''$}) that carries his name: his work was mostly on
Astronomy. One of the ways these numbers enter the theory of approximate
integration is through the fact for all non-negative integers~$j$ we have
  \[
  G_j = \int_0^1\binom{x}{j}\,\d x.
  \]

The exponential generating function of the Bernoulli numbers and the
ordinary generating function of the Gregory coefficients are, respectively, 
 \[ \label{eq:gen:b}
 \sum_{j\geq0}B_j\frac{t^j}{j!} = \frac{t}{e^t-1},
 \qquad
 \sum_{j\geq0}G_jt^j = \frac{z}{\ln(1+z)}.
 \]
It can be checked that none of the Gregory coefficients vanish.

\begin{table}
  \small\centering
  \renewcommand{\arraystretch}{1.3}
  \begin{tabular}{@{}c@{\hspace{2em}}*{11}{c}@{}}
  $n$ & 0 & 1 & 2 & 3 & 4 & 5 & 6 & 7 & 8 & 9 & 10 
    \\ \toprule
  $B_n$ 
    & $1$ 
    & $-\frac{1}{2}$ 
    & $\frac{1}{6}$ 
    & $0$ 
    & $-\frac{1}{30}$ 
    & $0$ 
    & $\frac{1}{42}$ 
    & $0$ 
    & $-\frac{1}{30}$ 
    & $0$ 
    & $\frac{5}{66}$ 
    \\ \midrule
  $G_n$ 
    & $1$ 
    & $\frac{1}{2}$ 
    & $-\frac{1}{12}$ 
    & $\frac{1}{24}$ 
    & $-\frac{19}{720}$ 
    & $\frac{3}{160}$ 
    & $-\frac{863}{60480}$ 
    & $\frac{275}{24192}$ 
    & $-\frac{33953}{3628800}$ 
    & $\frac{8183}{1036800}$ 
    & $-\frac{3250433}{479001600}$ 
  \end{tabular}
\caption{The first Bernoulli numbers and Gregory coefficients.}
\end{table}

\begin{Lemma}\label{lemma:cj}
For each $j\geq0$ the polynomial~$c_j(q)$ has degree~$j$, leading
coefficient $(-1)^jj!G_j$ and constant term $B_j$.
The exponential generating series of the sequence $(c_j(q))_{j\geq0}$ is
  \[ \label{eq:gen}
  \sum_{j\geq0}c_j(q)\frac{t^j}{j!} = \frac{t}{(1-qt)^{-1/q}-1}.
  \]
\end{Lemma}

The limit as $q$ approaches~$0$ (taken in~$\CC$, of course) of the function
that appears in the right hand side of this last equality is the
exponential generating function for the Bernoulli numbers that we wrote
in~\eqref{eq:gen:b}.

\begin{proof}
Setting~$q$ to~$0$ in~\eqref{eq:ci} we see that for all $j\in\NN_0$ we have
that
  \[
  \sum_{i=0}^j c_i(0)\binom{j+1}{i} = \delta_{j,0},
  \]
and comparing this with the defining recurrence equation for the Bernoulli
numbers shows that for all~$j\in\NN_0$ the constant term of~$c_j$ is~$B_j$.
As $c_0(q)=1$ and $c_1(q)$ is the constant polynomial~$1$, it is clear that
its degree and its leading coefficient are~$0$ and~$(-1)^00!G_0$,
respectively. On the other hand, if $j>0$, then according to~\eqref{eq:ci}
we have that
  \[
  \frac{c_j(q)}{j!} 
    = -\sum_{i=0}^{j-1} \frac{(1)_{q,j+1-i}}{(j+1-i)!}\frac{c_i(q)}{i!}
  \]
and, since $(1)_{q,j+1-i}$ is a polynomial of degree~$j-i$ and leading
coefficient~$(j-i)!$, the first claim of the lemma follows by induction
from the definition of the Gregory coefficients and the fact that they are
all non-zero.

Finally, the left hand side of the defining equation~\eqref{eq:ci} is the
coefficient of~$t^j$ in the product
  \[
  \sum_{j\geq0}c_j(q)\frac{t^j}{j!}
  \cdot
  \sum_{j\geq0}\frac{(1)_{q,j+1}}{(j+1)!}t^j,
  \]
whose second factor sums to $((1-qt)^{-1/q}-1)/t$: the
equality~\eqref{eq:gen} follows from this.
\end{proof}

Using the polynomials~$c_j(q)$ we are able to write down explicitly the
elements~$\Phi_j$ from Lemma~\ref{lemma:phi}.

\begin{Proposition}\mbox{}\label{prop:phi}
For each $j\geq1$ we have 
  \[ \label{eq:psi}
  \Phi_j = 
        \frac{1}{j}
        \sum_{i=0}^{j-1}
        \binom{j}{i}
        c_i(N-1)
        x^{i(N-1)}y^{j-i}.
  \]
\end{Proposition}

\begin{proof}
Let us work in the algebra~$A[[t]]$ of formal power series with
coefficients in~$A$. An easy induction shows that
$\ad(y)^j(x) = (1)_{N-1,j}x^{1+j(N-1)}$
for all $j\in\NN_0$, so that
  \[
  e^{yt}xe^{-yt} 
        = \sum_{j\geq0}\ad(y)^j(x)\frac{t^j}{j!} 
        = \sum_{j\geq0}(1)_{N-1,j}x^{1+j(N-1)}t^j 
        = x\,(1-(N-1)x^{N-1}t)^{-\frac{1}{N-1}}
  \]
and
  \[ \label{eq:exp}
  [e^{yt},x] 
        = \bigl(e^{yt}xe^{-yt}-x\bigr)e^{yt}
        = x\left(
                \left(1-(N-1)x^{N-1}t)^{-\frac{1}{N-1}}\right)
                -1
           \right)e^{yt}.
  \]
If we write~$\Psi_j$ the right hand side of the equality~\eqref{eq:psi}
that we want to prove, we have that
  \begin{align}
  \sum_{j\geq1}\Psi_j\frac{t^{j-1}}{(j-1)!}
  &= \sum_{j\geq0}c_j(N-1)\frac{x^{j(N-1)}t^j}{j!}
      \cdot
      \sum_{j\geq1} \frac{y^jt^j}{j!} \\
  &= \frac{x^{N-1}t}{(1-(N-1)x^{N-1}t)^{-1/(N-1)}-1}
     \cdot
     \frac{e^{yt}-1}{t}, 
  \end{align}
so that
  \begin{align}
  \sum_{j\geq1}[\Psi_j,x]\frac{t^{j-1}}{(j-1)!}
   = \frac{x^{N-1}t}{(1-(N-1)x^{N-1}t)^{-1/(N-1)}-1}
     \cdot
     \left[\frac{e^{yt}-1}{t},x\right]
   = x^{N}e^{yt},
  \end{align}
using~\eqref{eq:exp} to obtain the last equality. Looking at the
coefficients in these series, we see that $[\Psi_j,x]=x^Ny^{j-1}$ for all
$j\in\NN$. Since $\Psi_j$ is homogeneous of degree~$j(N-1)$ and belongs
to~$Ay$, we can conclude that $\Psi_i=\Phi_j$, as the proposition claims.
\end{proof}

Let us single out three special cases of this proposition.
\begin{itemize}

\item If $N=1$, then for all~$j\geq1$ we have
  \[ \label{eq:bern}
  \Phi_j = \frac{1}{j}\sum_{i=0}^{j-1}\binom{j}{i}B_iy^{j-i}
        = \frac{B_j(y)-B_j}{j}, 
  \]
with $B_j(t)\coloneqq\sum_{i=0}^j\binom{j}{i}B_it^{j-i}\QQ[t]$, the usual
$j$th Bernoulli polynomial. The last member of this equality appears in the
famous Faulhaber formula for the sums of powers of the first integers: for
any $n$,~$p\in\NN$ we have
  \[
  \frac{B_p(n) - B_p}{p} = \sum_{k=0}^{n-1} k^{p-1}.
  \]
It would be interesting to know what is behind this.

\item If $N=2$, then the right hand side of the equality~\eqref{eq:gen}
that appears in
Lemma~\ref{lemma:cj} with $q=N-1=1$ is simply $1-t$, so that
  \[
  c_0(1) = 1, 
  \qquad
  c_1(1) = -1,
  \qquad
  c_j(1) = 0 \quad\text{if $j\geq2$.}
  \]
It follows from this and the proposition that for all $j\in\NN$ we have
  \[
  \Phi_j
        = \frac{y^j-jxy^{j-1}}{j}.
  \]

\item If $N=3$, then the right hand side of~\eqref{eq:gen}  with $q=N-1=2$
simplifies to 
  \[
  \frac{\sqrt{1-2t}-(1-2t)}{2},
  \]
and we obtain its Taylor series at once from Newton's binomial series,
getting
  \[
  c_0(2) = 1, 
  \qquad
  c_1(2) = -\frac{3}{2},
  \qquad
  c_j(2) = -\frac{(2j-3)!!}{2} \quad\text{if $j\geq2$.}
  \]
From this we can get explicit yet unenlightening  descriptions of the
elements~$\Phi_j$.

\end{itemize}
If $N\geq4$ and we put $q=N-1$ in~\eqref{eq:gen}, then the coefficients of
the series do not seem to have a simple form --- they do not appear in the
OEIS \cite{OEIS}. 

\section{Explicit representatives for classes in
\texorpdfstring{$\HH^1(A)$}{HH1(A)}}
\label{sect:explicit}

In this section, as promised, we will give explicit derivations
representing the elements of~$\HH^1(A)$ when $N\geq2$. We start with a
simple observation that allows us to show that some derivations are not
inner.

\begin{Lemma}\label{lemma:not-inner}
Let $d:A\to A$ be a derivation.
\begin{thmlist}

\item If $d(x)=0$, then $d(y)\in\kk[x]$.

\item If $d$ is inner, then $d(x)\in x^NA$. If additionally $d(x)=0$, then
also $d(y)\in x^N\kk[x]$.

\end{thmlist}
\end{Lemma}

\begin{proof}
\thmitem{1} If $d(x)=0$, then $0 = d([y,x]-x^N) = [d(y),x]$, and $d(y)$ is
in the centralizer of~$x$ in~$A$, which we know to be equal to~$\kk[x]$.

\thmitem{1} If $d$ is inner, then there is an $u\in A$ such
that $d(a)=[u,a]$ for all $a\in A$ and, in particular, $d(x)=[u,x]\in
[A,x]=x^NA$. If additionally $d(x)=0$, then $[u,x]=0$ and therefore
$u\in\kk[x]$: it follows from this that $d(y)=[u,y]=x^Nu'\in  x^N\kk[x]$.
\end{proof}

The simplest derivations of our algebra are those of non-positive degree.
We describe them in the following result.

\begin{Proposition}\label{prop:hh1-low}
Suppose that $N\geq2$. For each $l\in\inter{-N+1,0}$ there is a derivation
$\partial_l:A\to A$ of degree~$l$ such that
  \begin{align}
  \partial_l(x) &= 0, 
  &
  \partial_l(y) &= x^{l+N-1},
\intertext{and there is a derivation $E:A\to A$ of degree~$0$ such that}
  E(x) &= x, 
  &
  E(y) &= (N-1)y.
  \end{align}
For each $l\in\inter{-N+1,-1}$ the space~$\HH^1(A)_l$ is freely
spanned by the class of~$\partial_l$, and $\HH^1(A)_0$ is freely spanned
by~$\partial_0$ and~$E$. 
\end{Proposition}

The derivation~$\partial_0$ is the same one we encountered in
Corollary~\ref{coro:expp0}.

\begin{proof}
A very simple calculation shows that there are indeed derivations of~$A$ as
described in the statement of the proposition, and they manifestly have the
degrees given there. None of the derivations
$\partial_{-N+1}$,~\dots,~$\partial_0$ and~$E$ is inner ---~this follows
immediately from the second part of Lemma~\ref{lemma:not-inner}, because
$N\geq2$~--- and then, in view of Proposition~\ref{prop:hh1-series}, we see
that second claim of the proposition holds.
\end{proof}

To deal with classes in~$\HH^1(A)$ of positive degree, we start by showing
that they have a representative with a useful normalization:

\begin{Lemma}\label{lemma:normalization}
Suppose that $N\geq2$. Let $l$ be a positive integer, and let $i$ and $j$
be the integers such that $l+1=i+j(N-1)$, $j\geq0$, and $1\leq i\leq
N-1$. There exists a derivation $d:A\to A$ that is homogeneous of
degree~$l$, is not inner, and has $d(x)  = x^iy^j$.
\end{Lemma}

\begin{proof}
Let $l$ be a positive integer. The vector space $\HH^1(A)_l$ has
dimension~$1$. We pick a derivation~$d:A\to A$ homogeneous of degree~$l$
whose cohomology class spans that vector space. The element~$d(x)$ of~$A$
is then homogeneous of degree~$l+1$, and we can write it in the form
$u+x^Nv$ with $u$ an homogeneous element of
$P\coloneqq\sum_{k=0}^{N-1}x^k\kk[y]$ of degree~$l+1$, and $v$ a
homogeneous element of~$A$ of degree~$l+1-N$. There is a $w\in A$ such that
$[w,x]=x^Nv$, and it can be taken to be homogeneous of degree~$l$: the
derivation $d'\coloneqq d-\ad(w)$ is then homogeneous of degree~$l$,
cohomologous to~$d$, and has $d'(x)\in P$. The upshot of all this is that
we could have simply chosen our original derivation~$d$ so that $d(x)$ is
in~$P$ from the start, and we do so now.

An homogeneous element of degree~$l+1$ in~$P$, such as $d(x)$, is a linear
combination of the monomials of the form $x^ry^s$ with $r+s(N-1)=l+1$ and
$0\leq r<N$. We consider now two cases.
\begin{itemize}

\item Suppose first that $N-1\nmid l+1$, so that in each such monomial we
have in fact that $1\leq r<N-1$, and therefore that~$r$ and~$s$ are
necessarily equal to~$i$ and~$j$, respectively: in particular, there is
exactly one such monomial, $x^iy^j$, and there is thus a scalar~$\alpha$
such that $d(x)=\alpha x^iy^j$.

\item Suppose next that $N-1\mid l+1$, so that $i=N-1$. There are
now two monomials of degree~$l+1$ in~$P$, namely $x^{N-1}y^{j}$ and
$y^{j+1}$, and there are scalars~$\alpha$ and~$\beta$ such that $d(x) = \alpha
x^{N-1}y^{j} + \beta y^{j+1}$. As
  \begin{align}
  0 &= d([y,x]-x^N)
     = [d(y),x] + [y,d(x)] - \sum_{s+1+t=N}x^sd(x)x^t \\
    &= [d(y),x] + \alpha(N-1)x^{2N-2}y^{j} 
         - \alpha \sum_{s+1+t=N}x^{s+N-1}y^{j}x^t  
         - \beta\sum_{s+1+t=N}x^sy^{j+1}x^t
       \\
    &\equiv -\beta Nx^{N-1}y^{j+1} \mod(F_{j}+x^NA),
  \end{align}
the scalar~$\beta$ is actually~$0$.
\end{itemize}
The end result of all this is that in any case there is a scalar~$\alpha$
such that $d(x) = \alpha x^iy^j$. If $\alpha$ is~$0$, then the first part
of Lemma~\ref{lemma:not-inner} tells us that $d(y)\in\kk[x]$ and, since
$d(y)$ is a homogeneous element of degree~$l+N-1$, that in fact there is a
scalar~$\gamma$ such that $d(y)=\gamma x^{l+N-1}$: this is impossible,
since in that case we have $d=\gamma\ad(x^l)$, and the derivation~$d$ is
not inner. The derivation $\alpha^{-1}d$ thus satisfies the condition we
want.
\end{proof}

With the same idea that we used in the end of this proof we can also obtain
the following criterion that allows us to prove a derivation is inner:

\begin{Lemma}\label{lemma:cut}
Suppose that $N\geq2$. Let $l$ be a positive integer, and let $i$ and $j$
be the integers such that $l+1=i+j(N-1)$, $j\geq0$, and $1\leq i\leq N-1$.
A homogeneous derivation $d:A\to A$ of degree~$l$ such that $d(x)\in
F_{j-1}$ is inner.
\end{Lemma}

\begin{proof}
Let $d:A\to A$ be a homogeneous derivation of degree~$l$ such that $d(x)$
belongs to~$F_{j-1}$. There are then scalars $a_1$,~\dots,~$a_j$ in~$\kk$
such that $d(x)=\sum_{k=1}^ja_kx^{i+k(N-1)}y^{j-k}$ and therefore $d(x)\in
x^NA$. As we know, this implies that there is an element~$u$ in~$A$, which
can be chosen of degree~$l$, such that $d(x)=[u,x]$, and therefore the
derivation $d'\coloneqq d-\ad(u)$ is homogeneous of degree~$l$ and vanishes
on~$x$. According to the first part of Lemma~\ref{lemma:not-inner}, we have
$d'(y)\in k[x]$, so that in view of the homogeneity of~$d'$ there is a
scalar~$b$ in~$\kk$ such that $d'(y)=bx^{l+N-1}=-\frac{b}{l}[x^{l},y]$.
It follows from this that $d=\ad(u)-\frac{b}{l}\ad(x^{l})$, and this
proves the lemma.
\end{proof}

We need yet one more commutation identity. We promise it is the last one.

\begin{Lemma}\label{lemma:identity}
For each element $u$ of $A$ we have that
  \begin{equation*}
  \sum_{s+1+t=N}x^sux^t - Nx^{N-1}u
        = \left[
          \sum_{s+2+t=N}(s+1)x^sux^t,x
          \right].
  \end{equation*}
\end{Lemma}

\begin{proof}
The identity can be proved by expanding the commutators
appearing in its right hand side and simplifying.
\end{proof}

Using all these results, we can finally exhibit representatives for classes
in~$\HH^1(A)$:

\begin{Proposition}\label{prop:hh1-high}
Suppose that $N\geq2$. Let $l$ be a positive integer, and let $i$ and~$j$
be the integers such that $l+1=i+j(N-1)$, $j\geq0$, and $1\leq i\leq N-1$.
The vector space~$\HH^1(A)_l$ is spanned by a derivation $\partial_l:A\to
A$ that is homogeneous of degree~$l$ and such that
  \[
  \partial_l(x) = x^iy^j,
  \qquad
  \partial_l(y) = \sum_{s+2+t=N}(s+1)x^{s+i}y^jx^t + (N-i)x^{i-1}\Phi_{j+1}.
  \]
\end{Proposition}

The element~$\Phi_j$ appearing here is the one from
Proposition~\ref{prop:phi} of Section~\ref{sect:phi}. In this proposition
the integer~$l$ is assumed to be positive: we can observe that if we allow
it to be~$0$, then the formulas in the statement actually do also produce a
derivation of degree~$0$ that maps~$x$ and~$y$ to~$x$ and to
$N(N-1)x^{N+1}/2+(N-1)y$, and that this is the derivation
$E+N(N-1)\partial_0/2$ that we have already found in
Proposition~\ref{prop:hh1-low}.

\begin{proof}
Let $d:A\to A$ be a derivation that is homogeneous of degree~$l$
and not inner, and that satisfies the condition of
Lemma~\ref{lemma:normalization}, so that $d(x)=x^iy^j$. As $[y,x]-x^N=0$
in~$A$, we have, using the result of Lemma~\ref{lemma:identity}, that
  \begin{align}
  [d(y),x] 
        &= \sum_{s+1+t=N}x^sd(x)x^t - [y,d(x)]  \\
        &= \left[
           \sum_{s+2+t=N}(s+1)x^{s+i}y^jx^t,x
           \right] 
           + (N-i)x^{i+N-1}y^j,
  \end{align}
and therefore that
  \[
  \left[
  d(y) - \sum_{s+2+t=N}(s+1)x^{s+i}y^jx^t - (N-i)x^{i-1}\Phi_{j+1},
  x
  \right]
  = 0.
  \]
If follows from this that there exists an $f\in\kk[x]$ such that
  \[
  d(y) = \sum_{s+2+t=N}(s+1)x^{s+i}y^jx^t + (N-i)x^{i-1}\Phi_{j+1}+ f.
  \]
Clearly $f$ has to be homogeneous of degree~$l+N-1$, so equal to
$(l+N-1)\lambda x^{l+N-1}$ for some~$\lambda\in\kk$. The derivation
$\partial_l\coloneqq d-\lambda\ad(x^{l-1}y)$, which is cohomologous to~$d$,
is as in the statement of the theorem.
\end{proof}

We remark the nice fact that we were able exhibit the derivation
in~Proposition~\ref{prop:hh1-high} without needing to prove that it
actually is a derivation ---~which is probably a very unpleasant
calculation! 

\section{The Lie algebra structure on \texorpdfstring{$\HH^1(A)$}{HH1(A)}}
\label{sect:lie}

Now that we have explicit derivations whose classes freely span  $\HH^1(A)$
we can compute the canonical Lie algebra structure of this vector space. In
what follows we will write $\sim$ for the relation of cohomology between
derivations of~$A$.

\begin{Lemma}\label{lemma:brackets}
Let $l$ and $m$ be two integers such that $l\geq -N+1$, $m\geq -N+1$, and
let $i$, $j$, $u$ and~$v$ be the unique integers such that $l+1=i+j(N-1)$,
$m+1=u+v(N-1)$, $i$,~$u\in\inter{1,N-1}$ and $j$,~$v\geq-1$. Suppose that,
moreover, $l\leq m$.
\begin{thmlist}

\item If $i+u>N$, then $[\partial_l,\partial_m]\sim0$.

\item If $i+u\leq N$, then $[\partial_l,\partial_m]$ is cohomologous to
  \[
  \begin{c-dcases*}
    \left(u-i+\frac{N-i}{j+1}v-\frac{N-u}{v+1}j\right)\partial_{l+m}
          & if $l\geq1$ and $m\geq1$; \\
    0     & if $l\leq-1$ and~$m\leq-1$ or $lm=0$ \\
    v\partial_{l+m}
          & if $lm<0$ and $l+m\geq1$; \\
    -(l+m)\partial_{l+m}
          & if $lm<0$ and $l+m\leq-1$; \\
    E     & if $l=-N+1$ and~$m=N-1$.
  \end{c-dcases*}
  \]
Here we cannot have that $lm<0$, $l+m=0$ and $l>-N+1$.

\item Finally, $[E,\partial_m]=m\partial_m$.

\end{thmlist}
\end{Lemma}

\begin{proof}
The third claim of the lemma can be proved by a very simple direct
calculation that we omit. We will split the calculation that proves the
first two claims in several parts, and the following table describes in
which part we consider each particular combination of indices~$l$ and~$m$.
  \[
  \renewcommand{\arraystretch}{1.2}
  \begin{tabular}{c@{\hskip3em}c@{\hskip2em}c@{\hskip2em}c}
  {} & $l\leq-1$ & $l=0$ & $l\geq1$ \\ \toprule
  $m\leq-1$ & \textsc{Second} & \textsc{Third} & \textsc{Fourth} \\
  $m=0$     & \textsc{Third}  & \textsc{Third} & \textsc{Third} \\
  $m\geq1$  & \textsc{Fourth} & \textsc{Third} & \textsc{First}
  \end{tabular}
  \]

\textsc{First part.} Let us suppose first that $l\geq1$ and~$m\geq1$, so
that $j\geq0$ and $v\geq0$. If~$j\geq1$, then in view of
Proposition~\ref{prop:hh1-high} and Lemma~\ref{lemma:phi} we have that
modulo $F_{j-1}$
  \begin{align}
  \partial_l(x) 
        &\equiv x^iy^j
\shortintertext{and}
  \partial_l(y) 
        &=\sum_{s+2+t=N}(s+1)x^{s+i}y^jx^t + (N-i)x^{i-1}\Phi_{j+1} \\
        &\equiv \frac{N(N-1)}{2}x^{i+N-2}y^j
                +(N-i)x^{i-1}
                 \left(
                 \frac{1}{j+1}y^{j+1}-\frac{N}{2}x^{N-1}y^{j}
                 \right) \\
        &= \frac{N-i}{j+1}x^{i-1}y^{j+1} + \frac{N(i-1)}{2}x^{i+N-2}y^j,
  \end{align}
and using this we can see that
  \[ \label{eq:dldm}
  \partial_l(\partial_m(x))
        = \partial_l(x^uy^v)
        \equiv \left(u+\frac{N-i}{j+1}v\right) x^{i+u-1}y^{j+v}
        \mod F_{j+v-1}.
  \]
If instead $j=0$, then Lemma~\ref{lemma:phi} gives us a slightly different
formula for~$\Phi_{j+1}$ and what we have is that $\partial_l(x)=x^i$ and
  \[
  \partial_l(y) 
         =\sum_{s+2+t=N}(s+1)x^{s+i+t} + (N-i)x^{i-1}\Phi_{1} 
         = \frac{N(N-1)}{2}x^{i+N-2}
                +(N-i)x^{i-1}
                 y
  \]
and using this we can see that the congruence~\eqref{eq:dldm} also holds
when $j=0$.

By symmetry we get from~\eqref{eq:dldm} a formula for
$\partial_m(\partial_l(x))$, and finally conclude that
  \[ \label{eq:plpm:x}
  [\partial_l,\partial_m](x) 
        \equiv
                \left(
                u-i+\frac{N-i}{j+1}v-\frac{N-u}{v+1}j
                \right) x^{i+u-1}y^{j+v}
                \mod F_{j+v-1}.
  \]
There are now two cases.
\begin{itemize}

\item If $i+u>N$, then we have 
  \[
  l+m+1=(i+u-N)+(j+v+1)(N-1), 
  \qquad
  1\leq i+u-N\leq N-1,
  \]
and the above formula~\eqref{eq:plpm:x} tells us that the
derivation~$[\partial_l,\partial_m]$, which is homogeneous of degree~$l+m$,
maps~$x$ into~$F_{j+v}$: it follows from Lemma~\ref{lemma:cut} that
$[\partial_l,\partial_m]$ is inner in this situation.

\item Suppose now that $i+u\leq N$, so that
  \[
  l+m+1=(i+u-1)+(j+v)(N-1),
  \qquad
  1\leq i+u-1\leq N-1,
  \]
and let $\alpha$ be the scalar that appears between parentheses
in the right hand side of the congruence~\eqref{eq:plpm:x}. We then have that
the derivation $[\partial_l,\partial_m]-\alpha\partial_{l+m}$, which is
homogeneous of degree~$l+m$, maps~$x$ into~$F_{j+v-1}$: again using
Lemma~\ref{lemma:cut} we see that that difference is inner. 
We thus have that $[\partial_l,\partial_m]\sim\alpha\partial_{l+m}$ in this
situation.

\end{itemize}

\textsc{Second part.} If $l\leq-1$ and $m\leq-1$, then it follows immediately
from the description of~$\partial_l$ and~$\partial_m$ given in
Proposition~\ref{prop:hh1-low} that $[\partial_l,\partial_m]=0$,
independently of whether the inequality $i+u\geq N$ holds or not.

\medskip

\textsc{Third part.} We now want to prove that the
derivation~$[\partial_l,\partial_m]$ is inner if one of~$l$ or~$m$ is zero
and, of course, we can suppose that it is $l$ that is zero. Let us notice
that in this situation we have that 
$i=1$, $j=0$ and $i+u\leq N$.

If $m\geq1$, then $\partial_m(\partial_0(x))=0$ and
  \[
  \partial_0(\partial_m(x))
        = \partial_0(x^uy^v)
        = \sum_{s+1+t=v}x^uy^sx^{N-1}y^t
        \in F_{v-1},
  \]
so that $[\partial_0,\partial_m]$ is a homogeneous derivation of degree~$m$
that maps~$x$ into~$F_{v-1}$: as $m+1=u+v(N-1)$ and $1\leq u\leq N-1$, we know
from Lemma~\ref{lemma:cut} that that commutator is inner. If instead
$m\leq0$, then using the description of~$\partial_m$ given in
Proposition~\ref{prop:hh1-low} can compute directly that
$[\partial_0,\partial_m]=0$.

\medskip

\textsc{Fourth part.} Finally, let us consider the case in which $lm<0$
and, without any loss of generality thanks to anti-symmetry, $l\leq-1$ and
$m\geq1$. We have that 
  \begin{gather}
  j = -1, \qquad
  i = l+N,  \qquad
  v \geq 0, \\
  l+m+1=l+u+v(N-1), \qquad
  -N+2\leq l+u\leq N-2, \label{eq:r2}
  \end{gather}
and that, as $\partial_m(\partial_l(x))=0$, 
  \[
  [\partial_l,\partial_m](x)
        = \partial_l(\partial_m(x))
        = \partial_l(x^uy^v)
        = \sum_{s+1+t=v}x^uy^sx^{l+N-1}y^t. \label{eq:vx0}
  \]
Let us suppose first that $i+u>N$.
\begin{itemize}

\item If $v=0$ and $i+u>N+1$, then 
$l+m = i-N+u-1 \geq 1$, so that the degree of~$[\partial_l,\partial_m]$ is
positive, and, since $1\leq i-N+u= l+u\leq N-2$, it follows from the
first equality in~\eqref{eq:r2}, the formula~\eqref{eq:vx0} tells us 
that $[\partial_l,\partial_m](x)=0\in F_{v-1}$, and Lemma~\ref{lemma:cut},
that $[\partial_l,\partial_m]\sim0$.

\item If $v=0$ and $i+u=N+1$, then the formula~\eqref{eq:vx0} tells us
again that $[\partial_l,\partial_m](x)=0$, and using the definitions
of~$\partial_l$ and~$\partial_m$ we see that
  \[
  \partial_l(\partial_m(y))
        = \partial_l
                \left(
                \frac{N(N-1)}{2}x^{u+N-2}+(N-u)x^{u-1}y
                \right)
        = (N-u)x^{N-1}
  \]
and
  \[
  \partial_m(\partial_l(y))
        = \partial_m(x^{l+N-1})
        = (l+N-1)x^{N-1}
        = (i-1)x^{N-1}
  \]
so that $[\partial_l,\partial_m](y)=(N-u-i+1)x^{N-1}=0$ and, therefore,
$[\partial_l,\partial_m]=0$.

\item If $v\geq1$, then $l+m=i-N+u+v(N-1)-1>N-2\geq0$, so that the degree
of the derivation~$[\partial_l,\partial_m]$ is positive, and since
$l+m+1=i-N+u+v(N-1)$, $1\leq i-N+u\leq N-2$, and
$[\partial_l,\partial_m](x)\in F_{v-1}$ because of~\eqref{eq:vx0},
Lemma~\ref{lemma:cut} tells us that $[\partial_l,\partial_m]\sim0$.

\end{itemize}
Let us suppose, finally, that $i+u\leq N$. We then have that 
  \[ \label{eq:vy0}
  l+m+1 = (i+u-1)+(v-1)(N-1), \qquad
  1\leq i+u-1\leq N-1,
  \]
and, as before, we consider several cases.
\begin{itemize}

\item If $l+m\geq1$, then \eqref{eq:vy0} implies that
$\partial_{l+m}(x)=x^{i+u-1}y^{v-1}$ while \eqref{eq:vx0} implies that
$[\partial_l,\partial_m](x)\equiv vx^{i+u-1}y^{v-1}\mod F_{v-1}$: it
follows from this that $[\partial_l,\partial_m]-v\partial_{l+m}$, a
homogeneous derivation of degree~$l+m$, maps~$x$ into~$F_{v-2}$, so that
Lemma~\ref{lemma:cut} and~\eqref{eq:vy0} let us conclude that it is inner
and, therefore, that $[\partial_l,\partial_m]\sim v\partial_{l+m}$.

\item Suppose now that $l+m\leq -1$. We have
that $v=0$, for otherwise $v\geq1$ and 
  \[
  0 \geq l+m+1 = i-N+u+v(N-1) \geq i-N+u+N-1 = i+u-1 \geq1,
  \]
which is absurd. From~\eqref{eq:vx0} we see then that
$[\partial_l,\partial_m](x)=0$ and we can compute directly that
  \begin{gather}
  \partial_l(\partial_m(y))
        = \partial_l
                \left(
                \frac{N(N-1)}{2}x^{u+N-2} + (N-u)x^{u-1}y
                \right)
        = (N-u)x^{l+m+N-1}
\shortintertext{and}
  \partial_m(\partial_l(y))
        = \partial_m(x^{l+N-1})
        = (l+N-1)x^{l+m+N-1},
  \end{gather}
so that $[\partial_l,\partial_m](y)=-(l+m)x^{l+m+N-1}$, and therefore
$[\partial_l,\partial_m]=-(l+m)\partial_{l+m}$.

\item It is easy to check that we cannot have $l+m=0$ and $l>-N+1$.

\item We have one last case to consider: that in which $l+m=0$ and $l=-N+1$.
We then have that $m=N-1$, $u=1$ and $v=1$, and computing directly we see
that $[\partial_l,\partial_m]$ maps~$x$ and~$y$ to~$x$ and~$(N-1)y$, so
that it is equal to~$E$.

\end{itemize}
We have proved all the claims in the lemma.
\end{proof}

It will be convenient to work with a different basis of~$\HH^1(A)$ with
respect to which the structure constants of the Lie bracket are simpler. If
$l$ is an integer such that $l\geq-N+1$, there is a unique way of choosing
integers~$i$ and~$j$ such that $i\in\inter{1,N-1}$, $j\geq-1$ and
$l+1=i+j(N-1)$, and we define
  \[
  L_l \coloneqq
        \begin{r-dcases*}
        -\frac{j+1}{N-1}\partial_l & if $l\geq1$; \\ 
        -\frac{1}{N-1}E & if $l=0$; \\
        \frac{l}{N-1} \partial_l & if $l\leq -1$.
        \end{r-dcases*}
  \]
Clearly the classes of $\partial_0$ and the derivations $L_l$ with
$l\geq-N+1$ freely span~$\HH^1(A)$.

\begin{Corollary}\label{coro:mlmm}
Let $l$ and $m$ be two integers such that $l\geq -N+1$, $m\geq -N+1$, and
let $i$, $j$, $u$ and~$v$ be the unique integers such that $l+1=i+j(N-1)$,
$m+1=u+v(N-1)$, $i$,~$u\in\inter{1,N-1}$ and $j$,~$v\geq-1$. We have that
  \[ \label{eq:mlmm}
  [L_l, L_m] \sim
    \begin{dcases*}
    0 & if $i+u>N$ or $l+m<-N+1$; \\
    \frac{l(v+1)-m(j+1)}{N-1}L_{l+m} & if $i+u\leq N$.
    \end{dcases*}
  \]
\end{Corollary}

\begin{proof}
Let us consider first the situation in which $i+u>N$. Neither~$i$ nor~$u$
is equal to~$1$: were $i=1$, for example, we would have that $i+u\leq
1+(N-1)=N$. It follows then, in particular, that $l\neq0$ and $m\neq0$, so
there are scalars~$a$ and~$b$ such that $L_l=a\partial_l$ and
$L_m=b\partial_m$, and therefore that
$[L_l,L_m]=ab[\partial_l,\partial_m]\sim0$ because of the first part of
Lemma~\ref{lemma:brackets}. This proves the first line in~\eqref{eq:mlmm}.

In order to prove the second line let us suppose from now on that 
$i+u\leq N$ and, as if bound by a hex cast upon us, handle each of the 
several following special cases separately.
\begin{itemize}

\item Suppose first that $l\leq-1$ and $m\leq-1$. As before, there are
scalars~$a$ and~$b$ such that $L_l=a\partial_l$ and $L_m=b\partial_m$, and
therefore $[L_l,L_m]=ab[\partial_l,\partial_m]\sim0$ according to
Lemma~\ref{lemma:brackets}. On the other hand, we have that $j=-1$ and
$v=-1$, so that the numerator of the fraction appearing in~\eqref{eq:mlmm}
is~$0$ and that cohomology holds in this case.

\item Suppose now that $l\geq1$ and $m\geq1$. Using
Lemma~\ref{lemma:brackets} we see that $[L_l,L_m]$ is 
  \[
  \frac{(j+1)(v+1)}{(N-1)^2} [\partial_l,\partial_m] 
        \sim \frac{(j+1)(v+1)}{(N-1)^2}
          \left(u-i+\frac{N-i}{j+1}v-\frac{N-u}{v+1}j\right)\partial_{l+m},
  \]
and a little calculation shows that this is equal to 
  \[
  \frac{l(v+1)-m(j+1)}{N-1}L_{l+m}.
  \]

\item Suppose next that $lm=0$, so that one of~$l$ or~$m$ is zero --- and
by anti-symmetry we can suppose additionally that $l$ is, so that $j=0$. If
$m\geq1$, then
  \[
    [L_l,L_m]
      = \left[-\frac{1}{N-1}E,-\frac{v+1}{N-1}\partial_m\right] 
      \sim m\frac{v+1}{(N-1)^2}\partial_m
      = -\frac{m}{N-1}L_m,
  \]
if $m=0$, then~$[L_l,L_m]=0$, and if $m\leq-1$, then
  \[
    [L_l,L_m]
      = \left[-\frac{1}{N-1}E,\frac{m}{N-1}\partial_m\right] 
      \sim -\frac{m^2}{(N-1)^2}\partial_m
      = -\frac{m}{N-1}L_m 
  \]
In any case, we see that \eqref{eq:mlmm} holds.

\end{itemize}
At this point we are left with considering the case in which $lm<0$ and,
thanks to anti-symmetry, we can further suppose that in fact $l\leq-1$ and
$m\geq1$.
\begin{itemize}

\item If $l\leq-1$, $m\geq1$ and $l+m\geq1$, then
$j=-1$, $1\leq i+u-1\leq N-1$, and $l+m+1=(i+u-1)+(v-1)(N-1)$, so that
$L_{l+m}=-v\partial_{l+m}/(N-1)$ and
  \[
  [L_l,L_m]
        = \left[\frac{l}{N-1}\partial_l,-\frac{v+1}{N-1}\partial_m\right]
        = -\frac{l(v+1)v}{(N-1)^2}\partial_{l+m}
        = \frac{l(v+1)}{N-1}L_{l+m}.
  \]

\item Suppose now that $l\leq-1$, $m\geq1$ and $l+m=0$. We then have that
$j=-1$, that $0=l+m=(i+u-2)+(v-1)(N-1)$, and that $0\leq i+u-2\leq N-2$, so
that in fact $i=1$, $u=1$ and $v=1$: this tells us that $l=-N+1$ and
$m=N-1$, so that
  \[
  [L_l,L_m]
        = \left[
                \frac{-N+1}{N-1}\partial_{-N+1},
                -\frac{2}{N-1}\partial_{N-1}
          \right]
        = \frac{2}{N-1} E
        = -2L_0
  \]

\item Finally, if $l\leq-1$, $m\geq1$ and $l+m\leq-1$, then
  \[
  [L_l,L_m]
        = \left[
                \frac{l}{N-1}\partial_{l},
                -\frac{v+1}{N-1}\partial_{m}
          \right]
        = \frac{l(v+1)(l+m)}{(N-1)^2}\partial_{l+m}
        = \frac{l(v+1)}{N-1}L_{l+m}.
  \]

\end{itemize}
In all cases what we have found is a specialization of the formula that
appears in~\eqref{eq:mlmm}.
\end{proof}

For each integer~$l$ let us write $\rho(l)$ for the element
of~$\inter{0,N-2}$ that is the remainder of the division of~$l$ by~$N-1$.
Above we have used many times the fact that $l$ determines uniquely
integers~$i$ and~$j$ such that $l+1=i+j(N-1)$: we have $\rho(l)=i-1$.
Clearly, whenever $l$ and $m$ are integers we have that
  \[
  \claim{$\rho(l+m)=\rho(l)+\rho(m)$ if $\rho(l)+\rho(m)\leq
  N-2$.}
  \]
From now on we will, in contexts where this does not lead to confusion,
take the liberty of not making an explicit difference between a derivation
of~$A$ and its class in~$\HH^1(A)$.

Let us recall that the Lie algebra~$\Der(\kk[t])$ of derivations of the
polynomial algebra~$\kk[t]$ is freely spanned as a vector space by the
derivations
  \[
  - t^{j+1}\frac{\d}{\d t}, \qquad j\in\ZZ,\; j\geq-1,
  \]
which are such that
  \[ \label{eq:witt}
  \left[- t^{j+1}\frac{\d}{\d t}, - t^{v+1}\frac{\d}{\d t}\right] 
        = (j-v)L_{j+v}
  \]
whenever $j$ and~$v$ are integers such that $j$,~$v\geq-1$. This Lie
algebra is often called the \newterm{Witt algebra}, although the same name
is more commonly applied to the Lie algebra~$\Der(\kk[t^{\pm1}])$, which is
different --- there is a \emph{third} Lie algebra called the Witt algebra,
but it only occurs in positive characteristic. The Lie algebra
$\Der(\kk[x])$ is simple --- in fact, David Jordan proved in~\cite{Jordan}
that the Lie algebra of derivations on an affine variety is simple exactly
when the variety is smooth, and that applies here (see also Thomas
Siebert's \cite{Siebert}), but one can prove this particular case very
easily «by hand».

\begin{Proposition}\label{prop:nil}
For each $r\in\NN_0$ let $\Nil_r$ be the span of $\{L_l:\rho(l)\geq r\}$
in~$\HH^1(A)$, so that, in particular, $\Nil_r=0$ when $r\geq N-2$ and
there is a descending chain of subspaces
  \[ \label{eq:nils}
        \Nil_0 \supsetneq 
          \underbracket{
              \Nil_1
              \supsetneq \cdots
              \supsetneq \Nil_{N-2}
              \supsetneq \Nil_{N-1} = 0
              }
        .
  \]
\begin{thmlist}

\item For each $r$,~$s\in\inter{0,N-1}$ we have that
$[\Nil_r,\Nil_s]\subseteq\Nil_{r+s}$.

\item For each $r\in\inter{1,N-2}$ we have that
$\Nil_{r+1}\subseteq[L_{-N+2},\Nil_r]\subseteq[\Nil_1,\Nil_r]$.

\item $\Nil_1$ is a nilpotent ideal in~$\HH^1(A)$, the chain underlined
in~\eqref{eq:nils} is its lower central series and, in particular, its
nilpotency index is exactly~$N-2$.

\item The center of~$\HH^1(A)$ is the subspace~$\kk\partial_0$ spanned
by~$\partial_0$, the derived subalgebra of~$\HH^1(A)$ is~$\Nil_0$, which is
a perfect subalgebra, and 
  \[
  \HH^1(A) = Z \oplus \Nil_0.
  \]

\item There is an injective morphism of Lie algebras
$\Phi:\Der(\kk[t])\to\Nil_0$ that for all integers~$j$ such that $j\geq-1$
has
  \[
  \Phi\left(-t^{j+1}\frac{\d}{\d t}\right) = L_{j(N-1)}.
  \]

\item $\Nil_0=\Phi(\Der(\kk[t]))\oplus\Nil_1$ and $\Nil_1$ is the unique
maximal ideal of~$\Nil_0$.

\end{thmlist}
\end{Proposition}

\begin{proof}
\thmitem{1} Let $r$ and~$s$ be elements of~$\inter{0,N-1}$, let $l$ and~$m$
be integers such that $l\geq-N+1$, $m\geq-N+1$, $\rho(l)\geq r$
and~$\rho(m)\geq s$, and let $i$,~$j$,~$u$,~$v$ be the integers such that
$i$,~$u\in\inter{1,N-1}$, $j$,~$v\geq-1$, $l+1=i+j(N-1)$ and
$m+1=u+v(N-1)$, so that $i=\rho(l)+1$ and $u=\rho(m)+1$. If
$\rho(l)+\rho(m)>N-2$, then we know from Corollary~\ref{coro:mlmm} that
$[L_l,L_m]=0\in\Nil_{r+s}$. If instead $\rho(l)+\rho(m)\leq N-2$, then
$\rho(l+m)=\rho(l)+\rho(m)\geq r+s$ and that corollary tells us now that
there is a scalar~$\lambda\in\kk$ such that $[L_l,L_m]=\lambda
L_{l+m}\in\Nil_{r+s}$. This shows that
$[\Nil_r,\Nil_s]\subseteq\Nil_{r+s}$.

\thmitem{2} Let $r$ be an element of~$\inter{1,N-2}$, so that in particular
we have $N\geq3$, and let $m$ be an integer such that $m\geq-N+1$ and
$\rho(m)\geq r+1$. There is an integer $v$ such that $v\geq-1$ and
$m+1=\rho(m)+1+v(N-1)$, and we set $m'\coloneqq\rho(m)-1+(v+1)(N-1)$. Since
$2\leq r+1\leq \rho(m)\leq N-2$, we have that $\rho(m')=\rho(m)-1\geq r$ and,
therefore, that $L_{m'}\in\Nil_{r+1}$. Since $\rho(m)\geq2$ and
$v\geq-1$, we have $m'\geq1$. Since $(-N+2)+m'=m\neq0$ because
$\rho(m)\geq r+1>0$, it follows from Lemma~\ref{lemma:brackets} that
  \[
  [L_{-N+2},L_{m'}]
        = -\frac{N-2}{N-1}(v+2)L_{m}
  \]
and, in any case, that $L_m\in[L_{-N+2},\Nil_r]$. We have
proved that $\Nil_{r+1}\subseteq[L_{-N+2},\Nil_r]$.

\bigskip

Lemma~\ref{lemma:brackets} tells us that the class of the
derivation~$\partial_0$ is central in~$\HH^1(A)$. On the other hand, a
central element of~$\HH^1(A)$ has to commute with $E$, so has
degree~$0$: as $E$ is not central, we see that the center of~$\HH^1(A)$
es precisely~$\kk\partial_0$. Clearly $\HH^1(A)=\kk\partial_0\oplus\Nil_0$
and the derived subalgebra of~$\HH^1(A)$ is~$[\Nil_0,\Nil_0]$.

The map~$\Phi$ described in the lemma is immediately seen to be an
injective morphism of Lie algebras --- thanks to Corollary~\ref{coro:mlmm}
and the formulas in~\eqref{eq:witt} --- and clearly
$\Nil_0=\Phi(\Der(\kk[t]))\oplus\Nil_1$.
As the algebra~$\Der(\kk[x])$ is simple, it is perfect, and thus
$\Phi(\Der(\kk[t]))$ is contained in~$[\Nil_0,\Nil_0]$. As also
$\Nil_1=[M_0,\Nil_1]\subseteq[\Nil_0,\Nil_0]$, we see that $\Nil_0$ is
perfect and equal to the derived subalgebra of~$\HH^1(A)$.

From~\thmitem{1} we have that $[\Nil_0,\Nil_1]\subseteq\Nil_1$, so that the
subspace $\Nil_1$ is an ideal of~$\Nil_0$ and in~$\HH^1(A)$.
From~\thmitem{1} and~\thmitem{2} we see that for each $r\in\inter{1,N-2}$
we have $[\Nil_1,\Nil_r]=\Nil_{r+1}$, so that
$\Nil_1\supsetneq\Nil_2\supsetneq\cdots\supsetneq\Nil_{N-2}$ is the
beginning of the lower central series of~$\Nil_1$. As $\Nil_{N-2}=0$ and
$\Nil_{N-1}\neq0$, we see that $\Nil_1$ is nilpotent of index
exactly~$N-2$. As $\Nil_0/\Nil_1\cong\Der(\kk[t]))$, a simple Lie
algebra, $\Nil_1$ is a maximal ideal in~$\Nil_0$. 

Suppose, finally, that
$I$ is an ideal of~$\Nil_0$: as $I$ is $\ad(E)$-invariant and
$\ad(E):\Nil_0\to\Nil_0$ is diagonalizable and its eigenspaces are the
homogeneous components of~$\Nil_0$, we see that $I$ is a \emph{homogeneous}
ideal of~$\Nil_0$. The intersection $I\cap\Phi(\Der(\kk[t])$ is zero,
because it is an ideal in~$\Phi(\Der(\kk[t]))$, and this tells us that $I$
is generated by homogeneous elements of degree not divisible by~$N-1$ and
thus that it is contained in~$\Nil_1$. This proves that~$\Nil_1$ is the
unique maximal ideal of~$\Nil_0$.
\end{proof}

An immediate consequence of this is that the number~$N$ is a derived
invariant of the algebra~$A$:

\begin{Corollary}\label{coro:n-inv}
If $N$ and $N'$ are two non-negative integers such that the algebras~$A_N$
and~$A_{N'}$ are derived equivalent, then $N=N'$.
\end{Corollary}

\begin{proof}
Let $N$ and~$N'$ be two non-negative integers such that $N\leq N'$ and the
algebras $A_N$ and~$A_{N'}$ are derived equivalent. According to a theorem
of Bernhard Keller in~\cite{Keller} we then have that the Lie
algebras~$\HH^1(A_N)$ and~$\HH^1(A_{N'})$ are isomorphic as Lie algebras.
If $N\geq2$, then Proposition~\ref{prop:nil} tells us that we can compute
the number~$N-2$ from the Lie algebra~$\HH^1(A_N)$ as the nilpotency index
of the unique maximal ideal of the quotient of $\HH^1(A_N)$ by its center
is $N-2$, and we therefore have that $N=N'$.

It is easy to deal with the remaining possibilities. If $N=0$, then $A_N$
is the Weyl algebra and $\HH^1(A)=0$: from
Propositions~\ref{prop:hh1-series} and~\ref{prop:hh1-n1} we see that then
$N'=0$. If $N=1$, then we cannot have $N'>1$, for in that case
$\dim\HH^1(A_N)=1<\infty=\dim\HH^1(A_{N'})$, so again $N=N'$.
\end{proof}

The next thing we want to do is to obtain a computational criterion for a
derivation of~$A$ to be the restriction of an inner derivation of the
algebra~$A_x$ obtained by localizing at~$x$. This is possible because $A_x$
can also be obtained as a localization of a Weyl algebra, with which we can
work with ease. The result we need is the following:

\begin{Proposition}\label{prop:xinner:weyl}
Let $\W$ be the algebra freely generated by letters~$p$ and~$q$ subject to
the relation
  \[
  pq-qp = 1,
  \]
and let $\W_q$ be the localization of~$\W$ at~$q$. The first Hochschild
cohomology space~$\HH^1(\W_q)$ is one-dimensional and is spanned by the
cohomology class of the derivation $\partial:\W_q\to\W_q$ such that
  \[
  \partial(p) = q^{-1}, 
  \qquad
  \partial(q) = 0.
  \]
There is a representation of~$\W_q$ on the algebra~$\kk[t^{\pm1}]$ of
Laurent polynomials such that 
  \[
  p\lact f(t) = \frac{\d}{\d t}f(t),
  \qquad
  q\lact f(t) = tf(t)
  \]
for all~$f\in\kk[t^{\pm1}]$. A derivation $\delta:\W_q\to\W_q$ is inner if
and only if
  \[ \label{eq:res:inn}
  \Res_0\left(\delta(pq)\lact\frac{1}{t}\right) = 0.
  \]
\end{Proposition}

Here $\Res_0:\kk[t^{\pm1}]\to\kk$ is the usual residue map that picks the
coefficient of $t^{-1}$ in its argument. Just as the Weyl algebra~$\W$ is
the algebra of regular differential operators on the affine line~$\AA^1$,
the localization~$\W_q$ is the algebra of regular differential operators on
the punctured affine line~$\AA^1\setminus\{0\}$ --- and this is the
representation~$\lact$ that appears in the proposition. As this is a smooth
affine scheme and our ground field has characteristic zero, we know that
the Hochschild cohomology~$\HH^*(\W_q)$ is isomorphic to the algebraic De
Rham cohomology~$\H^*(\AA^1\setminus\{0\})$. In particular, cohomology
classes of derivations of~$\W_q$ can be viewed as cohomology classes of
$1$-forms on~$\AA^1\setminus\{0\}$: the condition~\eqref{eq:res:inn} above
for a derivation to be inner corresponds to the condition that a $1$-form
«integrate to zero along a curve around the puncture» for it to be a
exact, so it is a version of Poincaré duality in our context. We will prove
the proposition in a purely algebraic way, but geometry and
various comparison maps could be used instead.

\begin{proof}
Let $W$ be the subspace of~$\W$ spanned by~$p$ and~$q$, let $\hat W$ be its
dual space, and let $(\hat p,\hat q)$ be the ordered basis of~$\hat W$ dual
to~$(p,q)$. There is a projective resolution of~$\W$ as a $\W$-bimodule of
the form
  \[
  \begin{tikzcd}
  \W\otimes\Lambda^2W\otimes\W \arrow[r, "d"]
    & \W\otimes W\otimes\W \arrow[r, "d"]
    & \W\otimes\W \arrow[r, dashed, "\epsilon"]
    & \W
  \end{tikzcd}
  \]
with differentials such that
  \begin{align}
  &d(1\otimes 1) = 1, \\
  &d(1\otimes w\otimes 1) = w\otimes 1-1\otimes w, \qquad\forall w\in W, \\
  &d(1\otimes p\wedge w\otimes 1) =
        p\otimes q\otimes1 + 1\otimes p\otimes q
        - q\otimes p\otimes 1 - 1\otimes q\otimes p,
  \end{align}
and augmentation $\epsilon:\W\otimes\W\to\W$ given by the multiplication of
the algebra~$\W$. The localization $\W_q$ of~$\W$ at~$q$ is a $\W$-bimodule
that is flat on both sides, so applying the functor
$\W_q\otimes_\W(\mathord-)\otimes_\W\W_q$ to the resolution above gives a
projective resolution 
  \[
  \begin{tikzcd}
  \W_q\otimes\Lambda^2W\otimes\W_q \arrow[r, "d"]
    & \W_q\otimes W\otimes\W_q \arrow[r, "d"]
    & \W_q\otimes\W_q \arrow[r, dashed, "\epsilon"]
    & \W_q\otimes_\W\W_q
  \end{tikzcd}
  \]
of the $\W_q$-bimodule $\W_q\otimes_\W\W_q$. As the map 
$\W_q\otimes_\W\W_q\to\W_q$ induced by the multiplication of~$\W_q$ is an
isomorphism of~$\W_q$-bimodules, what we have is a projective resolution
of~$\W_q$ as a bimodule over itself.
Applying to it the functor $\hom_{\W_q^e}(\place,\W_q)$ and doing
standard identifications we obtain the complex
  \[ \label{eq:w:comp}
  \begin{tikzcd}
  \W_q \arrow[r, "\delta"]
    & \W_q\otimes W \arrow[r, "\delta"]
        \arrow[l, dashed, shift left=1ex, "s_1"]
    & \W_q\otimes\Lambda^2\hat W 
        \arrow[l, dashed, shift left=1ex, "s_2"]
  \end{tikzcd}
  \]
with differentials such that for all $a$ and~$b$ in~$\W_q$ have
  \[
    \delta(a) 
        = [p,a]\otimes\hat p + [q,a]\otimes\hat q, 
    \qquad
    \delta(a\otimes\hat p+b\otimes\hat q)    
        = \bigl([p,b]-[q,a]\bigr)\otimes\hat p\wedge\hat q.
  \]
Let us write $\X$ for this complex of vector spaces, whose cohomology
is~$\HH^*(\W_q)$, the Hochschild cohomology we are trying to compute, up
to canonical isomorphisms.

There is a $\ZZ$-grading on the algebra~$\W_q$ that assigns to~$p$ and~$q$
degrees~$-1$ and~$1$, respectively. On the other hand, the inner derivation
$\ad(pq):\W_q\to\W_q$ is diagonalizable with spectrum~$\ZZ$, and for each
$n\in\ZZ$ the eigenspace of~$\ad(pq)$ corresponding to the eigenvalue~$n$
is precisely the homogeneous component of~$\W_q$ of degree~$n$ --- this
follows at once from the facts that $\ad(pq)(p)=-p$ and $\ad(pq)=q$, on one
hand, and, on the other, that the set $\{p^iq^j:i\in\NN_0,j\in\ZZ\}$ is a
basis of~$\W_q$.

We consider on the complex~$\X$ the maps~$s_1$ and~$s_2$ indicated with
dashed arrows in the diagram~\eqref{eq:w:comp} above that are given by
  \[
  s_1(a\otimes\hat p+b\otimes\hat q) = aq+pb, 
  \qquad
  s_2(a\otimes\hat p\wedge\hat q) = -pa\otimes\hat p+aq\otimes\hat q
  \]
for all $a$ and~$b$ in~$A$, and define maps
  \begin{align}
  & r_0 \coloneqq s_1\circ\delta 
        : \W_q\to\W_q, \\
  & r_1 \coloneqq \delta\circ s_1 + s_2\circ\delta
        : \W_q\otimes W\to\W_q\otimes W, \\
  & r_2 \coloneqq \delta\circ s_1
        : \W_q\otimes\Lambda^2W \to \W_q\otimes\Lambda^2W.
  \end{align}
Of course, in this way we obtain an endomorphism $r_*:\X\to\X$ of the
complex~$\X$ that is homotopic to zero. A simple calculation shows that
  \begin{align}
  & r_0(a) 
        = \ad(pq)(a),  \\
  & r_1(a\otimes\hat p+b\otimes\hat q)
        = (\ad(pq)(a)-a)\otimes\hat p + (\ad(pq)(b)+b)\otimes\hat q, \\
  & r_2(a\otimes\hat p\wedge\otimes\hat q)
        = \ad(pq)(q)\otimes\hat p\otimes\hat q
  \end{align}
for all choices of~$a$ and~$b$ in~$\W_q$, and it follows from this that the
endomorphism~$r_*$ is diagonalizable, since $\ad(pq)$ is. As a consequence
of all this, the subcomplex~$\X'$ of~$\X$ spanned by all eigenvectors of
the map~$r_*$ corresponding to non-zero eigenvalues is exact and a
complement to~$\ker r_*$, so that 
  \[ \label{eq:xx:desc}
  \X=\X'\oplus\ker r_*.
  \]
In particular, the inclusion $\ker r_*\hookrightarrow\X$ is a quasi-isomorphism.

The kernel~$\ker r_*$ is easily computed to be the complex
  \[ \label{eq:w:comp:red}
  \begin{tikzcd}
  \kk[pq] \arrow[r, "\delta"]
    & \bigl(\kk[pq]q^{-1}\otimes\kk\hat p\bigr)
        \oplus 
      \bigl(\kk[pq]q\otimes\kk\hat q\bigr)
      \arrow[r, "\delta"]
    & \kk[pq]\otimes\kk\hat p\wedge\hat q
  \end{tikzcd}
  \]
with differentials such that
  \begin{align}
  & \delta(p^iq^i) 
        = ip^iq^{i-1}\otimes\hat p - ip^{i-1}q^i\otimes\hat q, \\
  & \delta(p^iq^{i-1}\otimes\hat p)
        = ip^{i-1}q^{i-1}\otimes\hat p\wedge\hat q, \\
  & \delta(p^iq^{i+1}\otimes\hat q)
        = (i+1)p^iq^i\otimes\hat p\wedge\hat q
  \end{align}
for all $i\geq0$. In writing this we have used the fact that the set
$\{p^iq^i:i\in\NN_0\}$ is a basis for the subalgebra~$\kk[pq]$ of~$\W_q$.
A standard calculation now shows that the cohomology of the
complex~\eqref{eq:w:comp:red} is of total dimension~$2$: in degree~$0$ it
is spanned by the cohomology class of~$1\in\kk[pq]$, and in degree~$1$ by
the cohomology class of~$q^{-1}\otimes\hat p$. The class of this last
cocycle thus freely spans $\HH^1(\W_q)$, and translating back we see that
it corresponds to the derivation $\partial:\W_q\to\W_q$ such that
$\partial(p)=q^{-1}$ and $\partial(q)=0$.

If $i\in\NN_0$ and $j\in\ZZ$, then one can easily check that
  \[ \label{eq:res:delta}
  \Res_0\left(p^iq^j\lact\frac{1}{t}\right) 
        = \begin{cases*}
          1 & if $i=j=0$; \\
          0 & in any other case.
          \end{cases*}
  \]
Let $\delta:\W_q\to\W_q$ be a derivation. There is a sequence of
derivations $(\delta_l)_{l\in\ZZ}$, all of which apart from a finite number
are zero, such that $\delta=\sum_{l\in\ZZ}\delta_l$ and for each $l\in\ZZ$
the derivation~$\delta_l:\W_q\to\W_q$ is homogeneous of degree~$l$. Since
$pq$ has degree~$0$, for each $l\in\ZZ\setminus0$ the element
$\delta_l(pq)$ is a linear combination of monomials of the form $p^iq^j$
with $i\in\NN_0$, $j\in\ZZ$ and $j-i\neq0$: in view
of~\eqref{eq:res:delta}, we then have that 
  \[
  \Res_0\left(\delta(pq)\lact\frac{1}{t}\right) 
        = \Res_0\left(\delta_0(pq)\lact\frac{1}{t}\right).
  \]
According to the decomposition~\eqref{eq:xx:desc}, the derivation~$\delta$
is inner if and only if the homogeneous derivation~$\delta_0$ is inner.
Now, if $\delta_0$ is inner, then there is a $u\in\W_q$ of degree~$0$ such
that $\delta_0=\ad(u)$ and, in particular,
$\delta_0(pq)=\ad(u)(pq)=[u,pq]=0$. Putting everything together, we see the
map
  \[ \label{eq:map:res}
  \delta\in\Der(\W_q)\mapsto\Res_0\left(\delta(pq)\lact\frac{1}{t}\right)\in\kk
  \]
vanishes on the subspace~$\InnDer(\W_q)$ of all inner derivations. Our
calculation of~$\HH^1(\W_q)$ implies that
$\Der(A)=\kk\partial\oplus\InnDer(\W_q)$:  as the map~\eqref{eq:map:res}
takes the value~$1$ on the derivation~$\partial$, we conclude that its
kernel is exactly~$\InnDer(\W_q)$.
\end{proof}

Going back to our algebra~$A$ is just a matter of changing coordinates.

\begin{Corollary}
There is a representation of~$A$ on the algebra~$\kk[t^{\pm1}]$ of Laurent
polynomials such that
  \[
  y \lact f(t) = t^N\frac{\d}{\d t}f(t),
  \qquad
  x \lact f(t) = tf(t)
  \]
for all~$f\in\kk[t^{\pm1}]$. If $\delta:A\to A$ is a derivation and
$\tilde\delta:A_x\to A_x$ is its extension to~$A_x$, then the
derivation~$\delta$ is $A_x$-inner if and only if
  \[
  \Res_0\left(\tilde\delta(x^{-N+1}y)\lact \frac{1}{t} \right) = 0.
  \]
\end{Corollary}

\begin{proof}
There is an injective algebra map $\iota:A\to\W$ such that $\iota(x)=q$ and
$\iota(y)=q^Np$ --- this is easy to check and is done in detail in
\cite{BLO:1}*{Section 3} --- and clearly this map extends to one on
localizations $\tilde\iota:A_x\to\W_q$ that is an isomorphism. The
representation of~$A$ on~$\kk[t^{\pm1}]$ that is mentioned in the corollary
is the restriction to~$A$ along~$\tilde\iota$ of the representation
of~$\W_q$ given in Proposition~\ref{prop:xinner:weyl}. Let $\delta:A\to A$
be a derivation of~$A$. It extends uniquely to a derivation
$\tilde\delta:A_x\to A_x$ and, conjugating by~$\tilde\iota$, gives a
derivation on~$\W_q$: the claim of the corollary thus follows immediately 
from the last part of Proposition~\ref{prop:xinner:weyl}.
\end{proof}

\section{Twisted cohomology}
\label{sect:twist}

In Proposition~\ref{prop:finite-subgroups} we described the finite
subgroups of~$\Aut(A)$ or, equivalently, the finite groups that act
faithfully on our algebra~$A$ by algebra automorphisms, and showed that
they are all cyclic. A natural next step is to describe the finite
dimensional Hopf algebras that act on~$A$ faithfully in the appropriate
sense. In that generality this is a problem that we do not know how to
approach, so we will settle for the much smaller problem of describing the
faithful actions of Taft Hopf algebras, motivated by the fact that these
Hopf algebras can be viewed as «quantum thickenings» of cyclic groups and
are therefore very close to them. We will do this in the next section.

We will recall the details about Taft algebras later. For now, let us
mention simply that an action of a Taft Hopf algebra on our algebra~$A$ is
determined by an algebra automorphism~$\phi:A\to A$ of finite order and a
$\phi$-twisted derivation of~$A$. We know the automorphisms very well
already, and our task in the present section will be to obtain a
description of the twisted derivations.

\bigskip

Let $\omega\in\kk$ be different from~$0$ and~$1$, and let
$\phi=\phi_{0,\omega}:A\to A$ be the algebra automorphism such that
$\phi(x)=\omega x$ and $\phi(y)=\omega^{N-1}y$; clearly $\phi$ has finite
order exactly when $\omega$ has finite order, but we do not need to impose
that restriction yet. We write ${}_\phi A$ for the $A$-bimodule which
coincides with $A$ as a right $A$-module and whose left action is such that
$a\lact x=\phi(a)x$ for all $a\in A$ and all $x\in{}_\phi A$. As the
automorphism~$\phi$ is homogeneous, the $A$-bimodule~${}_\phi A$ is a
graded one. 

We are interested in computing the space of \newterm{$\phi$-twisted
derivations}: such a map is a linear function $\partial:A\to A$ such that 
  \[ 
  \partial(ab) = \partial(a)\cdot b+\phi(a)\cdot\partial(b) 
  \] 
for all $a$ and~$b$ in~$A$. It is immediate that a $\phi$-twisted
derivation $A\to A$ is exactly the same thing as a derivation $A\to{}_\phi A$
into the bimodule~${}_\phi A$, and because of this we will write
$\Der(A,{}_\phi A)$ for the space of all $\phi$-twisted derivations and
conflate the two notions.

Now, if $m$ is an element of~$A$, then the map $\ad_\phi(m):a\in A\mapsto
ma-\phi(a)m\in A$ is an element of~$\Der(A,{}_\phi A)$: we say that such a
$\phi$-twisted derivation is an \newterm{inner} one. We can thus split the
calculation of~$\Der(A,{}_\Phi A)$ in two steps: compute the $\phi$-twisted
derivations modulo inner ones, and then glue back the inner ones into the
result. The first step amounts, precisely, to the calculation of the first
Hochschild cohomology space $\H^1(A,{}_\phi A)$, and this is what we will
do --- and just as in Section~\ref{sect:cohomology} we will do this by
first computing the spaces $\H^0(A,{}_\phi A)$ and~$\H^2(A,{}_\phi A)$.

\bigskip

The Hochschild cohomology~$\H^*(A,{}_\phi A)$ can be identified with the
cohomology of the complex~$\hom_{A^e}(P_*,{}_\phi A)$, with $P_*$ the
projective resolution described in Section~\ref{sect:cohomology}, and this
complex, in turn --- just as we built the complex~\eqref{eq:comp} in that
section --- is naturally isomorphic to the complex
  \[ \label{eq:comp:phi}
  \begin{tikzcd}
    & A \arrow[r, "\delta_0"]
    & A\otimes V^* \arrow[r, "\delta_1"]
    & A\otimes \Lambda^2V^* \arrow[r]
    & 0
  \end{tikzcd}
  \]
with differentials such that
  \begin{align}
  \delta_0(a) &= (\omega xa-ax)\otimes\hat x+(\omega^{N-1}ya-ay)\otimes\hat y, 
        \label{eq:d:phi:0} \\
  \delta_1(b\otimes\hat x+c\otimes\hat y) &=
    \left((\omega^{N-1}yb-by) + (cx-\omega xc) - \sum_{s+1+t=N}\omega^sx^sbx^t\right)
    \otimes\hat x\wedge\hat y
        \label{eq:d:phi:1} 
  \end{align}
for all choices of~$a$,~$b$ and~$c$ in~$A$. We remark that we are writing
«$A$» here and in~\eqref{eq:comp:phi} and not~«${}_\phi A$»: we will write
the twisted left action explicitly, as in the formulas above for the
differentials --- we hope this is less confusing than the alternative.

Just as when we dealt with the «untwisted» Hochschild cohomology, we start
our calculation with a lemma that tells us how some specific commutators
---~twisted commutators now\footnote{So we have not broken our promise not
to do any more commutation formulas!}~--- behave. In its statement and in
the rest of this section we will use the convention that if
$\mathit{snark}$ is an object related to~$A$ then we will write
$\widetilde{\smash{\mathit{snark}}}$ the `corresponding' object related to
the localization~$A_x$.

\begin{Lemma}\label{lemma:comm:phi}
The map $\alpha:a\in A\mapsto \omega xa-ax\in A$ is injective and has image
equal to~$xA$, while the map $\tilde\alpha:a\in A_x\mapsto \omega xa-ax\in
A_x$ is bijective.
\end{Lemma}

\begin{proof}
Let $a$ be a non-zero element of~$A$. There are $d\in\NN_0$ and
$a_0$,\dots,~$a_d\in\kk[x]$ such that $a_d\neq0$ and
$a=\sum_{i=0}^da_iy^i$, and we have that
  \[ \label{eq:alpha-a}
  \alpha(a) = \sum_{i=0}^da_i(\omega xy^i-y^ix)
        \equiv (\omega-1)a_dxy^d \mod F_{d-1}.
  \]
Since $\omega\neq1$, it follows immediately from this that
$\alpha(a)\neq0$, so that $\ker\alpha=0$.

Let us now show that $xa$ is in the image of~$\alpha$: we will proceed by
induction with respect to the integer~$d$. If $d=0$, then $a=a_0\in\kk[x]$
and we have that $xa=\alpha(a/(\omega-1))\in\img\alpha$. If instead
$d\geq1$, then we have, according to~\eqref{eq:alpha-a}, that
  \[ \label{eq:alpha-a:2}
  xa - \alpha\left(\frac{a_dy^d}{\omega-1}\right) \in F_{d-1}.
  \]
As $xA=Ax$, the image of~$\alpha$ is certainly contained in~$xA$, and
therefore the difference in~\eqref{eq:alpha-a:2} is both in~$F_{d-1}$ and
in~$xA$: the inductive hypothesis therefore implies that it is in the image
of~$\alpha$, and so is $xa$, as we wanted. With this we have proved the
claims of the lemma about the map~$\alpha$.

Let now $a$ be an element of~$A_x$. There is a non-negative integer
$l\in\NN_0$ such that $x^la\in A$. If $\tilde\alpha(a)=0$, then
$\alpha(x^la)=x^l\tilde\alpha(a)=0$ and what we have already proved implies
that $x^la=0$, so that $a=0$: this shows that the map~$\tilde\alpha$ is
injective. On the other hand, as $x^{l+1}a\in xA$, we already know there is
a $b\in A$ such that $\alpha(b)=x^{l+1}a$ and then
$\tilde\alpha(x^{-l-1}b)=a$: the map~$\tilde\alpha$ is also surjective.
\end{proof}

With this lemma we can now easily determine the twisted cohomology
$\HH^*(A,{}_\phi A)$ in degrees~$0$ and~$2$. Recall that when $N\geq2$ the
homogeneous components of~$A$ are finite-dimensional, so we can work with
Hilbert series in that case.

\begin{Proposition}
We have that $\HH^0(A,{}_\phi A) = 0$. If $N\geq2$, then the Hilbert series of
$\HH^2(A,{}_\phi A)$ is
  \[
  h_{\HH^2(A,{}_\phi A)} = 
        \begin{dcases*}
        \hfil t^{-N} & if $\omega^{N-1}\neq1$; \\[5pt]
        \frac{t^{-N}}{1-t^{N-1}} & if $\omega^{N-1}=1$.
        \end{dcases*}
  \]
If $N=1$, then $\HH^2(A,{}_\phi A)=0$.
\end{Proposition}

In the proof of the corollary we will exhibit concrete
representatives for all cohomology classes in~$\HH^2(A,{}_\phi A)$.

\begin{proof}
Since the map~$\alpha$ of the lemma is injective, it follows 
from the formula~\eqref{eq:d:phi:0} for the differential~$\delta_0$ that
the latter is itself injective and, therefore, that $\HH_0(A,{}_\phi A)=0$.

Let us suppose that $N\geq2$ and that $\omega^{N-1}\neq1$. If
$f\in\kk[y]$ and $v\in A$, then there is $c\in A$ such that
  \[
  \omega xc-cx 
        = \alpha(c) 
        = - (\omega^{N-1}-1)xv - \sum_{s+1+t=N}\omega^sx^sfx^t,
  \]
because the right hand side of this equality is in~$xA$ --- recall that
$xA=Ax$ --- and therefore
  \begin{align}
  \delta_1(f\otimes\hat x+c\otimes y) 
    &= \left(
        (\omega^{N-1}yf-fy) 
        + (cx-\omega xc) 
        - \sum_{s+1+t=N}\omega^sx^sfx^t
        \right)
       \otimes\hat x\wedge\hat y \\
    &= (\omega^{N-1}-1)(fy + xv) \otimes\hat x\wedge\hat y.
  \end{align}
This tells us that 
  \[ \label{eq:d1:phi}
  (\kk[y]y+xA)\otimes\hat x\wedge\hat y\subseteq\img\delta_1.
  \]
The subspace~$\kk[y]y+xA$ that appears here is the bilateral ideal~$I$ of~$A$
generated by~$x$ and~$y$. It is clear from the formula~\eqref{eq:d:phi:1}
and our hypothesis that $N\geq2$ that the image of~$\delta_1$ is contained
in~$I\otimes\hat x\wedge\hat y$, so in~\eqref{eq:d1:phi} we actually have
an equality. It follows from this that $\HH^2(A,{}_\phi A)$ can be
identified with $A/I\otimes\hat x\wedge\hat y$, so it has dimension~$1$,
and is generated by the class of the element $1\otimes\hat x\wedge\hat y$,
which is homogeneous of degree~$-N$. We thus have that
$h_{\HH^2(A,{}_\phi A)} = t^{-N}$.

Next, we suppose that $N\geq2$ and that $\omega^{N-1}=1$. According
to~\eqref{eq:d:phi:1}, if $b$ and~$c$ are in~$A$ we have now
  \[
  \delta_1(b\otimes\hat x+c\otimes\hat y)  =
    \left([y,b] + (cx-\omega xc) - \sum_{s+1+t=N}\omega^sx^sbx^t\right)
    \otimes\hat x\wedge\hat y.
  \]
The three terms in the sum wrapped with big parentheses are clearly in~$xA$,
so 
  \[ \label{eq:d1:phi:2}
  \img\delta_1 \subseteq xA\otimes\hat x\wedge\otimes\hat y.
  \]
As $\delta_1(c\otimes 1)=-\alpha(c)\otimes\hat x\wedge\hat y$ for each $c\in
A$, Lemma~\ref{lemma:comm:phi} implies us that in fact we have an equality
in~\eqref{eq:d1:phi:2} and therefore that $\HH^2(A,{}_\phi A)$ is in this
case isomorphic to $A/xA\otimes\hat x\wedge\hat y$. This is isomorphic as a
graded vector space to~$\kk[y]\otimes\hat x\wedge\hat y$, and therefore in
this case the Hilbert series we are after is
  \[
  h_{\HH^2(A,{}_\phi A)} = \frac{t^{-N}}{1-t^{N-1}}.
  \]

Finally, let us suppose that $N=1$. For all $b$ and $c$ in~$A$ we
have now that
  \[
  \delta_1(b\otimes\hat x+c\otimes\hat y) =
    \Bigl([y,b] + (cx-\omega xc) - b\Bigr)
    \otimes\hat x\wedge\hat y.
  \]
If $u$ is an element of~$A$, then there are $f\in\kk[y]$ and $v\in A$ such
that $u=f+xv$, and there is a $g\in A$ such that $gx-\omega
xg=-\alpha(g)=xv$, so that
  \(
  \delta_1(-f\otimes\hat x+g\otimes\hat y) = u\otimes\hat x\wedge\hat y
  \).
This tells us that the map~$\delta_1$ is in this case surjective, so that
we have $\HH^2(A,{}_\phi A)=0$.
\end{proof}

We can now compute $\HH^1(A,{}_\phi A)$ using the same strategy that
we used for the regular cohomology in Section~\ref{sect:cohomology}.

\begin{Proposition}\label{prop:h1:phi}
If $N\geq2$, then the Hilbert series of $\HH^1(A,{}_\phi A)$ is
  \[
  h_{\HH^1(A,{}_\phi A)}(t)
        = \begin{dcases*}
          0 & if $\omega^{N-1}\neq1$; \\ 
          \frac{1}{t(1-t^{N-1})} & if $\omega^{N-1}=1$.
          \end{dcases*}
  \]
If $N=1$, the $\HH^1(A,{}_\phi A)=0$.
\end{Proposition}

\begin{proof}
Let us suppose first that~$N\geq2$.
The Hilbert series of our algebra~$A$, as we saw before, is
  \(
  h_A(t) = (1-t)^{-1}(1-t^{N-1})^{-1}
  \),
and the Euler characteristic of the complex~\eqref{eq:comp:phi}~is
  \[
  h_A(t) - (t^{-1}+t^{-(N-1)})h_A(t)+t^{-N}h_A(t) = t^{-N},
  \]
so the invariance of the Euler characteristic when passing to cohomology
implies that
  \[
  h_{\HH^0(A,{}_\phi A)}(t)
  -h_{\HH^1(A,{}_\phi A)}(t)
  +h_{\HH^2(A,{}_\phi A)}(t)
  = t^{-N}.
  \]
The Hilbert series~$h_{\HH^1(A,{}_\phi A)}(t)$ can be computed immediately
from this and the preceding proposition, and turns out to be what is
described in the statement.

Let us now suppose that $N=1$ and let $\eta=b\otimes\hat x+c\otimes\hat y$,
with $b$ and~$c$ in~$A$, be a $1$-cocycle in the
complex~\eqref{eq:comp:phi}. If $f\in A$, then 
  \[
  \eta+\delta_0(f) = (b+\omega xf-xf)\otimes\hat x +
  (\text{something})\otimes\hat y,
  \]
and Lemma~\ref{lemma:comm:phi} implies that we can choose~$f$ so that
$b+\omega xf-xf$ is in~$\kk[y]$. This means that, up to replacing~$\eta$ by
a cohomologous $1$-cocycle, we can suppose that $b$ itself is in~$\kk[y]$.
In that case we have, according to~\eqref{eq:d:phi:1}, that
  \[
  0 = \delta_1(\eta) = 
    (cx-\omega xc - b) \otimes\hat x\wedge\hat y.
  \]
Since $cx-\omega xc\in xA$ and $b\in\kk[y]$, this implies that $cx-\omega
xc=b=0$ and the injectivity of the map~$\alpha$ that, in fact, also $c=0$.
This shows that every $1$-cocycle in our complex~\eqref{eq:comp:phi} is
cohomologous to zero and therefore that $\HH^1(A,{}_\phi A)=0$.
\end{proof}

The proposition we have just proved shows that all $\phi$-twisted derivations
$\Der(A,{}_\phi A)$ are in fact inner except in the special case in which
$N\geq2$ and $\omega^{N-1}=1$. In this special case we can make the
following observations. We will work with the localization~$A_x$ of~$A$ at
the powers of the normal element~$x$, in which we have the following
commutation rule:
  \[
  yx^{-1}-x^{-1}y = -x^{N-2}.
  \]
The canonical map $A\to A_x$ is injective, so we view $A$ as a subalgebra
of~$A_x$.
The automorphism $\phi:A\to A$ such that $\phi(x)=\omega x$ and
$\phi(y)=\omega^{N-1}y$ clearly extends to an automorphism
$\tilde\phi:A_x\to A_x$, and every $\phi$-twisted derivation
$\partial:A\to A$ extends uniquely to a $\tilde\phi$-twisted derivation
$\tilde\partial:A_x\to A_x$ that has $\tilde\partial(x^{-1}) =
-\omega^{-1}x^{-1}\partial(x)x^{-1}$.

The following technical result will allow us to construct twisted
derivations.

\begin{Proposition}\label{prop:nin}
Suppose that $N\geq2$ and $\omega^{N-1}=1$.
\begin{thmlist}

\item There is a $\phi$-twisted derivation $\partial_0^\phi:A\to A$ such
that
  \[
  \partial_0^\phi(x) = 1-\omega, 
  \qquad
  \partial_0^\phi(y) = x^{N-2},
  \]
and it is homogeneous of degree~$-1$ and not inner. The inner
$\tilde\phi$-twisted derivation $\ad_{\tilde\phi}(x^{-1}):A_x\to A_x$
of~$A_x$ corresponding to~$x^{-1}$ preserves the subalgebra~$A$ and, in
fact, the map $\partial_0^\phi$ is the restriction
of~$\ad_{\tilde\phi}(x^{-1})$ to~$A$.

\item There is a derivation $\partial_{N-1}:A\to A$ such that
  \[
  \partial_{N-1}(x) = xy, 
  \qquad
  \partial_{N-1}(y) 
    = \hskip-1em\sum_{s+2+t=N}(s+1)x^{s+1}yx^t 
         +  \frac{N-1}{2}
            \left(
              y^{2}
              -N
              x^{N-1}y
            \right).
  \]
This derivation commutes with~$\phi:A\to A$.

\item If $\partial:A\to{}_\phi A$ is a twisted derivation that is
homogeneous of some degree~$l$, then the map
  \[
  \partial_{N-1}\circ\partial-\partial\circ\partial_{N-1}:A\to A
  \]
is also a $\phi$-twisted derivation that is homogeneous, now of
degree~$l+N-1$. There is therefore a map  
  \[
  \ad(\partial_{N-1}):
        \partial
        \in\Der(A,{}_\phi A)
        \longmapsto
        \partial_{N-1}\circ\partial-\partial\circ\partial_{N-1}
        \in\Der(A,{}_\phi A)
  \]
that is homogeneous of degree~$N-1$.

\item If $u$ is an element of~$A_x$ such that the inner
$\tilde\phi$-derivation $\ad_{\tilde\phi}(u):A_x\to A_x$ preserves the
subalgebra~$A$, so that the restriction
$\ad_{\tilde\phi}(u)|_A:A\to A$ is an element of~$\Der(A,{}_\phi A)$, then
the map $\ad_{\tilde\phi}(\tilde\partial_{N-1}(u)):A_x\to A_x$ also
preserves~$A$ and
  \[ \label{eq:adad}
  \ad(\partial_{N-1})\Bigl(\ad_{\tilde\phi}(u)|_A\Bigr)
        = \ad_{\tilde\phi}(\tilde\partial_{N-1}(u))|_A.
  \]
Here $\tilde\partial_{N-1}(u)$ denotes the value at~$u$ of the
extension~$\tilde\partial_{N-1}$ of the derivation~$\partial_{N-1}$
of~\thmitem{2} to~$A_x$.

\end{thmlist}
\end{Proposition}

\begin{proof}
\thmitem{1} Let $\kk\gen{X,Y}$ be the free algebra on two generators~$X$
and~$Y$. There is an automorphism $\hat\phi:\kk\gen{X,Y}\to\kk\gen{X,Y}$
such that $\hat\phi(X)=\omega X$ and
$\hat\phi(Y)=Y$, and it preserves the bilateral ideal
$I=(YX-XY-X^N)$, since $\omega^{N-1}=1$. The quotient $\kk\gen{X,Y}/I$ is
of course isomorphic to~$A$, and we may thus view $A$ as a
$\kk\gen{X,Y}$-bimodule. Since the algebra $\kk\gen{X,Y}$ is free, there is
a unique derivation $\hat\partial:\kk\gen{X,Y}\to{}_{\hat\phi} A$ such that
$\hat\partial(X)=1-\omega$ and $\hat\partial(Y)=x^{N-2}$. This derivation
vanishes on the ideal~$I$: to check this, it is sufficient to show that it
vanishes on the generator $YX-XY-X^N$ of that ideal, and that is easy. One
need only keep in mind that $\hat\partial$ is a \emph{$\hat\phi$-twisted}
derivation, so that
  \[
  \hat\partial(X^N)
        = (1+\omega+\cdots+\omega^{N-1})(1-\omega)x^{N-1}
        = (1-\omega)x^{N-1}.
  \]
It follows from this that $\hat\partial$ induces a derivation
$\partial_0^\phi:A\to{}_\phi A$ such that
$\partial_0^\phi(x)=1-\omega$ and $\partial_0^\phi(y)=x^{N-2}$. It
is manifestly homogeneous of degree~$-1$, and it is not inner because there
are no elements of negative degree in~$A$.

Now, we certainly have an inner $\tilde\phi$-twisted derivation
$\ad_{\tilde\phi}(x^{-1}):A_x\to A_x$, and we can compute that
  \begin{align}
  \ad_{\tilde\phi}(x^{-1})(x)
       &= x^{-1}x-\tilde\phi(x)x^{-1}
        = 1-\omega, \\
  \ad_{\tilde\phi}(x^{-1})(y)
       &= x^{-1}y-\tilde\phi(y)x^{-1}
        = x^{N-2}.
  \end{align}
Since $\tilde\phi(A)=A$, the facts that $x$ and~$y$ generate~$A$ and
that $\ad_{\tilde\phi}(x^{-1})$ is a $\tilde\phi$-twisted derivation imply
together that $\ad_{\tilde\phi}(x^{-1})$ maps the subalgebra~$A$ into
itself. Moreover, its restriction to~$A$ is a $\phi$-twisted derivation
that coincides with $\partial^\phi_0$ at~$x$ and at~$y$, so
we have that $\ad_{\tilde\phi}(x^{-1})|_A=\partial^\phi_0$.

\thmitem{2} The derivation~$\partial_{N-1}$ described here is simply the
one given by Proposition~\ref{prop:hh1-high} when $l=N-1$, which we know to be
homogeneous of degree~$N-1$ and not inner. To make this explicit we have
used the fact that $\Phi_2 = \frac{1}{2}(y^{2} -N x^{N-1}y)$. A simple
direct calculation shows that $\partial_{N-1}$ commutes with~$\phi$.

\thmitem{3} This follows from an easy direct calculation. They key fact
that makes things work is that the derivation $\partial_{N-1}:A\to A$
commutes with the automorphism $\phi:A\to A$.

\thmitem{4} A direct calculation shows that we have an equality of
maps~$A_x\to A_x$
  \[
  \ad(\tilde\partial_{N-1})\Bigl(\ad_{\tilde\phi}(u)\Bigr)
        = \ad_{\tilde\phi}\Bigl(\tilde\partial_{N-1}(u)\Bigr)
  \]
The left hand side preserves~$A$ because $\ad_{\tilde\phi}(u)$
and~$\tilde\partial_{N-1}$ do, so the right hand side also preserves it.
The equality~\eqref{eq:adad} that appears in the statement of the
proposition is now immediate.
\end{proof}

Combining the different parts of the proposition we have just proved we can
construct $\phi$-twisted derivations $A\to{}_\phi A$ of all degrees of the
form $-1+l(N-1)$ with $l\in\NN_0$.

\begin{Corollary}\label{coro:partial:phi}
Suppose that $N\geq2$ and $\omega^{N-1}=1$, and let $\phi:A\to A$ be the
automorphism such that $\phi(x)=\omega x$ and $\phi(y)=y$.
Let $\tilde\phi:A_x\to A_x$ and $\tilde\partial_{N-1}:A_x\to A_x$ be the
extensions of~$\phi$ and of the $\phi$-twisted
derivation~$\partial_{N-1}:A\to A$ of Proposition~\ref{prop:nin}. 
For each $l\in\NN_0$ the $\tilde\phi$-twisted derivation
  \[
  \ad_{\tilde\phi}\bigl(\tilde\partial_{N-1}^l(x^{-1})\bigr):A_x\to A_x
  \]
preserves~$A$ and its restriction to~$A$ is a map 
  \[
  \partial^\phi_l:A\to A
  \]
that is a homogeneous $\phi$-twisted derivation of degree~$-1+l(N-1)$ and
that coincides with the map
  \[
  \ad(\partial_{N-1})^l(\partial_0^\phi) \in \Der(A,{}_\phi A).
  \]
\end{Corollary}

\begin{proof}
This can be proved immediately by induction starting with~\thmitem{1} of
the proposition and using part~\thmitem{4} for the inductive step.
\end{proof}

\begin{Remark}\label{rem:general-nonsense}
There is a general-nonsense way of looking at our last few results. We will
explain it, as it helps to make sense of them.

Suppose that $N\geq2$ and that $\omega^{N-1}=1$, and let $\nu$ be the order
of~$\omega$ in~$\kk^\times$, which is a divisor of~$N-1$. General
considerations imply that the automorphism~$\phi$ acts in a canonical way
on the Hochschild cohomology~$\HH^*(A)$ of~$A$, the invariant
subspace~$\HH^*(A)^\phi$ is a sub-Gerstenhaber algebra of~$\HH^*(A)$, and
there is a «Gerstenhaber module structure»
  \[
  \HH^*(A)^\phi \otimes \HH^*(A,{}_\phi A)
  \to
  \HH^*(A,{}_\phi A)
  \]
that in particular restricts to a Lie module structure
  \[ \label{eq:hhhh}
  \HH^1(A)^\phi \otimes \HH^1(A,{}_\phi A)
  \to
  \HH^1(A,{}_\phi A),
  \]
which itself is induced by the map
  \[ \label{eq:derder}
  f\otimes g\in\Der(A)\otimes\Der(A,{}_\phi A)
  \mapsto
  \ad(f)(g) \in\Der(A,{}_\phi A),
  \]
using the notation of part~\thmitem{3} of Proposition~\ref{prop:nin}. 

If we view~$\HH^1(A)$ as $\Der(A)/\InnDer(A)$ then the action of~$\phi$
on~$\HH^1(A)$ is induced by the map  
  \[
  \phi^\sharp:d\in\Der(A)\mapsto\phi\circ d\circ\phi^{-1}\in\Der(A),
  \]
and a direct calculation proves that if $d:A\to A$ is a homogeneous
derivation of degree~$l$, then $\phi^\sharp(d)=\omega^l d$. It follows from
this that the invariant subspace~$\HH^1(A)^\phi$ is spanned by the classes
of the derivations of the form~$L_{j\nu}$, with~$j\in\ZZ$ such that
$j\geq-1$, and $\partial_0$, using the notation of
Section~\ref{sect:explicit}. On the other hand,
Proposition~\ref{prop:h1:phi} tells us that the only homogeneous components
of~$\HH^1(A,{}_\phi A)$ that are non-zero are those of degrees~$l$ such
that $l\geq-1$ and $l\equiv-1\mod N-1$: as the action~\eqref{eq:hhhh} is
homogeneous, this implies that the only $L_l$ in~$\HH^1(A)^\phi$ that can
act non-trivially on~$\HH^1(A,{}_\phi A)$ are those with those with
$l\equiv0\mod N-1$. What Corollary~\ref{coro:partial:phi} does is,
essentially, to describe the action of~$L_{N-1}$, the simplest of those:
in part~\thmitem{1} of Proposition~\ref{prop:nin} we exhibited an
element~$\partial_0^\phi$ in~$\Der(A,{}_\phi A)$, and in
Corollary~\ref{coro:partial:phi} we produced many others by using the
action~\eqref{eq:derder}. We will see below that in this way we 
obtain a basis for~$\HH^1(A,{}_\phi A)$.
\end{Remark}

We are just a hop, skip and jump away from the description of
$\H^1(A,{}_\phi A)$ in the exceptional case in which it is non-zero ---
apart, alas, for one exception.

\begin{Proposition}\label{prop:h1:phi:base}
Suppose that $N\geq4$ and that $\omega^{N-1}=1$, and let $\phi:A\to A$ be
the automorphism of~$A$ such that $\phi(x)=\omega x$ and $\phi(y)=y$. The
cohomology classes of the $\phi$-twisted derivations $\partial^\phi_l$ with
$l\in\NN_0$ that we described in Corollary~\ref{coro:partial:phi} freely
span the vector space~$\H^1(A,{}_\phi A)$.
\end{Proposition}

This excludes the cases in which $2\leq N\leq 3$. Now if $N=2$ then the
hypothesis on~$\omega$ is that $\omega=1$ and, since we are assuming
throughout that $\omega\neq1$, this case cannot actually occur. The case in
which $N=3$, on the other hand, is a real possibility that is excluded in
this proposition --- and for good reason, as we will see below.

\begin{proof}
It is enough that we show that the cohomology classes of those
$\phi$-twisted derivations are linearly independent in~$\H^1(A,{}_\phi A)$,
because if that is the case then they span the whole space: this follows
immediately from our calculation of the Hilbert series of~$\H^1(A,{}_\phi
A)$ in Proposition~\ref{prop:h1:phi} and the fact that for all $l\in\NN_0$
the twisted derivation $\partial^\phi_l$ has degree~$-1+l(N-1)$.

Let us suppose that the cohomology classes of the twisted derivations in
the statement are linearly dependent. There are then
an integer $d\geq0$ and scalars $a_0$,~\dots,~$a_d$ in~$\kk$
such that $a_d\neq0$ and the $\phi$-twisted derivation 
  \[
  \delta\coloneqq\sum_{i=0}^da_i\partial_i^\phi:A\to A
  \]
is inner, so that there is an element~$u$ in~$A$ such that
$\delta=\ad_\phi(u)$. We let $P$ be the polynomial
$\sum_{i=0}^da_iT\in\kk[T]$ and consider the $\tilde\phi$-twisted
derivation
  \[ \label{eq:urgh}
  \ad_{\tilde\phi}\bigl(P(\tilde\partial_{N-1})(x^{-1})\bigr):A_x\to A_x
  \]
It follows from Corollary~\ref{coro:partial:phi} that this map
preserves~$A$ and its restriction to~$A$ is precisely~$\delta$. In
particular, the map in~\eqref{eq:urgh} and~$\ad_{\tilde\phi}(u):A_x\to A_x$
take the same value on~$x$ and~$y$ and therefore, since they are
$\tilde\phi$-twisted derivations, they are in fact equal --- in other
words, we have that
  \[ \label{eq:urgh:1}
  \ad_{\tilde\phi}\Bigl(P(\tilde\partial_{N-1})(x^{-1})-u\Bigr) = 0.
  \]
We know form Lemma~\ref{lemma:comm:phi} that the map $\tilde\alpha:a\in
A_x\mapsto\omega xa-xa\in A_x$ is injective, and as a consequence of that
  \[ \label{eq:urgh:2}
  \claim{if $v\in A_x$ and the map $\ad_{\tilde\phi}(v):A_x\to A_x$
  vanishes at~$x$, then $v=0$.}
  \]
From this and the equality~\eqref{eq:urgh:1} we can conclude that
  \[
  P(\tilde\partial_{N-1})(x^{-1}) = u \in A.
  \]

Let us next check that for all $l\in\NN_0$ we have that
  \[ \label{eq:urgh:3}
  \tilde\partial_{N-1}^l(x^{-1})
        \equiv \prod_{i=0}^{l-1}\left(\frac{t(N-1)}{2}-1\right)\cdot x^{-1}y^{l} 
        \mod \tilde F_{l-1}.
  \]
This is clear when $l=0$, and when $l=1$ we have that
  \[
  \tilde\partial_{N-1}(x^{-1})
        = -x^{-1}\tilde\partial_{N-1}(x)x^{-1}
        = -yx^{-1}
        \equiv -x^{-1}y\mod \tilde F_0,
  \]
in accordance with the formula we want to prove. The general case follows
by induction and the observation that for all $l\in\NN_0$ we have
  \begin{align}
  \MoveEqLeft
  \tilde\partial_{N-1}(x^{-1}y^l)
        = \tilde\partial_{N-1}(x^{-1})y^l + x^{-1}\tilde\partial_{N-1}(y^l) 
        = -yx^{-1}y^l + x^{-1}\sum_{i=0}^{l-1}y^i\tilde\partial_{N-1}(y)y^{l-1-i} \\
       &= -yx^{-1}y^l + x^{-1}\sum_{i=0}^{l-1}y^i
                \left[
                \sum_{s+2+t=N}(s+1)x^{s+1}yx^t 
                         +  \frac{N-1}{2}
                            \left(
                              y^{2}
                              -N
                              x^{N-1}y
                            \right)
               \right]
               y^{l-1-i}
          \\
       &\equiv \left(\frac{l(N-1)}{2}-1\right)
                x^{-1}
               y^{l+1} \mod \tilde F_l.
  \end{align}

Now, using the congruence~\eqref{eq:urgh:3} we have just proved we see that
  \[
  P(\tilde\partial_{N-1})(x^{-1})
    \equiv \underbrace{a_d \prod_{i=0}^{d-1}\left(\frac{t(N-1)}{2}-1\right)}
        {}\cdot x^{-1}y^{d} 
    \mod \tilde F_{d-1}
  \]
As $N\geq4$, the scalar marked with a brace is non-zero. If we call it
$\gamma$ for brevity, then what we have is that 
  \[
  A \ni P(\tilde\partial_{N-1})(x^{-1}) = \gamma\cdot x^{-1}y^d + w
  \]
for some $w\in\tilde F_d$, and this is absurd. This contradiction proves
the proposition.
\end{proof}

When $N=3$, in the situation of Proposition~\ref{prop:h1:phi:base} we can
compute that 
  \[
  \tilde\partial_2(x)=xy, 
  \qquad
  \tilde\partial_2(x^{-1})=-x^{-1}y+x,
  \qquad
  \tilde\partial_2(y)=x^4+y^2
  \]
and using that that
  \[
  \tilde\partial_2^2(x^{-1}) = -x^3 \in A.
  \]
This tells us that the $\phi$-twisted derivation $\partial^\phi_2:A\to A$
is inner, and implies that in fact the $\phi$-twisted derivation
$\partial^\phi_l:A\to A$ is inner whenever $l$ is an integer such that
$l\geq2$. The conclusion of Proposition~\ref{prop:h1:phi:base} is therefore
very false when $N=3$. We can fix it as follows:

\begin{Proposition}
Suppose that $N=3$ and that $\omega^{N-1}=1$, so that $\omega=-1$, and let
$\phi:A\to A$ be the automorphism of~$A$ such that $\phi(x)=-x$ and
$\phi(y)=y$. The $\tilde\phi$-twisted
derivations of~$A_x$  
  \[
  \ad_{\tilde\phi}(x^{-1}),
  \qquad\quad
  \ad_{\tilde\phi}(\tilde\partial_2(x^{-1})),
  \qquad\quad
  \ad_{\tilde\phi}(\tilde\partial_2^l(x^{-1}y^2)), \quad l\geq0
  \]
preserve the subalgebra~$A$ and the cohomology classes of their
restrictions to~$A$ freely span the vector space~$\H^1(A,{}_\phi A)$.
\end{Proposition}

\begin{proof}
This can be proved in essentially the same way as the previous proposition,
and we therefore omit the details.
\end{proof}

With this we have concluded the calculation of the twisted Hochschild
cohomology $\H^*(A,{}_\phi A)$. One nice observation we can immediately
make with the result at hand is:

\begin{Proposition}\label{prop:der:xinner}
Suppose that $N\geq1$ and that $\omega\in\kk^\times\setminus\{1\}$, and let
$\phi:A\to A$ be the automorphism of~$A$ such that $\phi(x)=\omega x$ and
$\phi(y)=\omega^{N-1}y$. Every $\phi$-twisted derivation $A\to{}_\phi A$ is
the restriction to~$A$ of an inner $\tilde\phi$-twisted derivation
$A_x\to{}_{\tilde\phi} A_x$ that preserves~$A$. \qed
\end{Proposition}

By analogy with the notion of X-inner automorphisms of Har\v{c}enko
\citelist{\cite{Harchenko:1}\cite{Harchenko:2}}, we can rephrase this
proposition saying that all $\phi$-twisted derivations of~$A$ are X-inner.

\bigskip

When $N=0$, so that the algebra $A$ is the first Weyl
algebra, the calculation of the twisted Hochschild cohomology
$\H^*(A,{}_\phi A)$ for an automorphism $\phi:A\to A$ such that
$\phi(x)=\omega x$ and $\phi(y)=\omega^{-1}y$ for some
$\omega\in\kk^\times$ was done in~\cite{AFLS} by Jacques Alev, Thierry
Lambre, Marco Farinati and Andrea Solotar: the end result is that
  \[
  \dim\H^p(A,{}_\phi A) \cong 
    \begin{cases*}
    1 & if $p=0$ and $\omega=1$, or if $p=2$ and $\omega\neq1$; \\
    0 & in any other case.
    \end{cases*}
  \]
This is rather different from what we have found above for $N\geq1$. It
should be noticed that while the automorphism group of the first Weyl
algebra is considerably larger than that of our algebras $A$ with $N\geq1$,
but every one of its elements of finite order is conjugated to one of
the form of the automorphism~$\phi$ described here. On the other hand,
there are non-cyclic finite groups of automorphisms of the first Weyl
algebra.

\section{Actions of Taft algebras}
\label{sect:taft}

Let now $n$,~$m\in\NN$ be such that $1<m$ and $m\mid n$, and let
$\lambda\in\kk^\times$ be a primitive $m$th root of unity. The
\newterm{generalized Taft algebra} $T=T_n(\lambda,m)$ is the algebra
freely generated by two letters $g$ and~$\xi$ subject to the relations
  \[ \label{eq:T:rels}
  g^n = 1,
  \qquad
  \xi^m = 0,
  \qquad
  g\xi = \lambda\xi g.
  \]
This algebra was originally studied by David Radford in~\cite{Radford}.
It is finite dimensional of dimension~$nm$, and the set $\{\xi^ig^j:0\leq
i<m,0\leq j<n\}$ is one of its bases. It is a Hopf algebra with
comultiplication $\Delta:T\to T\otimes T$ and augmentation
$\epsilon:T\to\kk$ such that
  \[
  \Delta(g) = g\otimes g,
  \qquad
  \Delta(\xi) = \xi\otimes1+g\otimes\xi,
  \qquad
  \epsilon(g) = 1,
  \qquad
  \epsilon(\xi) = 0.
  \]
When $n=m$ and $\theta=0$ this is the classical Taft Hopf algebra
constructed by Earl Taft in~\cite{Taft} with the purpose of exhibiting
finite dimensional Hopf algebras with antipode of arbitrarily large order.
We will use freely the Heyneman--Sweedler notation for the coproduct of~$T$
and even omit the sum: we will write $\Delta(a)$ in the form $a_1\otimes
a_2$.

\bigskip

We are interested in left $T$-module-algebra structures on our algebra~$A$,
that is, left $T$\nobreakdash-module structures $\lact:T\otimes A\to A$
such that the multiplication and unit map $A\otimes A\to A$ and $\kk\to A$
are both $T$-linear --- this means that $h\lact 1_A=\epsilon(h)1_A$ and
$h\lact ab =(h_1\lact a)(h_2\lact b)$ for all choices of~$a$ and~$b$
in~$A$ and $h$ in~$T$. We refer the reader to \cite{Montgomery}*{Chapter 4}
for general information about module-algebras over Hopf algebras.

We will further restrict our attention to $T$-module algebra structures
on~$A$ that are \newterm{inner-faithful} --- a notion introduced by
T.\,Banica and J.\,Bichon in~\cite{BB} ---  as these correspond to faithful
group actions on~$A$: the condition is that there be no non-zero Hopf
ideal~$I$ in~$T$ such that $I\lact A=0$. According to
\cite{Cline}*{Corollary 3.7}, in the specific case in which the Hopf
algebra is our generalized Taft algebra $T$ we have a handy criterion:
  \[ \label{eq:faithful}
  \claim[.7]{a $T$-module-algebra structure on~$A$ is inner-faithful if and
  only if the group $G(T)=\gen{g}$ of group-like elements of~$T$ acts
  faithfully on~$A$ and $\xi\lact A\neq0$.}
  \]

As the algebra~$T$ is generated by~$g$ and~$\xi$ subject to the
relations in~\eqref{eq:T:rels}, giving a $T$-module structure on~$A$ is the
same as giving the two maps $\phi:a\in A\mapsto g\lact a\in A$ and
$\partial:a\in A\mapsto\xi\lact a\in A$ such that $\phi^n=\id_A$,
$\partial^m=0$ and $\phi\partial=\lambda\partial\phi$, and
that structure will be a $T$-module-algebra structure exactly when $\phi$ is
an automorphism of~$A$ and $\partial$ a $\phi$-twisted derivation $A\to A$.
The automorphism~$\phi$ will have finite order: according to
Proposition~\ref{prop:finite-subgroups}, up to conjugating the whole
module-algebra structure by an algebra automorphism of~$A$ we can suppose
then that there is a scalar~$\omega\in\kk$ such that $\phi(x)=\omega x$ and
$\phi(y)=\omega^{N-1} y$, and then the group of group-like
elements~$G(T)=\gen{g}$ of~$T$, which is cyclic of order~$n$, will clearly
act faithfully on~$A$ if and only if the scalar~$\omega$ is a primitive
$n$th root of unity.

We are left with the task of understanding the possibilities for the
map~$\partial$. Since it is a $\phi$-twisted derivation $A\to A$, we now
from Proposition~\ref{prop:der:xinner} that the map~$\partial$ will be, in
fact, the restriction of an inner $\tilde\phi$-twisted derivation of the
localization~$A_x$ that preserves~$A$. The following lemma imposes
significant restrictions on what can actually happen.

\begin{Lemma}\label{lemma:partial:inner}
Let $\omega$ be a primitive $n$th root of unity in~$\kk$, let $\phi:A\to A$
be the algebra automorphism such that $\phi(x)=\omega x$ and
$\phi(y)=\omega^{N-1}y$, and let $\tilde\phi:A_x\to A_x$ be the unique
extension of~$\phi$ to the localization~$A_x$. Let $u$ be a non-zero
element of~$A_x$ such that the inner $\tilde\phi$-twisted derivation
$\ad_{\tilde\phi}(u):A_x\to A_x$ preserves the subalgebra~$A$ and let
  \[
  \partial\coloneqq\ad_{\tilde\phi}(u)|_A:A\to A
  \]
be its restriction to~$A$, which is a $\phi$-twisted derivation.
\begin{thmlist}

\item The map~$\partial$ is non-zero.

\item If $\lambda$ is a scalar, then we have that
$\phi\partial=\lambda\partial\phi$ exactly when $\tilde\phi(u)=\lambda u$,
and when that is the case we have that $\lambda^n=1$.

\item If there are a scalar~$\lambda$ and a positive integer~$m$
such that $\tilde\phi(u)=\lambda u$ and $\partial^m=0$, then $n\leq m$.

\end{thmlist}
\end{Lemma}

\begin{proof}
\thmitem{1} We have $\partial(x)=ux-\omega xu=-\tilde\alpha(u)\neq0$,
according to Lemma~\ref{lemma:comm:phi}.

\thmitem{2} If $\lambda\in\kk$ is such that
$\phi\partial=\lambda\partial\phi$, then
  \begin{align}
  \tilde\phi(u)\omega x - \omega^2 x\tilde\phi(u)
       &= \tilde\phi(ux-\omega xu) 
        = \phi(ux-\omega xu) 
        = \phi(\partial(x))
        = \lambda\partial(\phi(x)) \\
       &= \lambda\partial(\omega x) 
        = \lambda(u\omega x-\omega^2xu),
  \end{align}
and therefore
  \[
  \tilde\alpha\bigl(\tilde\phi(u)-\lambda u\bigr) 
        = \omega x\bigl(\tilde\phi(u)-\lambda u\bigr)
          - \bigl(\tilde\phi(u)-\lambda u\bigr)x 
        = 0,
  \]
so that $\tilde\phi(u)=\lambda u$ because of Lemma~\ref{lemma:comm:phi}.
Conversely, if $\lambda$ is a scalar such that $\tilde\phi(u)=\lambda u$,
then a simple and direct calculation shows that we have that
$\phi\partial=\lambda\partial\phi$. Moreover, since the
automorphism~$\tilde\phi$ is manifestly diagonalizable and all its
eigenvalues are powers of~$\omega$, in that case we have that
$\lambda^n=1$.

\thmitem{3} Let us suppose that there are $\lambda\in\kk$ and $m\in\NN$
such that $\tilde\phi(u)=\lambda u$ and $\partial^m=0$. Since $u\neq0$,
there is an integer~$d\in\NN_0$ such that $u\in\tilde F_d\setminus\tilde
F_{d-1}$. An induction shows that if $l\in\NN_0$, $a\in A$ and $t\in\kk$
are such that $a\in F_l\setminus F_{l-1}$ and $\phi(a)=ta$, then for all
$k\in\NN_0$ we have that
  \[
  \partial^k(a) \equiv \prod_{i=0}^{k-1}(1-t\lambda^i)\cdot au^k
        \mod \tilde F_{l+kd-1}
  \]
and, in particular, since $\partial^m(a)=0$, that
  \[
  \prod_{i=0}^{m-1}(1-t\lambda^i)\cdot au^m \in \tilde F_{l+md-1}.
  \]
As the graded algebra~$\gr A_x$ for the filtration~$(\tilde F_i)_{i\geq-1}$
is an integral domain, this tells us that, in fact,
  \[
  \prod_{i=0}^{m-1}(1-t\lambda^i) = 0
  \]
and, then, that $t\in\{\lambda^{-i}:0\leq i<m\}$. In particular, for each
$j\in\{0,\dots,n-1\}$ we can take $a=x^j$, that has $\phi(a)=\omega^j a$,
and conclude that the $n$ pairwise different scalars
$\omega^0$,~\dots,~$\omega^{n-1}$ all belong to the set
$\{\lambda^{-i}:0\leq i<m\}$. This set has cardinal at most $m$, so
$n\leq m$, as the lemma claims.
\end{proof}

The obvious question after having proved Lemma~\ref{lemma:partial:inner}
is: when is the map~$\partial$ appearing there such that $\partial^n=0$?
The answer is simple: never. With the objective of proving this, let us
recall some standard notations from the theory of $q$-variants. If $q$ is a
variable, then for each $n\in\NN_0$ we let $[n]_q\coloneqq
1+q+\cdots+q^{n-1}$ and $[n]_q!\coloneqq[1]_q[1]_q\cdots[n]_q$, and
consider for each choice of $k$ and~$i$ in~$\NN_0$ such that $0\leq i\leq
k$ the \newterm{$q$-binomial} or \newterm{Gaussian binomial coefficient}
  \[
  \binom{k}{i}_q \coloneqq \frac{[k]_q!}{[i]_q!\cdot[k-i]_q!},
  \]
which is an element of~$\ZZ[q]$.

\begin{Lemma}\label{lemma:non-zero}
Let $\omega$ and~$\lambda$ be primitive $n$th roots of unity in~$\kk$, let
$\phi:A\to A$ be the algebra automorphism such that $\phi(x)=\omega x$ and
$\phi(y)=\omega^{N-1}y$, and let $\tilde\phi:A_x\to A_x$ be the unique
extension of~$\phi$ to the localization~$A_x$. If $u$ is a non-zero
element of~$A_x$ such that $\tilde\phi(u)=\lambda u$ and the inner
$\tilde\phi$-twisted derivation $\ad_{\tilde\phi}(u):A_x\to A_x$ preserves
the subalgebra~$A$, then the restriction
  \(
  \partial\coloneqq\ad_{\tilde\phi}(u)|_A:A\to A
  \),
which is a $\phi$-twisted derivation, has $\partial^n\neq0$.
\end{Lemma}

\begin{proof}
Let $u$ be a non-zero element of~$A_x$ such that $\tilde\phi(u)=\lambda u$
and $\ad_{\tilde\phi}(u)(A)\subseteq A$, and let us consider the
$\phi$-twisted derivation $\partial\coloneqq\ad_{\tilde\phi}(u)|_A:A\to A$.

Let $a$ in~$A$ and $t\in\kk$ be such that $\phi(a)=ta$. For all $k\in\NN_0$
we have that
  \[ \label{eq:pka}
  \partial^k(a) 
        = \sum_{i=0}^k (-1)^{k-i}
                       \lambda^{\binom{i}{2}}
                       \binom{k}{i}_\lambda 
                       t^i u^{k-i}au^i,
  \]
as can be proved by an obvious induction using the well-known analogue of
Pascal's identity,
  \[
  \binom{k}{i}_q = \binom{k-1}{i}_q + q^{k-i}\binom{k-1}{i-1}_q,
  \]
valid in~$\ZZ[q]$ whenever $0<i<k$. In particular, taking $k=n$
in~\eqref{eq:pka} we find that
  \[ \label{eq:pka:1}
  \partial^n(a) 
        = \sum_{i=0}^n 
                (-1)^{n-i}
                \lambda^{\binom{i}{2}}
                \binom{n}{i}_\lambda 
                t^i u^{n-i}au^i.
  \]
Now, according to the Cauchy $q$-binomial theorem we have that
  \[
  \prod_{i=0}^n (y+zq^i) = \sum_{i=0}^nq^{\binom{i}{2}}\binom{n}{i}_qy^{n-i}z^i
  \]
in the polynomial ring~$\ZZ[q,y,z]$. Extending scalars to~$\kk$ and
specializing $z$ at~$-1$ and $q$ at~$\lambda$, the left hand side of this
equality becomes $y^n-1$, so by looking at the coefficients of the powers
of~$y$ in the right hand side we see that $\binom{n}{i}_\lambda=0$ when
$0<i<n$. The equality~\eqref{eq:pka:1} therefore violently simplifies to
  \[
  \partial^n(a) = \lambda^{\binom{n}{2}}au^i + (-1)^nu^na,
  \]
since $t^n=1$, as $t$ is an eigenvalue of the automorphism~$\phi$ and thus
a power of~$\omega$. As $\lambda$ is a primitive $n$th root of unity,
we have that $\lambda^{\binom{n}{2}}=(-1)^{n+1}$ and, therefore, that
  \[
  \partial^n(a) = (-1)^n[u^n,a].
  \]

This equality holds whenever $a$ is an eigenvector of~$\phi$ in~$A$ and,
since the automorphism~$\phi$ is diagonalizable, it then follows
immediately that the equality holds in fact for all~$a$ in~$A$. In
particular, the map~$\partial^n$ is zero exactly when the element $u^n$
of~$A_x$ commutes with all elements of~$A$: this happens if and only if
$u^n$ is central in~$A_x$, so an element of~$\kk$ according to
Proposition~\ref{prop:center:x}, and clearly this can happen if and only if
$u\in\kk$. Now, as $\lambda\neq1$ and $\phi(u)=\lambda u$, we certainly
have that $u$ is not in~$\kk$, so $\partial^n\neq0$, as the lemma claims.
\end{proof}

\begin{Remark}
The vanishing of the Gaussian binomial coefficients at well-chosen roots of
unity that we used in this proof is a very special case of the $q$-Lucas
theorem, and it was in fact in that way that we originally proceeded. The
simpler recourse to the $q$-binomial theorem that we used above was
suggested by a comment of Richard Stanley on MathOverflow~\cite{A}. Let
us note that the $q$-Lucas theorem was first proved by Gloria Olive
in~\cite{Olive}, where the theorem appears as Equation (1.2.4). This
beautiful result, which generalizes the classical Lucas theorem for
binomial coefficients, has also been proved by Jacques Désarménien
in~\cite{Desarmenien}, by Volker Strehl in~\cite{Strehl}, by Bruce Sagan
in~\cite{Sagan}, and Donald E.\ Knuth and Herbert S.\ Wilf describe
in~\cite{KnW} a factorization of the Gaussian binomial coefficient
$\binom{n}{i}_q$ with $0<i<n$  which certainly includes the $n$th
cyclotomic polynomial among the factors.
\end{Remark}

After all this work we can state and proof the following markedly
disappointing result:

\begin{Proposition}\label{prop:taft-actions}
Let $n$ and~$m$ be integers such that $1<m$ and $m\mid n$, and let
$\lambda\in\kk^\times$ be a primitive $m$th root of unity in~$\kk$. There
is no inner-faithful action of the generalized Taft algebra
$T_n(\lambda,m)$ on~$A$.
\end{Proposition}

\begin{proof}
Let us suppose, in order to reach a contradiction, that there is an
inner-faithful module-algebra action of the Hopf algebra $=T_n(\lambda, m)$
on~$A$, and let us consider the maps $\phi:a\in A\mapsto g\lact a\in A$ and
$\partial:a\in A\mapsto \xi\lact a\in A$. The relations that define the
algebra~$T$ imply that $\phi^n=\id_A$, $\partial^m=0$ and
$\phi\partial=\lambda\partial\phi$, and the fact that what we have is a
module-algebra structure that $\phi$ is an automorphism and $\partial$ a
$\phi$-twisted derivation. 

The criterion~\eqref{eq:faithful} for inner-faithfulness from Cline's paper
\cite{Cline} implies that the order of~$\phi$ is exactly~$n$, since $g$ has
order~$n$ in~$T$, and then, according to
Proposition~\ref{prop:finite-subgroups}, up to conjugating the action
of~$T$ on~$A$ by an algebra automorphism we can suppose that there is a
primitive $n$th root of unity $\omega$ in~$\kk$ such that $\phi(x)=\omega
x$ and $\phi(y)=\omega^{N-1}y$. Proposition~\ref{prop:der:xinner} now tells
us that there is an element~$u$ in the localization~$A_x$ such that the
$\tilde\phi$-twisted derivation $\ad_{\tilde\phi}(u):A_x\to A_x$ preserves
the subalgebra~$A$ of~$A_x$ and the map $\partial$ coincides with its
restriction $\ad_{\tilde\phi}(u)|_A$ to~$A$. The second part of
Lemma~\ref{lemma:partial:inner} tells us tat $\tilde\phi(u)=\lambda u$ and,
since $\partial^m=0$, the third part of that lemma that $n\leq m$. Of
course, since $m$ divides $n$ we have in fact that $n=m$, and we are
therefore in the situation of Lemma~\ref{lemma:non-zero}: this is
absurd, for the lemma tells us that $\partial^n\neq0$.
\end{proof}

In~\cite{Radford} Radford considers a more general class of generalized
Taft algebras: given two positive integers $n$,~$m\in\NN$ such that $1<m$
and $m\mid n$, a primitive $m$th root of unity~$\lambda$ in~$\kk$, and an
arbitrary scalar~$\tau\in\kk$, he lets $T_n(\lambda,m,\tau)$ be the algebra
freely generated by two letters~$g$ and~$\xi$ subject to the relations
  \[ 
  g^n = 1,
  \qquad
  \xi^m = \tau(g^m-1),
  \qquad
  g\xi = \lambda\xi g,
  \]
which is a Hopf algebra with respect to the comultiplication and counit such that
  \[
  \Delta(g) = g\otimes g,
  \qquad
  \Delta(\xi) = \xi\otimes1+g\otimes\xi,
  \qquad
  \epsilon(g) = 1,
  \qquad
  \epsilon(\xi) = 0.
  \]
When $\tau=0$ we obtain the algebras we considered above, since
$T_n(\lambda,m,0)=T_n(\lambda,m)$, and the connection between the two
families is that $T_n(\lambda,m,0)$ is the graded Hopf algebra
corresponding to the coradical filtration of the pointed Hopf
algebra~$T_n(\lambda,m,\tau)$ --- we refer to Section 5.2 in Montgomery's
book \cite{Montgomery} for information of this. It is a natural guess,
after our last proposition, that there is also no inner-faithful
module-algebra action of these more general Hopf algebras on our
algebra~$A$, although a new idea is needed to verify this. For example, the
twisted derivation~$\partial$ that corresponds to the element~$\xi$ in an action
of~$T_n(\lambda,m,\tau)$ is locally finite --- in the sense that every
element of~$A$ is contained in a finite-dimensional subspace that is
invariant under~$\partial$ --- so classifying locally-finite twisted
derivations of~$A$ could well be helpful.

\phantomsection
\addcontentsline{toc}{section}{References}%
\bibliography{simples}

\end{document}